%% file: main.tex
\documentclass[reqno]{amsart}
\usepackage{graphicx} 
\usepackage{mathpazo}
\usepackage{microtype}
\usepackage{tabularx}
\usepackage[utf8]{inputenc}
\usepackage{hyperref}
\usepackage{xcolor}
\usepackage{amsmath,amsthm,amssymb,mathtools,bbm}
\usepackage[foot]{amsaddr}
\usepackage{graphicx}
\usepackage{upgreek}
\usepackage[normalem]{ulem} 
\usepackage{mathrsfs}
\usepackage{ytableau} 
\usepackage{tikz}
\usepackage{braids}
\usetikzlibrary{arrows}
\usetikzlibrary{braids}
\usepackage{float}
\usepackage[all,2cell]{xy}
\UseAllTwocells 
\usepackage{blkarray}
\title{Paravortices: loop braid representations \\ with both generators involutive }

\author{Paul P. Martin$^1$}
\address{$^1$University of Leeds}

\author{Eric C. Rowell$^{1,3}$}
\address{$^3$Texas A\&M University}
\author{Fiona Torzewska$^{2,4}$}
\address{$^2$University of Bristol}
\address{$^4$Heilbronn Institute for Mathematical Research}

\input{macros}

\begin{document}

\maketitle

\begin{center}\dedicatory{Dedicated to Zhenghan Wang on the occasion of his 60th birthday.}\end{center}

\begin{abstract}
    We 
    first motivate the study 
    of 
    a certain quotient of the loop braid category, both for the mathematics underpinning recent approaches to topological quantum computation; and as a key example in non-semisimple higher representation theory. 
    For reasons that will become clear, we call this quotient the mixed doubles category, $\MD$. 
    Then our main result is a theorem 
    classifying all mixed doubles representations in rank-2. 
    
    Each representation yields a mixed doubles group representation for every loop braid group $\LB_n$, and we are able to analyse the 
    unified 
    linear representation theory 
    of many of these sequences of representations, using a mixture of very classical, classical, and new techniques. In particular this is a motivating example for the `glue' generalisation of charge-conserving representation theory (a form of rigid higher non-semisimplicity) introduced recently. 
\end{abstract}
\tableofcontents

\section{Introduction}
The study of local representations of the loop braid groups 
$L_n$ was initiated in \cite{KMRW}, partially motivated by the desire to explore topological phases of matter in 3 spatial dimensions, see also for example \cite{QiuWang2021,BullivantMartinsMartin2019}.  The motions of free loops (e.g., vortices) in 3 dimensional media are natural generalisations of the motions of points (e.g., anyons) in 2 dimensions, so $L_n$ plays the role of the braid group $B_n$ in the higher dimensional sequel to the topological quantum computation via anyons story, see \cite{FLKW,Nayaketal,RowellWangBull} and references therein.   Indeed, $L_n$ resembles $B_n$ strongly: the motions of passing a loop through its neighbour (`leapfrogging') form a subgroup isomorphic to $B_n$, while interchanging adjacent loops have order 2 and yield a copy of the symmetric group $\Sym_n$.  

By a \emph{local} representation we mean a strict monoidal functor $F:\LB\rightarrow\Mat$ from the loop braid category $\LB$ to the category $\Mat$ of matrices   
(see e.g. 
\cite{MRT25} for 
a review of the setup in the setting of 
the braid category $\Bcat$  \cite{MacLane}). 
The  morphisms from object $n$ to itself in $\LB$ 
form
the loop braid group $L_n$, just as the morphisms in $\Bcat$ are braid groups $B_n$. In practice, one specifies such a functor $F$ by means of two Yang-Baxter operators $R,S$ satisfying certain compatibility relations and $S^2=I$. There are numerous reasons to focus on these local representations, not least of which is the connection to the quantum circuit model for quantum computation, in which circuits are cascades of matrices acting on a few adjacent factors of $\bigotimes_i\C^n$ \cite{NielsenChuang2000}. 
For existing examples in the literature of local loop braid representations in physical settings see also for example \cite{BullivantMartinsMartin2019,DMMx} and references therein. 

In this article we study monoidal functors from the mixed doubles category, which is obtained by
\ppmm{setting also $R^2=1$} and is 
denoted $\MD$.  The  {morphism groups} in this category are 
thus 
quotients $\MD_n$ of $L_n$ by the relation that the leapfrogging generators also have order $2$.  Thus $\MD_n$ 
\ppmm{has}
two copies of the symmetric group $\Sym_n$ on equal footing, satisfying some mixed relations: hence the `mixed doubles' moniker.

There are a number of compelling reasons to study local representations of $\MD_n$.

A physical motivation is that, philosophically, these are a natural 3-dimensional analogue of \emph{paraparticles}  \cite{WangHazzard2024Supplement} (a generalisation of parafermions, see eg. \cite{AliceaFendley2016})  where the exchange statistics of $n$ identical paraparticles yield representations of the symmetric group $\Sym_n$.  From this point of view, representations of $\MD_n$ correspond to \emph{paravortices}: loop-like excitations for which the loop leapfrog operation on adjacent loops squares to the identity.  Thus $\MD_n$ fits into the topological quantum computation story as the group governing the loop motion statistics of paravortices.  Representations of $\MD_n$ can then be studied as  quantum gates implemented on systems of paravortices, which could provide fault-tolerant models for quantum computation.  We will see that universality is unlikely, but they are a compelling potential quantum resource, nonetheless (cf. \cite{RowellWang12}).

Mathematically, 
$\MD_n$ is a non-trivial quotient of $L_n$, analogous to $\Sym_n$ as the first non-trivial quotient of $B_n$. Thus $\MD_n$ is a Coxeter-style combinatorial version of $L_n$.  Unlike $\Sym_n$, $\MD_n$ is an infinite group, so it retains some of the intrigue of $L_n$.  
Moreover, the mixing of the two components - the two copies of $\Sym_n$ - has the effect of making  
basis choices matter, 
which can be cast as 
a type of higher (or extended) representation theory
problem.  An intriguing but relatively tame presentation of $L_n$ as a semidirect product is available \cite{BardakovBellingeriDamiani15} which can be leveraged for representation theoretic computations, 
as we shall discuss.

Next, although perhaps the main problem in this area is the classification of braid representations, the part of it corresponding to the classification of symmetric representations (strict monoidal functors $F:\Sym\rightarrow\Mat$) is also a current focus. And of course $\MD$ representations can be seen as a generalisation of this. 
In the symmetric rep context, 
there is a classification of unitary, involutive solutions up to a representation-theoretic notion of equivalence,  
\ppmm{due to} 
\cite{LechnerPennigWood}.  Thus {the representation theory of $\MD$} potentially interpolates between the full braid representation 
problem 
- classifying strict monoidal functors $F:\Bcat\rightarrow \Mat\;$ - 
and partial solutions such as the \cite{LechnerPennigWood}
classification of symmetric representations/involutive braid representations $F:\Sym\rightarrow\Mat$.

We also note that we can make contact with the \ppmm{growing} community of researchers in set-theoretical solutions to the Yang-Baxter equation, where a main focus is on involutive solutions (see, e.g., \cite{gatevaivanova2009multipermutation}).  A natural subproblem of the classification of local $\MD_n$ representations would be to ask for pairs of set-theoretical solutions $(R,S)$ that yield such a representation.
Contact between this area and the classification programme for braid representations is already discussed in \cite{MRT25}.

\smallskip

Our main result, Theorem~\ref{thm: big theorem}, is a classification, up to natural equivalences, of all strict monoidal functors $F:\MD\rightarrow\Mat$ 
in rank 2, i.e. 
with $F(1)=2$. 
That is, 
we classify 
all local representations with $R,S$ being $4\times 4$.  
In \S\ref{ss:analyse}
we develop and deploy some machinery to enable us to 
analyse several of these functors as representations (i.e. as collections of ordinary representations taken together), 
providing, in some cases, all-$n$ descriptions of their irreducible content and their decompositions into indecomposable subrepresentations. 
In other notable cases we simply analyse representations in low rank by more direct methods. 
\ignore{{
\ecr{[kill from here ....]}

\section{no longer the Introduction, more like Preliminaries}
The mixed doubles category studied here, defined just below, and denoted $\MD$, 
is interesting for a number of reasons: 
\\ $\bullet$
It is a natural and useful quotient of the  loop braid category $\Lcat$. 
It is much simpler, but it has one of the key features of loop braid: 
the `mixing' of a symmetric group with a related group (in this case another 
symmetric group - hence {\em mixed doubles} category), 
where the mixing can be viewed as having the effect of making bases matter in representation theory - a type of higher (or extended) representation theory. 
\\ $\bullet$
It can be viewed as a generalisation of the  category $\Sym$ of symmetric groups. 
\\ $\bullet$
Its representation theory potentially interpolates between the full braid rep 
problem 
- classifying strict monoidal functors $F:\Bcat\rightarrow \Mat\;$ - 
and partial solutions such as Lechner--Pennig--Wood's 
classification of sym reps/involutive braid reps $F:\Sym\rightarrow\Mat$  \cite{LechnerPennigWood}. 
\\ $\bullet$ 
Each end-group is isomorphic to a group with a relatively very tame presentation, due to Bardakov--Bellingeri--Damiani \cite{BardakovBellingeriDamiani15} ... ...

\ecr{[...to here, once all points have been 'moved up'.]}
\medskip }}

\vspace{.53cm} 

\noindent {\bf Acknowledgements}.
PM thanks EPSRC for funding under project EP/W007509/1; 
and thanks the Royal Society for funding ER's Wolfson Fellowship at Leeds 
contributing to this work. PM also thanks 
Jarmo Hietarinta, 
Sarah Almateari 
and Paula Martin for useful discussions. 
The work of ER was also partially supported by US NSF grant DMS-2205962.  Part of this work was carried out while ER was visiting the University of Leeds on a Royal Society Wolfson Fellowship, and he thanks them for their hospitality and support.  And all 3 authors thank Zhenghan Wang for continued inspiration and guidance.

\section{Preliminaries}  \label{ss:prelims}
\newcommand{\RRR}{{\mathfrak{R}}}

In this paper we will be concerned with representation theory. However it is 
conceptually (and motivationally) appropriate, and illuminating, to start with a geometric-topological context for the groups under consideration. 

For $n \in \N$
the loop braid group $\LB_n$ is a motion group of a configuration of $n$ unlinked unknotted oriented loops in 3-dimensions (see e.g. \cite{Damiani} for a review and original references). 
For definiteness, for the reference configuration, $C_n$, we take circles of unit radius lying in the xy-plane, with centres regularly spaced along the x-axis. 
Hence in particular for each loop $l$ in the reference configuration there is a disk of which $l$ is the boundary. 
For example in $\LB_2$ 
we can consider a certain pair of  
motions. Both involve the exchange of the two loops. 
Let us consider both motions in a frame in which the first loop appears static - so that the second loop appears to pass from a place to the right to a place to the left of it. 
(N.B., in this frame the  {\em constant} representative of the identity motion has neither loop apparently static to the observer.)
In the first motion,  
$\sigma$, say, 
there is a representative concrete motion in which the moving loop completes its journey well away from the static loop. 
In the second motion, $\rho$,  
there is a representative in which the moving loop 
passes once through the disc of the static loop. It will be clear that these representatives give distinct motions; indeed that the first is involutive and that the second is not. 

In fact the two motions $\sigma$ and $\rho$  
generate $\LB_2$, and,  
fixing $n$,
a corresponding chain of pair exchange elements $\sigma_i$ and $\rho_i$  generates $\LB_n$. 
Thus $\sigma_i^2 =1$ and $\sigma_i \sigma_j = \sigma_j \sigma_i $ for $j\neq i\pm 1$, and so on.

\mdef   \label{de:loopBcat}
The category $\LB$ is the strict monoidal category with object class $\N_0$, with $\LB(n,n)=\LB_n$; 
$\; \LB(n,m)=\emptyset$ if $n \neq m$; and monoidal product given by juxtaposition of the reference configurations along the x-axis 
(see for example \cite{MartinRowellTorzewska} and references therein, 
and cf. the braid category in \cite{MacLane}).  

Observe that category  $\LB$ is generated as a monoidal category by the object 1 and the morphisms 
$\sigma$ and $\rho \in \LB(2,2)$. In particular , writing $ 1_1$ for the identity in $\LB(1,1)$ we have 
(fixing $n=3$ for a moment) 
$\sigma_2 = 1_1 \otimes \sigma \in \LB(3,3)$, and so on.


\mdef 
Consider 
 the abstract group $L_n$ with presentation 
\[
L_n \;\; = \;\; \langle\;  s_1, s_2,\ldots, s_{n-1}, r_1, r_2, \ldots, r_{n-1} \; | \; 
\RRR_- \;\rangle
\]
where the relations $\RRR_-$ are:
\begin{equation} \label{eq:sigmass}
 r_i r_{i+1} r_i = r_{i+1} r_i r_{i+1} 
, \hspace{1cm}
 \; r_i r_{i+1} s_i = s_{i+1} r_i r_{i+1}
, \hspace{1cm} 
 \; r_i s_{i+1} s_i = s_{i+1} s_i r_{i+1}  
\end{equation}
together with the symmetric group relations for the $s_i$:
\beq  \label{eq:sss-ss1}
s_i s_{i+1} s_i = s_{i+1} s_i s_{i+1}, \;  \; \;\;   \hspace{1cm} 
s_i^2 =1,
\eq 
and far commutation relations:
\beq\label{eq:farcommute}
r_ir_j=r_jr_i, \;\; s_is_j=s_js_i,\;\; r_is_j=s_jr_i,\;\;\;\;\; |i-j|\neq 1 
\eq 

\mdef   \label{de:Liso}
For each $n$ 
the group $\LB_n$ is isomorphic to the abstract group $L_n$ under the map on generators given by 
$\sigma_i \mapsto s_i$ and $\rho_i \mapsto r_i$.

\medskip 

\mdef   \label{de:MDgroup}
For $n \in \N$
the mixed doubles group $\MD_n$ is defined  
by the presentation    
\beq   \label{eq:presentMD}
\MD_n \;\;= \;\; \langle\;  s_1, s_2,\ldots, s_{n-1}, r_1, r_2, \ldots, r_{n-1} \; | \; 
\RRR \;\rangle
\eq
where the relations $\RRR$ are the relations $\RRR_-$ together with:
\beq  \label{eq:rr1}
r_i^2 =1,
\eq 
\ignore{{ 
\beq
s_i s_{i+1} s_i = s_{i+1} s_i s_{i+1}, \;  \; \;\;
r_i r_{i+1} r_i \; = \; r_{i+1} r_i r_{i+1}, \; \;    
\eq 
\beq 
r_{i} s_{i+1} s_i = s_{i+1} s_i r_{i+1}, \;\;\;\;\hspace{.1cm}
s_i r_{i+1} r_i \; = \; r_{i+1} r_i s_{i+1}
\eq 

\beq r_ir_j=r_jr_i, \;\; s_is_j=s_js_i,\;\; r_is_j=s_jr_i,\;\; |i-j|\neq 1
\eq
}}

\medskip


The groups $\MD_n$ rank-by-rank have various other names in the literature
depending on the background of the authors. See e.g. \cite{BardakovBellingeriDamiani15}, 
where $\MD_n$ is called the group of flat welded braids, and references therein.

\medskip

Here we collect some facts about the groups $\MD_n$, that will be useful later. 


\mdef {\bf Proposition}.   \label{pr:MDcut}
For any $i<n$, if we omit the generators $r_i, s_i$ then the subgroup $\MD^{(i)}_n$ of $\MD_n$ generated by the remaining subset obeys
\[
\MD^{(i)}_n \; \cong \; \MD_i \times \MD_{n-i}
\]
{\em Proof}.
 The generators  $r_j,s_j$ with $1\leq j<i$ and $r_k,s_k$ with $i<k\leq n-1$ commute in $\MD_n$, so the natural homomorphism $\MD_i\times \MD_{n-i}\rightarrow \MD_n$ is injective and has image $\MD_n^{(i)}$.  \qed 

\mdef {\bf Proposition}.    \label{pr:MDinc}
For $i \in \{ 0,1,\ldots,m\}$ there is an inclusion  
$$
\phi_i :\MD_n \hookrightarrow \MD_{m+n}
$$ 
given by $r_j \mapsto r_{j+i}$ and $s_j \mapsto s_{j+i}$. 
\\
{\em Proof}. 
 {Note that $j$ here is in $\{1,\ldots, n-1\}$, thus the image is in $\MD_{m+n}$. It is clear that the image under $\phi_i$ of all relations in $\MD_{n}$ are also satisfied in $\MD_{m+n}$.} \qed

\medskip 

\mdef  \label{de:MDcat}
The collection of all the groups $\MD_n$ naturally arranged in a strict monoidal category  
is the mixed doubles category $\MD$. 

\medskip 

A monoidal functor $F:\MD \rightarrow \Mat$ 
is here called a {\em mixed doubles representation} 
(following the nomenclature of spin-chain braid representation, shortened to braid representation, in \cite{MR1X}). 

\medskip

The group 
$\MD_n$ 
is shown in  \cite{BardakovBellingeriDamiani15} to be isomorphic to 
\ppmm{a certain semidirect product}
$\Z^{\binom{n}{2}}\rtimes \Sigma_n$ (see \S\ref{ss:babeda} for details).  
In particular $\MD_n$ is virtually abelian.  Classical Clifford theory \cite{isaacs1994} can therefore be used 
in principle 
to describe the irreducible representations of $\MD_n$. 
In practice this approach is not straightforward. 

Our interest here is 
in mixed doubles representations, as above.  
These are examples of 
\emph{local} representations, i.e. representations associated with a pair of Yang-Baxter operators $(R,S)$ yielding a strict monoidal functor from the category $\MD$ build from the groups $\MD_n$ to the category $\Mat$ of matrices.   
In particular, such representations can be non-semisimple, pushing us into the world of indecomposable, rather than irreducible, constituent representations.  Even for $\Z^{\binom{n}{2}}$ the classification problem for indecomposable representations is wild, so local representations are not easily decomposed.

\mdef 
Here `local 
representation' means in essence firstly that we consider 
representations $ \varrho_n$ for 
a natural-number-indexed sequence of presented groups, and that in each presentation every generator is of the form $\alpha_i$ where $\alpha$ is taken from the same finite set, $i$ ranges over natural numbers, 
and every relation involving \ppmm{two} non-adjacent $i$'s is a commutation. 
Secondly the representation $\varrho_n$ should also include in $\varrho_{n+1}$.

\mdef 
 {It is clear from the relations  $\RRR$ that 
{for $n>1$} 
the automorphism group of $\MD_n$ has a subgroup isomorphic to $\Z/2\Z\times \Z/2\Z$ generated by $s_i\leftrightarrow r_i$ and the simultaneous $(s_i,r_i)\leftrightarrow (s_{n-i},r_{n-i})$ for all $i$.}

\section{Generalised Wreaths and the BaBeDa isomorphism}   \label{ss:babeda}

In this section we define some generalised wreath products, and for one class 
establish an isomorphism with $\MD_n$ ($n \in \N$). 

Both wreaths and our generalisations are examples of semidirect products. 
Let $B,C$ be groups, and $\psi:C\rightarrow Aut(B)$ {$,w\mapsto \psi_w$} be a group homomorphism. Then there is a group structure on $B \times C$ given by 
\beq  \label{eq:defsd0}
(X_1, w_1) \; (X_2, w_2) \;\; = \;\;  (X_1 \psi_{w_1} (X_2), \; w_1 w_2 ) .
\eq 
The group can be denoted $B\rtimes_\psi C$. 

\mdef  \label{pa:conj-auto}
Observe that the subset $\{ (b,1_C) : b \in B \}$ of  $B\rtimes_\psi C$ is a normal subgroup isomorphic to $B$;
and  $\{ (1_B,c) : c \in C \}$ is a subgroup isomorphic to $C$.
We may write $b \in  B\rtimes_\psi C $ for $(b,1_C)$ and so on. 
Thus for example {\em in}   
$B\rtimes_\psi C $ we have  
$
c^{-1} b c = \psi_{c^{-1}}(b)  .
$ 

\ignore{{We will not need the following until \S\ref{ss:linRT}, but it is convenient to establish now. 

\mdef  \label{de:CactonBreps}
Since prepending an automorphism takes a $B$-rep to a $B$-rep, we have an action of $C$ on the set of $B$-representations. In particular it takes irreducibles to irreducibles. But in general let us write $\rho^c$ for the image of $B$-rep $\rho$ under the action of $c \in C$. 
}}

\input{tex/wreaths01}

\subsection{The BaBeDa isomorphism} $\;$ 

\mdef   \label{iso}
Consider 
the map $\varphi$ on the generators of $\MDD_n$  from (\ref{pres:MD'})
(and hence on the free group with these formal generators)
to $\MD_n$ 
given by $\varsigma_i \mapsto s_i$ and $x_{12} \mapsto r_1 s_1$.
It is shown in \cite[Prop.5.5]{BardakovBellingeriDamiani15} that $\varphi$ 
extends to an isomorphism 
$\MDD_n \cong \MD_n$. 

We call $\varphi$ the BaBeDa isomorphism (a portmanteau of the authors' surnames). We also give the proof here to give a self-contained explanation of how the copy of $\Z^{{n}\choose{2}}$ arises. 
For this we need some preliminary results given next.

\ignore{{
\ft{I am not sure what the following is saying...}
Explicitly, we have $\varphi(x_{i,i+1})= r_is_i$: indeed in $\MDD_n$ we have
\[
\varsigma_i\varsigma_{i+1}x_{i,i+1}\varsigma_{i+1}\varsigma_i=x_{i+1,i+2},\quad i<n-2
\]
by (\ref{de:MDprime}),
so under $\varphi$ this becomes, using the relations in $\MD_n$:
\[
\varphi(x_{i+1,i+2})=s_is_{i+1}r_is_is_{i+1}s_i
\stackrel{(\ref{eq:sigmass})}{=}
(r_{i+1}s_is_{i+1})(s_{i+1}s_is_{i+1})
\stackrel{(\ref{eq:sss-ss1})}{=}
r_{i+1}s_{i+1},
\]
so this result follows by induction.}}

\mdef\label{PBn} 
The `virtual braid group' $\VB_n$ is obtained by taking the same formal generators as $L_n$, and taking as relations the leftmost and rightmost relation in \eqref{eq:sigmass}, as well as all relations in \eqref{eq:sss-ss1} and \eqref{eq:farcommute}.  

There is a map $\pi\colon \VB_n\to \Sigma_n$ given by letting both $r_i$ and $s_i$ map to the element $(i,i+1)$. 
This is clearly surjective, since the elementary transpositions $(i,i+1)$ generate $\Sigma_n$. 
Thus defining $\VP_n$ to be the kernel of this map $\pi$ gives that $\VB_n$ breaks down as a semi direct product $\VP_n \rtimes \Sigma_n$.
It is proved in \cite[Th.1]{Bardakov} that $\VP_n$ admits a presentation with generators $\lambda_{kl}$, $1\leq k\neq l\leq n$, and relations
\begin{align}\label{eq:PCommute}
\lambda_{ij}\lambda_{kl}=\lambda_{kl}\lambda_{ij}
\end{align}
and
\begin{align}\label{eq:PQCommute}
\lambda_{ki}\lambda_{kj}\lambda_{ij}=\lambda_{ij}\lambda_{kj}\lambda_{ki}
\end{align}
where distinct letters represent distinct indices.
The group $\VP_n$ is realised as a subgroup of $\VB_n$ by letting
\begin{gather*}
    \lambda_{i,i+1} = s_ir_i^{-1}, \;\; \;\lambda_{i+1,i} = r_i^{-1}s_i,\;\; \; i=1,2,\ldots , n-1\\
    \lambda_{ij} = s_{j-1}s_{j-2}\ldots s_{i+1} \;\lambda_{i,i+1}\;s_{i+1}\ldots s_{j-2}s_{j-1}, \\
    \lambda_{ji} = s_{j-1}s_{j-2} \ldots s_{i+1} \;\lambda_{i+1,i}\;s_{i+1}\ldots s_{j-2}s_{j-1}, \; 1\leq i< j-1 \leq n-1.
\end{gather*}
Observe for example that the first line forces $\pi(\lambda_{i,i+1})=1$ as required.

\mdef \label{PConj}
It is also shown in \cite[Le.1]{Bardakov} 
that 
the subgroup $\VP_n\subset \VB_n$ satisfies the following conjugation relation
\[
\iota(s)\lambda_{ij} \iota(s)^{-1} = \lambda_{s(i)s(j)}
\]
where $\iota$ is the natural section of $\pi\colon \VB_n\to \Sigma_n$.

\begin{proof}(of \eqref{iso})    
Consider the mapping $\hat{\varphi}$ from generators of $\MD_n$ to generators of $\MDD_n$ given by 
$s_i\mapsto \varsigma_i$, $\; r_i \mapsto \varsigma_i x_{i,i+1} $.
We will prove that both $\varphi$ and $\hat{\varphi}$ lift to well defined homomorphisms $\varphi\colon \MDD_n\to \MD_n$, and $\hat{\varphi}\colon \MD_n\to \MDD_n$. It then follows, by confirming that the mappings on generators are indeed pairwise inverse, that the two groups are isomorphic.  

First note 
that all relations involving only the symmetric group generators in each group 
(meaning the $s_i$s in $\MD_n$)
are clearly preserved by both $\varphi$ and $\hat{\varphi}$. 

Observe that for $j\geq i$, \[\hat{\varphi}(\lambda_{ij})= \varsigma_{j-1}\varsigma_{j-2}\ldots \varsigma_{i+1} \;\lambda_{i,i+1}\;\varsigma_{i+1}\ldots \varsigma_{j-2}\varsigma_{j-1} =x_{ij}\] using \eqref{eq:Conj}, and similarly for $i\geq j$, $\hat{\varphi}(\lambda_{ij})= x_{ij}$.
Since relations \eqref{eq:PCommute} and \eqref{eq:PQCommute} are also satisfied by the $x_{ij}$, and $\hat{\varphi}(\lambda_{ij})=x_{ij}$, it follows from \eqref{PBn} that $\hat{\varphi}(r_ir_{i+1}r_i)=\hat{\varphi}(r_{i+1}r_{i}r_{i+1})$, and that $\hat{\varphi}(s_is_{i+1}r_i)=\hat{\varphi}(s_{i+1}s_ir_{i+1})$. 
Clearly $\hat{\varphi}(r_i^2)=\varsigma_i x_{i,i+1}\varsigma_i x_{i,i+1} = x_{i,i+1}^{-1}x_{i,i+1}=1$. 
Looking at the image of the middle relation in \eqref{eq:sigmass}
we have
\begin{align}\label{eq:Immixedrel}\hat{\varphi}(s_ir_{i+1}r_is_{i+1}r_{i}r_{i+1}) &= 
\varsigma_i\; \varsigma_{i+1} \;x_{i+1,i+2}\; \varsigma_i\; x_{i,i+1} \;\varsigma_{i+1} \;\varsigma_{i} \;x_{i,i+1}  \varsigma_{i+1}\;x_{i+1,i+2}\nonumber\\
&= \varsigma_i \;\varsigma_{i+1} \;x_{i+1,i+2}\; \varsigma_i\; x_{i,i+1}\; \varsigma_{i}\;\varsigma_{i}\;\varsigma_{i+1}\; \varsigma_{i}\; x_{i,i+1}  \;\varsigma_{i+1}\;x_{i+1,i+2} \nonumber \\
&=\varsigma_i \;\varsigma_{i+1} \;x_{i+1,i+2}\; \varsigma_i\; x_{i,i+1}\; \varsigma_{i}\;\varsigma_{i+1}\;\varsigma_{i}\; \varsigma_{i+1}\; x_{i,i+1}  \;\varsigma_{i+1}\;x_{i+1,i+2}\nonumber\\
&=\varsigma_i \;\varsigma_{i+1} \;x_{i+1,i+2}\; x^{-1}_{i,i+1}\;\varsigma_{i+1}\;\varsigma_{i}\;  x_{i,i+2} \;x_{i+1,i+2}
\nonumber\\
&= x^{-1}_{i,i+2}\;x^{-1}_{i+1,i+2}\;  x_{i,i+2} \;x_{i+1,i+2} = 1.
\end{align}
Thus $\hat{\varphi}$ is well defined.

We now show that $\varphi$ is well defined. That the conjugation relation \eqref{eq:Conj} in $\MDD_n$ is satisfied in $\MD_n$ follows from \eqref{PConj} by using again that the $x_{ij}$ map to $\lambda_{ij}$. It remains to check that $\varphi(x_{ij}x_{kl})=\varphi(x_{kl}x_{ij})$ for all $1\leq i <j \leq n$ and $1\leq k<l \leq n$. If $i=k$ and $j=l$ this is obvious, and if all indices are distinct, this follows from \eqref{eq:Commute}. We now address the remaining cases.
We have proved in \eqref{eq:Immixedrel} that $\varphi(x_{i,i+2}x_{i+1,i+2})=\varphi(x_{i+1,i+2}x_{i,i+2})$, we now prove it follows that 
\begin{align}\label{eq:jjCommute}
    \varphi(x_{i,j}x_{k,j})=\varphi(x_{k,j}x_{i,j})
\end{align} 
whenever $k\neq i$. This can be seen by observing that, using the conjugation rule \eqref{PConj} and that $\varphi(x_{ij})=\lambda_{ij}$,
\begin{align*}
    \varphi(x_{i,j}x_{k,j}) &= s_{j,i+2} \;\varphi(x_{i,i+2}) \;s_{j,i+2} \;s_{j,i+2}\; s_{k,i+1} \; \varphi(x_{i+1,i+2})\;s_{k,i+1}\;s_{j,i+2}\\
    & = s_{j,i+2} \;s_{k,i+1}\;\varphi(x_{i,i+2})  \; \varphi(x_{i+1,i+2})\;s_{k,i+1}\;s_{j,i+2}\\
& = s_{j,i+2} \;s_{k,i+1}  \; \varphi(x_{i+1,i+2}) \;\varphi(x_{i,i+2})\;s_{k,i+1}\;s_{j,i+2}\\
& = s_{j,i+2}   \;s_{k,i+1} \varphi(x_{i+1,i+2}) \;s_{k,i+1}\;\varphi(x_{i,i+2})\;s_{j,i+2}\\
& = s_{j,i+2}   \;\varphi(x_{k,i+2}) \;s_{j,i+2}\;\varphi(x_{i,j})\\
& =   \;\varphi(x_{k,j})\;\varphi(x_{i,j}) = \varphi(x_{k,j}x_{i,j}). \\
\end{align*}
Similarly, looking at the image of $s_{i+1}r_{i}r_{i+1}s_ir_{i+1}r_i$, which is obtained from the middle identity in \eqref{eq:sigmass} by using \eqref{eq:rr1}, gives that $\varphi(x_{i,i+1}x_{i,i+2})=\varphi(x_{i,i+2}x_{i,i+1})$
which in turn gives that 
\begin{align}\label{eq:iiCommute}
    \varphi(x_{i,j}x_{i,k})=\varphi(x_{i,j}x_{i,k})
\end{align}
 whenever $k\neq j$.
The only remaining case is $\varphi(x_{i,j}x_{k,i})= \varphi(x_{k,i}x_{i,j})$.
This follows by using \eqref{eq:iiCommute} and \eqref{eq:jjCommute} to rewrite  \eqref{eq:PQCommute} as
\[
\varphi(x_{k,j}x_{k,i}x_{i,j}) = \varphi(x_{k,j}x_{i,j}x_{k,i})
\]
and cancelling on the left to obtain
\[
\varphi(x_{k,i}x_{i,j}) = \varphi(x_{i,j}x_{k,i}).
\]
This gives that $\varphi$ is well defined.
\end{proof}

\medskip 


\section{Mixed doubles representations}

A mixed doubles ($\MD$) representation is a strict monoidal functor 
$F:\MD \rightarrow \Mat$.

\mdef Recall that the braid category $\Bcat$ is the 
strict monoidal category (SMC) of braid groups $B_n$. 
Here we will by default write $r \in \Bcat(2,2)$ for the elementary braid generator. 

\mdef  \label{de:MD}
The category $\MD$ is defined as follows. It is the strict monoidal category of mixed doubles groups $\MD_n$, analogous to the braid category or the Sym category. 
The object monoid is $(\N_0, +)$, thus it is freely generated by the object 1. 
The monoidal product on morphisms is given by the inclusion $\phi: \MD_m \times \MD_n \rightarrow \MD_{m+n}$ given by lateral juxtaposition - that is 
$a \otimes b = \phi(a,b) = \phi_0(a) \phi_m(b)$ as defined in (\ref{pr:MDinc}). 

\mdef \label{pr:MDgen} {\bf Proposition}. 
(I) Morphisms in the category $\MD$ are monoidally generated by two  
morphisms $s,r \in \MD(2,2)$. 
They obey the relations 
\beq  \label{eq:ss1}  
s^2 = r^2 =1
\eq 
and we have further, defining $s_1 = s \otimes 1$, $s_2 = 1\otimes s$
and so on
\beq 
s_1 s_2 s_1 = s_2 s_1 s_2 ,   \hspace{1cm}
r_1 r_2 r_1 = r_2 r_1 r_2
\eq 
and
\beq 
r_1 s_2 s_1 = s_2 s_1 r_2 ,   \hspace{1cm}
s_1 r_2 r_1 = r_2 r_1 s_2  
\eq 
(II) These generators and relations give a presentation.

\medskip

\noindent {\em Proof}. 
(I) Note from (\ref{eq:presentMD}) that each group $\MD(n,n)$ is generated by the corresponding $s_i$'s and $r_i$'s. 
Their images at fixed $n$ are 
$1 \otimes  \cdots \otimes1 \otimes s \otimes 1 \otimes \cdots\otimes 1$ 
($s$ in the $i$-th position) 
and similarly for $r_i$.
\\
(II){For a fixed $n$, the relations in the group presentation for $\MD(n,n)$ given in \eqref{eq:sigmass} - \eqref{eq:sss-ss1} are obtained from the given relations by taking the monoidal product with an appropriate number of identities on the left and the right. The far commutation relations given in \eqref{eq:farcommute} arise from the interchange law for monoidal categories: for morphisms $f,g$ we have $f\otimes g=(1\otimes f)(g\otimes 1) =(g\otimes 1) (1\otimes f)$. } \qed 

\mdef It follows from (\ref{pr:MDgen}) 
above that a strict monoidal functor $F:\MD \rightarrow \Mat$ is completely determined by the images 
$ R = F(r)$ and $S = F(s)$. 
Furthermore, a useful partial classification of such functors is by the value of $F(1) \in \N$, called the {\em rank} of $F$.

If $F(1) = N$, say, then $F$ is also a functor $F: \MD \rightarrow \Mat^N$ 
(the full monoidal subcategory of $\Mat$ generated by the object $N$ - which we re-declare as the object 1 in $\Mat^N$, so that this category has object monoid $(\N_0, +)$ on the nose). 

Note that such a functor $F$ is therefore also completely determined by $F(r)$ and $F(x =rs)$. 

\mdef   \label{de:Wangian} 
Observe that if $F': \Bcat \rightarrow \Mat$ is an involutive
braid representation
then $F(r) = F'(r)$ and $F(s) = F'(r)$ gives a 
loop braid
representation. 
A
loop braid
representation of this type, i.e. with $F(r)=F(s)$, is called {\em Wangian}, after Zhenghan Wang who pointed this out during the writing of \cite{KMRW}. 
In particular any Wangian representation is a representation of $\MD$. 
Note that this is then also a symmetric/$\Sym$ representation. 
Hence:

\mdef {\bf{Proposition}}.   \label{pr:Wangian}
Wangian $\MD$ representations yield $\MD_n$ representations for all $n$ that are always semisimple over the field $\C$. 
\qed

\subsection{Notions of equivalence of representations/strict monoidal functors} $\;$ 

Recall that we work by default over the complex field. 

Classification of strict monoidal functors 
$F: \catC \rightarrow \Mat \;$
from a natural category (a strict monoidal category with object monoid the free monoid on one generator, such as $\Bcat$ or $\MD$) 
is a branch of 
higher representation theory, and as such there is not a canonical notion of equivalence, in contrast to ordinary representation theory of an algebra. 
Instead we have various partial equivalence relations, each of which is more or less
useful depending on the situation. 

In this section we discuss notions of equivalence. The section has a `bootstrap problem' in that the best way to understand the various notions is through examples,
but most of our key examples will arrive later. 
For this reason it may be best to treat this as a reference section, and dip into it only when a definition is needed.

\mdef 
The most 
basic classification is according to the `rank', the value of $F(1)$. 
Fixing the rank $N$, we then consider the target to be $\Mat^N$, the full subcategory
of $\Mat$ generated by its object $N$ (we rename $N$ as 1 in $\Mat^N$, so that 
$F:\catC\rightarrow \Mat^N$ has $F(1)=1$, so note that the term `rank' applies to $F(1)$ for the target $\Mat$ specifically). 

\mdef  \label{de:local}
Fixing the rank $N$, local equivalence means that $F\sim F'$ 
whenever there is an invertible matrix $A \in \Mat^N(1,1) = \Mat(N,N)$ such that 
for $a \in \catC(m,n)$ then
$F'(a) = A^{\otimes m} F(a) (A^{-1})^{\otimes n}  $.  
\\ 
This is precisely a monoidal natural equivalence between the functors $F,F'$, in the sense of \cite{MRT25}.

For example an R-matrix $F(\sigma)$ is conjugated by $A \otimes A$.

\mdef \label{de:l-equiv}
For $l \in \N$ then $l$-equivalence means that there are invertible matrices 
$A_n \in \Mat^N(n,n)$ for $n\leq l$ such that  
$ F'(a) = A_m F(a) A_n^{-1}$ 
for $a \in \catC(m,n)$ with $m,n\leq l$.   

For example  
$l$-equivalence of braid representations says that  the representations of $B_n$ are isomorphic for every $n \leq l$.

(I) Note that local equivalence implies $\infty$-equivalence. 

(II) Note that 2-equivalence does not imply $\infty$-equivalence \cite{MRT25}.

\mdef \label{de:l-ch-equiv}
For $l \in \N$ then {\em  $l$-character-equivalence} 
means that 
the character of $F(\catC_n = \catC(n,n))$ is the same as for $F'$ for all 
$n\leq l$. 

(I) Note that $l$-equivalence implies $l$-character-equivalence. 

(II) Note that if each $\catC_n$ is a finite group then $l$-character-equivalence implies $l$-equivalence (since characters determine simple content and the group algebra is semisimple over $\C$).

(III) Continuing with the case of each $\catC_n$  a finite group, 
if there are a variety of $R$-matrices obtained by variation of a continuous parameter, then they are in the same $\infty$-equivalence class. 

Note that this is not true if each $\catC_n$ is not finite, such as in the $\Bcat$ case or $\MD$ case (although there are related cardinality-sensing differences between these cases, as we shall see).

\mdef  There are important (for efficient classification) weaker notions of equivalence (which in general do not imply even 2-equivalence), which might better be termed symmetries rather than equivalences. 
The point is that if there is a simple and manifest transformation that turns 
a solution into another solution then we need only classify one per orbit of this transformation. For more detail than we give here, see for example \cite[\S2.2 {\em et seq}]{MRT25}. Otherwise, read on.

\mdef\label{DS equiv for MD} As an example of such a symmetry we may extend \emph{Doiku-Smoktunowitz} (DS) equivalence to $\MD$ in the obvious way: let $A\in\Mat^N(1,1)$ be such that $A\otimes A$ commutes with $R=F(r)$ and $S=F(s)$.  Then $R'=(A\otimes I_N)R(A^{-1}\otimes I_N)$ and $S'=(A\otimes I_N)S(A^{-1}\otimes I_N)$ defines another functor $F'$.  In fact, the proof given in \cite{MRT25,DS} for a single $R$-matrix goes through \emph{mutatis mutandis}.  It 
follows that the corresponding representations of $\MD_n$ are equivalent, i.e., $\infty$-equivalent. 

\mdef  
Since $\Mat$ is $\C$-linear, it is convenient to pass to the $\C$-linearisation of $\catC$.
There may then be automorphisms of $\C\catC$ that induce symmetries among representations. 
For example overall sign change in braid or symmetric representations.  

\mdef  \label{pa:rs}
Automorphisms of $\catC$, and compounds, such as: opposite/transpose; inverse; 
{antidiagonal transform (transform matrices by transpose then conjugate by the antidiagonal matrix).}
For example note from \ref{pr:MDgen} that $\MD$ is invariant under:
\\
$\bullet$ $r \leftrightarrow s$. 
\\
$\bullet$ reversing the relations (this can be combined with transpose in $\Mat$ so that 
solutions are invariant under transpose.

\subsection{On non-local but unentangled equivalences} \label{ss:non-loc}

It is an important open problem to understand non-local equivalences. 
This is true already in the braid case, but the hope is that the extra constraints of the MD case will shed new light. 
We start simply with some toy examples. 

\newcommand{\hex}{\mathsf{h}}
\newcommand{\vex}{\mathsf{v}}
\newcommand{\flex}{\mathsf{x}}

Let us write $\hex$ for the non-identity 2x2 permutation matrix. We have 
\[
\hex = \mat 0&1 \\ 1& 0 \tam, \hspace{2cm} 
\vex := 1_2 \otimes \hex = \mat 0&1 \\ 1&0 \\ &&0&1 \\ &&1&0 \tam , \hspace{1cm} 
\flex :=\hex \otimes \hex = \mat 0&0&0&1\\ 0&0&1&0 \\ 0&1&0&0 \\ 1&0&0&0 \tam
\]
- so $\flex$ is local while $\vex$ is non-local but unentangled, with both involutive  - 
and let
\[
R_f(p) = \mat 1 \\ &&p \\ &1/p \\ &&&1 \tam , \hspace{1cm} 
R_a(p) = \mat 1 \\ &&p \\ &1/p \\ &&&-1 \tam 
\]
- both of which give braid reps for all $p\in\C^\times$. 
We have the non-local and local conjugations
\[
\vex.R_f(p).\vex = \mat 0&&&p \\ &1&0 \\ &0&1 \\ \frac{1}{p} &&&0 \tam ,
\hspace{1cm}
\flex.R_f(p).\flex  \; = \; R_f(\frac{1}{p}) , 
\hspace{1cm} 
\]
The first of these is a solution despite being non-locally conjugated; and despite also 
that (as the second result shows) 
$R_f(p)$ does not commute with $\hex\otimes\hex$ in general (specifically unless $p^2 =1$), so it is not of DS type. 
Note for comparison that $\vex.R_a(p).\vex$ is {\em not} a braid representation - more as one might expect. 
$\;$ 
This appears to raise several nice questions, but we will demote these to a later work. 

\mdef 
Equivalences for varieties.  
Our solutions are often members of varieties, so continuously related within their variety, 
although the continuous variation may change characters, and hence isomorphism class in the ordinary sense.  Organising by varieties is certainly a useful tool, but then one must drill down further to determine implications for representation theory. We will consider examples in \S\ref{ss:analyse}. 


\subsection{Towards classification of solutions in rank 2: 
classification of functors $F:\Sym\rightarrow\Mat^2$
}  $\;$ 

\medskip 

In this section we  
prepare 
to classify solutions $F:\MD\rightarrow \Mat^2$ 
with $F(1)=1$  
\ppmm{(note that, unlike the charge-conserving case,  
solutions with $F(1)=a>1$ are just the same as solutions $F:\MD\rightarrow\Mat^{2^a}$, which we address later)}.
That is, solutions where $R = F(r)$ is a 4x4 matrix. 
The classification Theorem itself is in \S\ref{ss:Theorem}.

Since restriction to $R$ only, i.e. forgetting $S$, yields an `involutive' braid representation - i.e. with $R^2 = 1$, we can start the classification by 
recalling the classification of such braid reps. This comes from Hietarinta's rank-2 braid rep classification \cite{Hietarinta1992} by restricting to the involutive cases.

Remark: It is interesting to note that of course this is the same as classifying `symmetric representations' - functors $F:\Sym \rightarrow \Mat^2$. 
In particular 
unitary ones 
have been classified up to $\infty$-equivalence in \cite{LechnerPennigWood}, but as we will see this is not the same as the classification 
up to $H_1$-equivalence that is needed here 
(see in particular Comment~(4) below). 

\mdef {\bf Proposition}.   \label{pr:N2inv}
Up to local equivalence, scalar multiplication by -1, and taking the transpose 
(collectively `$H_1$-equivalence') 
the involutive {rank 2} braid reps are given by the following {five} varieties:
\[  
\mat 1& \\ &1 \\ &&1 \\ &&&1 \tam, \hspace{.251cm} 
\mat 1&-p&p&pq \\ &0&1&q \\ &1&0&-q \\ &&&1  \tam, \hspace{.25cm}
\mat 1&0&0&p \\ &0&1&0 \\ &1&0&0 \\ &&&-1  \tam, \hspace{.25cm}
\mat 1&0&0&0 \\ &0&q&0 \\ &\frac{1}{q}&0&0 \\ &&&\pm1  \tam, \hspace{.25cm}
\mat 0&0&0&1 \\ &1&0  \\ &0&1 \\ 1 &0&0&0  \tam 
\]
{We observe that two of these are obtained from charge-conserving (CC) solutions \cite{MR1X} by adding 'glue' i.e., some entries in the upper right.  Two others are already charge-conserving.} Thus convenient names for these are, respectively: 
\\ \mbox{ } \hspace{-.321cm}
trivial, \hspace{1cm} f-glue-type, \qquad \qquad
a-glue-type, \qquad f/a-slash type, \hspace{.1cm}  
anti-slash type. 

\medskip 

Before proceeding (even to proof!), some comments are in order:
\\
(1)
There is some overlap between the varieties. 
For example f-glue with $p=q=0$ is the same as f-slash with its $q=1$ 
(i.e. they meet in the \emph{flip} solution). 
\\
(2)
Note that parameters play a different role here to the braid rep classification.
Any parameter that varies continuously cannot change the $\infty$-equivalence class, since $\Sym$ is finite
- see \ref{de:l-ch-equiv}(III). 
Thus there are only five such classes here, and only three up to sign. 
\\
(3)
One can compare with the (elementary, partial) classification by $\Sym_2$ character (or 2-equivalence class in the sense of 
\ref{de:l-ch-equiv} or 
\cite{MRT25}): of course the trace of $R$ here is one of 4,2,0,-2,-4. 
Observe that the trivial 
$H_1$-class contains the 4 and -4 character classes. 
The two f-types  and anti-slash contain the 2 and -2 character classes.
And the two a-types together give the 0 character class.
\\
(4)
One can compare with the classification of {\em unitary} representations 
up to $\infty$-equivalence
from \cite{LechnerPennigWood}. These are indexed by pairs of Young diagrams such that the total degree is 2. Thus we have $(2,0)$, $(1^2,0)$, $(1,1), (0, 1^2), (0,2)$, corresponding to the trivial, f/a-slash type with $q=1$, the anti-slash type and overall sign changes.
\\

\medskip 

\noindent   {\em Proof}. 
One approach is simply to list all rank-2 braid reps up to the given equivalences and then impose the involutive condition. 
The ten `non-trivial' varieties of such reps can be deduced from \cite{Hietarinta1992} as follows. 
\input{tex/HietarintaN2}
The convenient working names for these varieties are, respectively:
\\ 
slash;  $\;$ anti-slash; $\;$ f-glue-I; $\;$ f-glue-II; $\;$ f-glue-III; 
\qquad 
a; $\;$ a-glue; $\;$ f;  \qquad Ising; $\;$ 8-vertex.
\\ 
Observe that f-glue-I;  and Ising are not involutive; 
f-glue-II is never involutive;
and a; f; and 8-vertex are never involutive except where they touch another variety. 
This brings us to the five varieties as claimed. 
\qed

\mdef  
It is an important question why 
$H_1$ is the right notion of equivalence to use here.
In particular why 
to treat anti-slash separately, since it is DS-equivalent to slash/flip - see \S\ref{ss:non-loc}.  
In short the reason is that we are anticipating using the classification 
as a starting point for other classifications of objects that 
(1) have a restriction 
to this involutive braid case; 
(2) respect $H_1$ equivalence, but not necessarily stronger equivalences.
In particular $\MD$ does not respect DS-equivalence applied only to $R$ in general, so the conjugation 
that takes anti-slash to slash takes a corresponding $S$ to 
an incommensurate matrix. Indeed, DS-equivalence as in \eqref{DS equiv for MD} would have to be simultaneously applied to $R$ and $S$.

We will return to this point when we analyse the results from the main Theorem.

\subsection{Manji calculus}  $\;$ 

\newcommand{\AAA}{\mathtt{A}}
\newcommand{\NNN}{\mathtt{N}}
\newcommand{\RRRR}{\mathtt{R}}

As we see from \ref{pr:N2inv} and its proof, most braid representations can be 
expressed in the CC-with-glue form.  However there is one which is better expressed 
in the anti-slash form - and when we come to extending to $\MD$ this sector will benefit from a 
different calculus. 
In further preparation for the statement and proof of this part of the Theorem it will be 
useful to have a couple of interesting varieties of braid representations to hand which are far 
from CC form. 

\mdef 
Define the `manji matrices':
\[  \hspace{-1.25cm} 
\PP_\pm = \mat
1 \\ &0&\pm 1 \\ &\pm 1&0 \\ &&&1 \tam , 
\quad
\AAA_\pm = \mat 0&&&\pm 1 \\ &1 \\ &&1 \\ \pm 1&&&0 \tam , 
\quad
\NNN_\pm = \mat 0&&1 \\ \pm 1&0 \\ &&0&\pm 1 \\ &1&&0 \tam , 
\quad
\NNN'_\pm = \mat 0&1 \\ &0&&\pm 1 \\ \pm 1&&0 \\&&1&0 \tam
\]
and 
\beq  \label{eq:PANN}
\RRRR_{\pm} = \mat a&b&c&\pm d \\ \pm c&d&\pm a&\pm b \\ \pm b&\pm a&d&\pm c \\ \pm d&c&b&a \tam 
 =  a\PP_{\pm} + d \AAA_{\pm} +c \NNN_{\pm} +b \NNN'_{\pm} 
\eq

\mdef   \label{lem:PANN}
{\bf{Lemma}}.
The matrix $\RRRR_+  =  \RRRR_+(a,b,c,d)$  from (\ref{eq:PANN})
satisfies the Yang--Baxter equation; and hence gives a braid representation whenever invertible, and hence for almost all $a,b,c,d$. 
\medskip
\\
{\em Proof}. In this rank it is a routine calculation. \qed

\subsection{
Rank-2 classification of $\MD$ representations: Theorem and proof}   $\;$ 
\label{ss:Theorem}

To find all rank-2 pairs $(R,S)$ that yield a monoidal functor  
$F:\MD\rightarrow \Mat$ 
with $F(1)=2$
we work up to local
equivalence. We thus
choose an $R$ from Proposition  \ref{pr:N2inv} and then interrogate the relations $\RRR$ to find all possibilities for $S$ from a generic ansatz.

\mdef\label{thm: big theorem} {\bf{Theorem}}. The complete set of strict monoidal functors $F:\MD\rightarrow \Mat$ 
with $F(1)=2$ is given, up to local equivalence, transpose, and overall sign, as follows.

There are a few cases to consider 
according to the choice for $R$. 
We separate out a-slash and f-slash types; and
treat flip separately, so there are 7 cases:
\begin{enumerate}    
\item  \label{caseR=1}
    If $R=I$ then the relations imply that $S_1=S_2$, thus $S=\pm I$.

\item  \label{caseaglue}
    Next suppose that $R$ is $a$-glue-type, i.e. $R=\left[ \begin {array}{cccc} 1&0&0&p\\ \noalign{\medskip}0&0& 1&0
\\ \noalign{\medskip}0& 1&0&0\\ \noalign{\medskip}0&0&0&-1\end {array}
 \right]$
 as in Prop.\ref{pr:N2inv}.  
 We assume here that $p\neq 0$. The case $p=0$ will be covered 
 in case~\ref{case5}  
below.  
For any $p \neq 0$ the set of possible $S$ values is given,
up to $\pm 1$ rescaling, by:
\beq  \label{eq:S-case2} 
    S=\left[ \begin {array}{cccc} 1&0&0&q\\ \noalign{\medskip}0&0& 1&0
\\ \noalign{\medskip}0& 1&0&0\\ \noalign{\medskip}0&0&0&-1\end {array}
 \right]
\eq  
where $q$ can take any value.

\ignore{{ 
Jordan form of $X$ in Case (2): 
 $\left[ \begin {array}{cc} 1&1\\ \noalign{\medskip}0&1\end {array}\right]\oplus \left[ \begin {array}{cc} 1&1\\ \noalign{\medskip}0&1\end {array}\right]$
}}

\item  \label{casefglue}
Next  suppose that  
we are in the $f$-glue case:
\[
R=  \left[ \begin {array}{cccc} 1&-p&p&qp\\ \noalign{\medskip}0&0&1&q
\\ \noalign{\medskip}0&1&0&-q\\ \noalign{\medskip}0&0&0&1\end {array}
 \right]
\]
We assume here that $(p,q) \neq (0,0)$. We will deal with the omitted flip case in 
(\ref{caseFlip}). In fact 
we will assume  that $q\neq 0$
(we are saying that at least one of $p,q$ does not vanish, 
but if $q=0$, so $p\neq 0$, we can get these solutions from those given below by antidiagonal symmetry). 
There are two cases: 
\\ 
If $p\neq -q$ then we are in a Wangian case $S=R$ up to scaling by $\pm 1$.
\\
If $p=-q$ we obtain, up to the usual scaling by $\pm 1$ symmetry,   
\beq   \label{eq:casep-qfglue} 
[R,S]= \left[\left[ \begin {array}{cccc} 1&q&-q&-{q}^{2}\\ \noalign{\medskip}0&0&1
&q\\ \noalign{\medskip}0&1&0&-q\\ \noalign{\medskip}0&0&0&1
\end {array} \right] ,
 \left[ \begin {array}{cccc} 1&s&-s&-{s}^{2}\\ \noalign{\medskip}0&0&1
&s\\ \noalign{\medskip}0&1&0&-s\\ \noalign{\medskip}0&0&0&1
\end {array} \right] \right],
\eq 
where $q\neq 0$ and $s$ can take any value. 
\ignore{{ 
The Jordan form of $X$ is  $\left[ \begin {array}{cccc} 1&1&0&0\\ \noalign{\medskip}0&1&1&0
\\ \noalign{\medskip}0&0&1&0\\ \noalign{\medskip}0&0&0&1\end {array}
 \right].$
 }}

\item  \label{case5}
Next we consider $R$ of the $a$-slash type: 
\[ 
R=\left[ \begin {array}{cccc} 1&0&0&0\\ \noalign{\medskip}0&0&p&0
\\ \noalign{\medskip}0&\frac{1}{p}&0&0\\ \noalign{\medskip}0&0&0&-1
\end {array} \right] 
, \hspace{1cm} p \in \C^\times .
\]
Since we make no assumptions about $p$ except that it is non-zero, 
this will cover the 
$a$-glue case 
omitted from (\ref{caseaglue})
by setting $p=1$.  
\\ 
If   $p^2 \neq 1$   
we get 
$S=\left[ \begin {array}{cccc} 1&0&0&0\\ \noalign{\medskip}0&0&s&0
\\ \noalign{\medskip}0&\frac{1}{s}&0&0\\ \noalign{\medskip}0&0&0&\pm 1
\end {array} \right]$.  
\qquad 
If $p=\pm 1$ we get:
$S=\left[ \begin {array}{cccc} 1&0&0&s\\ \noalign{\medskip}0&0&\pm 1&0
\\ \noalign{\medskip}0&\pm 1&0&0\\ \noalign{\medskip}0&0&0&-1\end {array}
 \right]$. 


\item   \label{casefslash}
Consider $R$ of the $f$-slash type:

$R= \left[ \begin {array}{cccc} 1&0&0&0\\ \noalign{\medskip}0&0&p&0
\\ \noalign{\medskip}0&\frac{1}{p}&0&0\\ \noalign{\medskip}0&0&0&1
\end {array} \right]
$. 
We assume $p\neq 1$ 
(else see case (\ref{caseFlip})), 
but allow $p=-1$.  
{
That is, if $p \neq 1$ 
we get the varieties indicated by 
\beq   \label{eq:casepfslash}
S=  \left[ \begin {array}{cccc} 1&0&0&0\\ \noalign{\medskip}0&0&s&0
\\ \noalign{\medskip}0&\frac{1}{s}&0&0\\ \noalign{\medskip}0&0&0& \pm 1
\end {array} \right]  .
\eq 
}
\ignore{{
We also obtain the solution 
$S=\left[ \begin {array}{cccc} 1&0&0&0\\ \noalign{\medskip}0&0&s&0
\\ \noalign{\medskip}0&\frac{1}{s}&0&0\\ \noalign{\medskip}0&0&0&1
\end {array} \right]$. 
}}

In the case $p=-1$ we  also  
have 
\beq   \label{eq:casep-1fslash}
S=\left[ \begin {array}{cccc} 0&0&0&s\\ \noalign{\medskip}0&\pm1&0&0
\\ \noalign{\medskip}0&0&\pm1&0\\ \noalign{\medskip}\frac{1}{s}&0&0&0
\end {array} \right]  .
\eq 
\ignore{{ 
\\ 
This yields \[X= \left[ \begin {array}{cccc} 0&0&0&s\\ \noalign{\medskip}0&0&\pm 1&0
\\ \noalign{\medskip}0&\pm1&0&0\\ \noalign{\medskip}\frac{1}{s}&0&0&0
\end {array} \right],\] having eigenvalues: $[\pm 1,\pm 1]$.
}}

\item   \label{caseantislash}
Now consider the anti-slash case:
$R=\left[ \begin {array}{cccc} 0&0&0&1\\ \noalign{\medskip}0&1&0&0
\\ \noalign{\medskip}0&0&1&0\\ \noalign{\medskip}1&0&0&0
\end {array} \right] 
$.  
We have several choices for $S$ in this case:
\begin{enumerate}
\setlength\itemsep{1em}
\item
$S = z\,\PP_++(-\varepsilon - z)\,\AAA_+ + x\,\NNN_+ - x\,\NNN'_+ = \begin{bmatrix}
z & -x & x & -\varepsilon - z\\
x & -\varepsilon - z & z & -x\\
-x & z & -\varepsilon - z & x\\
-\varepsilon - z & x & -x & z
\end{bmatrix}$ where $\varepsilon^2=1,\quad x^2 = z^2 + \varepsilon z.$

\item $
S = r\PP_+ - r\AAA_+ + y\NNN_+ + (\varepsilon - y)\NNN'_+ \;=\;
\left[ \begin {array}{cccc}
r & -y+\varepsilon & y & -r\\
\noalign{\medskip} y & -r & r & -y+\varepsilon\\
\noalign{\medskip} -y+\varepsilon & r & -r & y\\
\noalign{\medskip} -r & y & -y+\varepsilon & r
\end {array} \right],$ where
$ \varepsilon^2=1,\; 2r^2 - 2y^2 + 2\varepsilon y - 1=0.
$

\item
$\left[
\begin{array}{cccc}
-r & y & y & -r-\varepsilon\\
-y & r+\varepsilon & r & -y\\
-y & r & r+\varepsilon & -y\\
-r-\varepsilon & y & y & -r
\end{array}
\right],$ where $ \varepsilon^2=1,\quad r^2 - y^2 + \varepsilon r = 0.$  

\end{enumerate}

\item  \label{caseFlip}
    If $R$ is the flip, then we observe that $R$ is invariant under all local symmetries as well as transposition.  Thus we may assume that $S$ is in one of the forms in Proposition 
    \ref{pr:N2inv}.  
    Thus we may 
    treat this case using  
    the $R\leftrightarrow S$ symmetry (\ref{pa:rs}). 
    
    Using this symmetry and the solutions  
    (\ref{caseaglue} - \ref{caseantislash})
    we find that if $R$ is the flip then we have $S$ among flip, 
    anti-slash (consider case (\ref{caseantislash})(a) with $x=0$, \ppmm{$\varepsilon=-1$} and $z=1$), 
    a-slash 
    or $f$-glue type with $p=-q$.

\end{enumerate}

\newcommand{\genericS}{
\left[\begin{array}{cccc}
a  & b  & c  & d  
\\
 e  & f  & g  & h  
\\
 j  & k  & l  & m  
\\
 n  & s  & t  & r  
\end{array}\right]
}

\medskip 

\noindent 
{\em Proof}. Having used local equivalence to write $R$ in each of the forms in the transversal, we then take an entirely general ansatz for $S$:
\beq  \label{eq:genericS} 
S = \genericS
\eq 
We now address the various possibilities for $R$. In most cases there are an 
overwhelming array of constraints to check. But the `SRR-anomaly' 
$$
\An_{SRR} \; := \; S_1R_2R_1-R_2R_1S_2
$$ 
tends to be most manageable, since it is only linear in the general ansatz, while all the $R$-matrix cases are relatively sparse. We may proceed as follows:
\\ 
(\ref{caseR=1}) (trivial-type)  is  elementary.   
\\
(\ref{caseaglue})   (a-glue-type)
Consider the general $S$ and hence the `SRR-anomaly' 
$\An_{SRR} \; = \; S_1R_2R_1-R_2R_1S_2$ in this case:
\input{tex/mapleSRR1}
Vanishing of this anomaly reduces $S$ immediately (up to the usual overall sign) to 
\[
\left[\begin{array}{cccc}
1  & 0  & 0  & d  
\\
 0  & f  & f-r  & 0  
\\
 0  & 1-f  & 1+r-f  & 0  
\\
 0  & 0  & 0  & r  
\end{array}\right]
\]
with trace $2+2r$. 
By case~(\ref{caseR=1}) 
and the $R\leftrightarrow S$ symmetry the trace cannot be 4, so $r=-1$. 
The other braid relation anomalies are now straightforward to compute. 
Vanishing then forces $f=0$,
but the relations are then satisfied, verifying (\ref{eq:S-case2}), and we are done.

\medskip 

\newcommand{\rv}[1]{| #1\rangle}
\newcommand{\lv}[1]{\langle #1 |}

\noindent 
(\ref{casefglue})  (f-glue-type) 
Here considering 
$\An_{SRR}= S_1R_2R_1-R_2R_1S_2$  we proceed as follows:
\[
\lv{222} \; R_2 \; =\; \lv{222}  
,\hspace{1.3cm}
\lv{222} \; R_2 R_1 \; =\; \lv{222}
,\hspace{1.3cm}
\lv{222} \; R_2 R_1 S_2 \; =\; [0,n,0,s,0,t,0,r]   
\] 
\[ 
\lv{222} S_1 = [0,0,0,0,n,s,t,r]
,\hspace{1cm}
\lv{222} S_1 R_2 R_1 = [0,n,0,s-nq,0,t-nq,0,q(nq-t-s)+r]
, 
\]\[
\langle 222 | \; \An_{SRR} \; =\; [0,0,0,-nq,0,-nq, q(nq-t-s)]
\]
so $n=0$ and $t=-s$; 
\[ \hspace{-1cm}
\lv{122} R_2 = \lv{122}  
,\hspace{.61cm}
\lv{122} R_2 R_1 = [0,0,0,0,0,1,0,-q]  
,\hspace{.61cm}
\lv{122} R_2 R_1 S_2 = [0,j-nq,0,k-sq,0,l-tq,0,m-rq]   
\] 
\[ 
\lv{122} S_1 = [0,0,0,0,j,k,l,m]
,\hspace{1cm}
\lv{122} S_1 R_2 R_1 = [0,j,0,k-jq,0,l-jq,0,q(jq-l-k)+m]
, \hspace{1cm} 
\]
\[
\langle 122 | \An_{SRR} \; = [0,nq,0,q(s-j),0,q(t-j), q(jq-l-k+r)]
\]
so $s=j=t=-s=0$ and $r=l+k$. Similarly 
\[
\langle 212 | \An_{SRR} \; = [0,nq,0,q(s-e),0,q(t-e), q(eq-g-f+r)]
\]
\[
\langle 112 | \; \An_{SRR} \; =\; q[0,e+j-nq,0,f+k-a-qs,0,g+l-a-qt,0,h+m-b-c +aq-r q]
\]
So  
$e=0$. 
And $r=l+k$, $r=g+f$ (so $2r=l+k+g+f$), $l=a-g$ and $k=a-f$.
So $2a = l+g+k+f =2r$, so $r=a$ and $m=b+c-h$. 
We have
\[ 
\An_{SRR} \leadsto (p+q) 
\input{tex/mapleSRRc2subs}
\] 
Since $S^2 =1$ we now have $a=\pm 1$, so can take $a=1$ up to overall sign. 
\ignore{{
The full anomaly now becomes:
\input{tex/mapleSRR4}
Noting immediately that $n=0$ and $j=-e$, and thus $e=j=s=t=0$, 
we see that $a=\pm 1$. Up to overall sign we may take $a=1$, and $k=1-f$, 
and $l=1-g$,
and hence $f+g=1$ and $r=1$.
Thus $c+b=h+m$.
}}
Already $S$ is reduced to
\[
\genericS \; \leadsto \; 
\left[\begin{array}{cccc}
1  & b  & c  & d  
\\
 0  & f  & 1-f  & h  
\\
0  & 1-f & f  & m  
\\
0  & 0  & 0  & 1  
\end{array}\right]
\]
Since $S^2 =1$ we have $f^2  -f=0$. 
By the ordinary involutive braid rep classification Prop.\ref{pr:N2inv} and 
the condition $S^2 =1$
(equation (\ref{eq:ss1})) 
the det of the middle 2x2 block must be -1, so $f=0$. 
Thus $b=-c$ and $h=-m$ and $d=bm$. 
The anomaly becomes:
\input{tex/mapleSRR4n0}
If $p+q\neq 0$ we have $b=-p$, $c=p$, $h=q$ and $m=-q$, and 
$d=-2pq(q   +p)/(-2(p+q)) = pq$,
and the rep is the Wangian for this $R$.
\\
If $p+q = 0$ then 
this SRR-anomaly now vanishes. 
From the SSR-anomaly we get $b=h$ and $c=m$.  
We have reached
\[   S \leadsto \; 
\left[ \begin {array}{cccc} 1&b&-b& -b^2
\\ \noalign{\medskip}0&0&1&b
\\ \noalign{\medskip}0&1&0&-b
\\ \noalign{\medskip}0&0&0&1\end {array}
 \right]
\]
The pair $[R,S]$ now satisfy all relations as claimed in (\ref{eq:casep-qfglue}). 

\medskip 

\noindent 
(\ref{case5}) (a-slash) 
Again consider $SRR-RRS$ with our generic $S$:
\input{tex/mapleSSR5}

\noindent
Thus $b=c=e=j=s=t=h=m=0$  (the `8-vertex condition') and so 
\beq   \label{eq:S2-8v}
S^2 = \left[  \begin{array}{cccc}
  a^2 +dn   &         &        &  (a+r)d  \\
            & f^2 +gk  & (f+l)g & \\
            & (f+l)k  & l^2 +gk   \\ 
  (a+r)n    &         &         & r^2+dn
\end{array}
\right]        
\eq 

\noindent 
(I) In case $\ppp^2 \neq 1$:
\\ 
we have   also $d=n=0$ from the SRR-anomaly.   
We can then also take $a=1$ without loss of generality. 
With this simplified form we can now consider also the SSR-anomaly.
We get $fl=0$ and (since $S$ is involutive) $l^2 = f^2$ so $l=f=0$ and $gk=1$
and $r^2 =1$ and we are done.
\medskip \\
(II) In case $\ppp^2 =1$:
\\
then  also    
$lf=dn=df=ln=0$ from SSR and $l^2 = f^2$ from involutivity,
so in fact $l=f=0$ and $gk=1$.
Since $dn=0$ we may put $a=1$ WLOG and $r=\pm 1$. 
We have reached
\[
\left[\begin{array}{cccc}
1  & 0  & 0  & d  
\\
 0  & 0  & g  & 0  
\\
0  & 1/g & 0  & 0  
\\
n  & 0  & 0  &  \pm 1 
\end{array}\right]
\]
(IIi) If $d=n=0$ all constraints are now satisfied.
\\ 
(IIii) If $r=1$ we have $d=n=0$ by involutivity; and then all constraints are satisfied - a subcase of the case~(IIi). 
\\ 
(IIiii) If $r=-1$ then one of $d,n$ can be nonzero;
and then $g=\pm 1$.
We can take  
$d\neq 0$ WLOG; and then 
all constraints are satisfied except finally SSR gives $g\ppp =1$. 

\medskip 

\noindent 
(\ref{casefslash})  (f-slash) 
For the SRR-anomaly  $\An_{SRR}$ here we have:
\[
\left[\begin{array}{cccccccc}
0 & 0 & b \ppp -b  & 0 & c \ppp -c  & 0 & d \,\ppp^{2}-d  & 0 
\\
 -e \ppp +e  & 0 & 0 & 0 & 0 & 0 & h \,\ppp^{2}-h \ppp  & 0 
\\
 -j \ppp +j  & 0 & 0 & 0 & 0 & 0 & m \,\ppp^{2}-m \ppp  & 0 
\\
 -n \,\ppp^{2}+n  & 0 & -\ppp^{2} s +\ppp s  & 0 & -\ppp^{2} t +\ppp t  & 0 & 0 & 0 
\\
 0 & 0 & 0 & \frac{b}{\ppp}-\frac{b}{\ppp^{2}} & 0 & \frac{c}{\ppp}-\frac{c}{\ppp^{2}} & 0 & d -\frac{d}{\ppp^{2}} 
\\
 0 & \frac{e}{\ppp^{2}}-\frac{e}{\ppp} & 0 & 0 & 0 & 0 & 0 & h -\frac{h}{\ppp} 
\\
 0 & \frac{j}{\ppp^{2}}-\frac{j}{\ppp} & 0 & 0 & 0 & 0 & 0 & m -\frac{m}{\ppp} 
\\
 0 & \frac{n}{\ppp^{2}}-n  & 0 & \frac{s}{\ppp}-s  & 0 & \frac{t}{\ppp}-t  & 0 & 0 
\end{array}\right]
\]
Recall that $p\neq 1$ here. 
We see immediately that  
$b=c=e=j=h=m=s=t=0$ (the `8-vertex' condition for $S$), 
whereupon the SRR-anomaly is zero'ed either by $p =-1$ or $d=n=0$.
The SSR anomaly has become:
\[
\input{tex/mapleSSR4f-bcejhmst0}
\]
And $S^2$ is now as in (\ref{eq:S2-8v}), so from $S^2 =1$ we have $l^2 = f^2$.

Note that if $n=0$ we have $f=l=0$ and hence 
we can take 
$a=1$, $dr=-d$ and $gk=1$ and $r^2 =1$ from $S^2=1$;
and $kd/p = -d$ and $d/p = dg$ from $\An_{SSR}$; and hence $d=0$.

\ignore{{ 
\ppm{[delete rest of this para now?]}
From the $SSR$-anomaly we then have $ f(ap-g) =0$ and $dnp^2  =lf$; and
from the SSS-anomaly we  have $ fl(f-l)=0$ and $dng=lfk$ and
$n(a(l-a) +k^2 -lr )=0$.
Since $S^2 =1$ we have 
$f^2+kg =1$, $kg+l^2 =1$, $g(f+l)=k(f+l)=0$
(so either $f=-l$ or $g=k=0$);
and
$a^2 +dn=1$, $nd+r^2 =1$ $ d(a+r) = n(a+r)=0 $
(so either $a=-r$ or $d=n=0$).
\\
If $fl\neq 0$ then $f=l=\pm 1$ and $g=k=0$ and so either $a=0$ (and hence $r=0$) or $f=0$ and $f=l=-1$ - a contradiction. 
If $a=r=0$ then $dn= 1$ and 
here 
\ppm{[complete argument to eliminate this or try different route!]} ... 
So $f=l=0$ and so $dn=0$. \ppm{[--delete to here?]}
\\
}}

\medskip 

\noindent 
(I) If $p=-1$ then: 
\\
(Ii) if $n=0$ 
then as noted 
$f=l=0$; and we can take $a=1$; and  
$gk=1$, 
and $d=0$; and 
all constraints are satisfied, as in (\ref{eq:casepfslash}).
\\ 
(Iii) if $n\neq 0$ we have $g=-a$, $r=-k$, $r=-a$, so $k=a$. 
We deduce $f \neq 0$ so $a=r=k=g=0$, $nd=1$ and $f=l$ and $f^2 =1$. 
All constraints are now satisfied, as in (\ref{eq:casep-1fslash}). 
\medskip 
\\
(II) If $p\neq -1$, so $p^2 \neq 1$,  
then in addition $n=d=0$ (making the SRR-anomaly zero) and so 
$l=f=0$ from above and 
$S$ is CC with $a=1$, WLOG, and $r=\pm 1$ and $gk=1$. 
All constraints are now satisfied,   
and we are done here, as in (\ref{eq:casepfslash}).



\medskip 


\medskip 

\noindent 
(\ref{caseantislash}) 
(Anti-slash) $\;$   
The SRR anomaly here is: 
\[ 
\An_{SRR} = 
\left[\begin{array}{cccccccc}
d -n  & 0 & c -s  & 0 & b -t  & 0 & a -r  & 0 
\\
 h -j  & 0 & g -k  & 0 & f -l  & 0 & e -m  & 0 
\\
 m -e  & 0 & l -f  & 0 & k -g  & 0 & j -h  & 0 
\\
 r -a  & 0 & t -b  & 0 & s -c  & 0 & n -d  & 0 
\\
 0 & d -n  & 0 & c -s  & 0 & b -t  & 0 & a -r  
\\
 0 & h -j  & 0 & g -k  & 0 & f -l  & 0 & e -m  
\\
 0 & m -e  & 0 & l -f  & 0 & k -g  & 0 & j -h  
\\
 0 & r -a  & 0 & t -b  & 0 & s -c  & 0 & n -d  
\end{array}\right]
\] 
so from this, 
writing $\RRRR_+$ from (\ref{eq:PANN}) as $\RRRR_+ = \RRRR_+(a,b,c,d)$, 
and similarly for $\RRRR_-$, we have 
\beq   \label{eq:SSsub1}
S \leadsto  
\left[\begin{array}{cccc}
a  & b  & c  & d  
\\
 e  & f  & g  & h  
\\
 h  & g  & f  & e  
\\
 d  & c  & b  & a  
\end{array}\right]
= \; \RRRR_+(\frac{a+g}{2},\frac{b+h}{2},\frac{c+e}{2},\frac{d+f}{2}) 
 \; +\; \RRRR_-(\frac{a-g}{2},\frac{b-h}{2},\frac{c-e}{2},\frac{-d+f}{2})
\eq 
From $\An_{SSR}$ we then get $h=\pm b$; and $g=\pm a$, $(c-e)(a+g)=0$, $ah=bg$ and $ab=gh$ (and some other identities). 
(While from involutivity of $S$ we get $ad=gf$ and other identities.)

In subcase $h=b$, assuming $b \neq 0$, we get $e=c$, $f=d$, $g=a$, 
whereupon the SSR anomaly vanishes. 
And in fact  $S \leadsto \RRRR_+$  from (\ref{eq:PANN}),  so 
the SSS anomaly also vanishes, by Lem.\ref{lem:PANN}.
\ignore{{
setting
\[  \hspace{-1.25cm} 
\PP_\pm = \mat
1 \\ &0&\pm 1 \\ &\pm 1&0 \\ &&&1 \tam , 
\qquad
\AAA_\pm = \mat 0&&&\pm 1 \\ &1 \\ &&1 \\ \pm 1&&&0 \tam , 
\qquad
\NNN_\pm = \mat 0&&1 \\ \pm 1&0 \\ &&0&\pm 1 \\ &1&&0 \tam , 
\qquad
\NNN'_\pm = \mat 0&1 \\ &0&&\pm 1 \\ \pm 1&&0 \\&&1&0 \tam
\]
We have that 
$S \leadsto \RRRR_+$  
from (\ref{eq:PANN}). 
\beq  \label{eq:PANN}
\RRRR_+ = \mat a&b&c&d \\ c&d&a&b \\ b&a&d&c \\ d&c&b&a \tam 
 =  a\PP_+ + d \AAA_+ +c \NNN_+ +b \NNN'_+
\eq 
{\bf{Lemma}}.
This $\RRRR_+$ satisfies the Yang--Baxter equation, and hence gives a braid representation whenever invertible, and hence for almost all $a,b,c,d$. 
}}
\medskip 

Involutivity of $S$ then gives $(b+c)(a+d)=0$, $a^2 +d^2 = 1-2cb$, $b^2 + c^2=-2ad$.
Thus either 
(I) $b=-c$, $a+d=\pm 1 $ and $ad=bc=-b^2$;
or (II) $a=-d$, $b+c=\pm 1$ and $bc-ad=1/2$. 
\\
In case (I) consider (I1) $a+d=1$. This is (\ref{caseantislash}a) with $\epsilon=-1$.
Then (I2) $a+d=-1$ gives (\ref{caseantislash}a) with $\epsilon=1$.
\\
In case (II) consider (II1) $b+c=1$. This is (\ref{caseantislash}b) with $\epsilon=1$.
Then (II2) $b+c=-1$ gives (\ref{caseantislash}b) with $\epsilon=-1$.

\medskip

Staying with $h=b$ we now check the $b\neq 0$ assumption. Suppose $b=0$, then we have $h=t=j=0$. Here either $a=g=0$ or $d=f$. 
\\
If $a=g=0$ then $c=e=0$ and $d^2 = f^2 =1$ and then all constraints are satisfied. 
We have reached: 
$S = \pm\left[\begin{array}{cccc}
0  & 0  & 0  & \pm 1  
\\
 0  & 1  & 0  & 0 
\\
 0  & 0  & 1  & 0 
\\
 \pm 1  & 0  & 0  & 0  
\end{array}\right] $,
which are special cases of the claimed solutions. 
\\
If $f=d$ then recall that we have $g = \pm a$.
\\
If $g=-a$ then $ac=dc=0$ but $a,d$ are not both zero so $c=0$, so $ad=0$,
so $(a,d) = (0, \pm 1)$ 
{again forcing $e=0$,}
or $(\pm 1 , 0)$ {(also forcing $e=0$, for involutivity)} 
and then all constraints are satisfied.
We have 
$S =\left[\begin{array}{cccc}
0  & 0  & 0  & \pm 1  
\\
 0  & \pm 1  & 0  & 0 
\\
 0  & 0  & \pm 1  & 0 
\\
 \pm 1  & 0  & 0  & 0  
\end{array}\right] $ 
or $S =\left[\begin{array}{cccc}
\pm 1  & 0  & 0  & 0  
\\
 0  & 0  & \mp 1  & 0 
\\
 0  & \mp 1  & 0  & 0 
\\
 0  & 0  & 0  & \pm 1  
\end{array}\right]$,
all of which are cases of the claimed solutions.
\\
If $g=a \neq 0$ 
(recall we are in sub-subcase $h=b=0$ and $f=d$)
then the SSR anomaly vanishes if and only if $e=c$, whereupon the SSS anomaly also vanishes. Involutivity gives $a^2 + d^2 =1$, $(a+d)c=0$ and $c^2 =-2ad$, so 
either $c\neq 0$ and $d=-a$ or $c=0$. Both immediately yield special cases of the claimed solutions. {This completes the $h=b$ subcase.} 

\medskip 

Next  
consider subcase $b=-h$. 
{If $b=0$ then in fact $h=b$ and this is dealt with above, so we may assume $b\neq 0$.}
If $b\neq 0$ then $g=-a$ and the SSR anomaly vanishes 
provided that 
$a(d -f) + (c-e)b  =0$. 
So far we have 
\[
S \leadsto  
\left[\begin{array}{cccc}
a  & b  & c  & d  
\\
 e  & f  & -a  & -b  
\\
 -b  & -a  & f  & e  
\\
 d  & c  & b  & a  
\end{array}\right]
\]
(and the condition involving $e,f$). 
From involutivity we now get $a(d+f)=0$ and $(b+c)(d+f)=0$ 
(necessary but not yet sufficient). 
\\
{Consider the sub-subcase $a \neq 0$.} 
So if $a\neq0$ then $f=-d$. 
Finally the SSS anomaly gives $(c+e)(b-c)d=0$, so if $d\neq 0$ then either
$e=-c$ or $c=b$. 
In the $c=b$ case we get $a(b+e)^2 =0$ so $e=-b = -c$ anyway. 
In the $e=-c$ case 
 we  get  $(b-c)(b+c)^2=0$ and $(b-c)(1-(b+c)^2)=0$ so either $b=c$ or $b=-c$
in which case $b=c=0$.  
The anomaly then vanishes. That is 
\[
S \leadsto  \RRRR_- =
\left[\begin{array}{cccc}
a  & b  & b  & d  
\\
 -b  & -d  & -a  & -b  
\\
 -b  & -a  & -d  & -b  
\\
 d  & b  & b  & a  
\end{array}\right]
= a\PP_- +d \AAA_-  + b (\NNN_- + \NNN'_-) 
\]
which satisfies the YBE. 
The remaining constraint arises from the sufficient condition for involutivity:
$a=d\pm 1$ and then $a(a\mp 1)= b^2$. 
Observe that these are (\ref{caseantislash}c) and (\ref{caseantislash}c) with $\epsilon=\mp 1$. 
We are done with the sub-sub-subcase $d \neq 0$.
If $d=0$ 
{(we are in case anti-slash; subcase $h=-b$ with $b \neq 0$; sub-subcase $a\neq 0$)}
then $f=0$ and $e=c$. Involutivity gives $c=b$ and then $ab=0$ - a contradiction, so there is no solution here. 
\\
Finally for subcase $h=-b \neq 0$ we have the sub-subcase $a=0$. Here $c=e$ and then the SSR anomaly vanishes (as well as SRR). Indeed from SSS we have $c=0$ and then involutivity yields a contradiction ($f=0=\pm1$), so there is no solution here. 
\\
We are done.  

\medskip \medskip 

\noindent 
(\ref{caseFlip}) (Flip) is  now clear, by symmetry and bookkeeping.

\qed

\medskip  \hspace{.1cm} 

\section{Analysis of 
MD representations in the classification}   \label{ss:analyse}


\subsection{Structure of representations: some preliminaries}
For applications of paravortices to quantum information there are several important questions to be asked about the representations of $\MD_n$ associated with the pair $(R,S)$, such as: \emph{ is the image a finite group? is the image semisimple? when are these representations unitary? how do they decompose as a direct sum of indecomposable/irreducible representations?}  We will provide some partial answers and methods for answering these questions in what follows.

\mdef \label{question: infinite/semisimple?} Since $\MD_n$ is virtually abelian for all $n$, 
the potentially infinite {order}
part 
\ppmm{of the image as a group}
of any representation comes from the image of the normal abelian subgroup $\Z^{\binom{n}{2}}$.  
There is a trichotomy where $\Z^{\binom{n}{2}}$ has
\begin{enumerate}
    \item[(a)] Finite \ppmm{(and hence by Maschke's Theorem necessarily semisimple)} image
    \item[(b)] Infinite, but semisimple (completely reducible) image
    \item[(c)] Infinite, non-semisimple image.
\end{enumerate}

The first two situations occur when the image of $\Z^{\binom{n}{2}}$ consists of diagonalisable matrices.
\ignore{{, and we can employ Theorems \ref{thm:infinite clifford} and \ref{th:allreps}. }}
Since the images of the $\Z^{\binom{n}{2}}$ generators are all similar to the Kronecker product of the matrix $X:=RS$ with identity matrix, we have:
case (a) if $X$ has finite order and (b) if $X$ is diagonalisable but infinite order and (c) otherwise.

\medskip

\mdef \label{ss:imageX}\textbf{The image of $X$.}  Here we describe $X=RS$ in each case of the classification in Theorem \ref{thm: big theorem}.

\noindent 
Case (1) is trivial. 

\noindent 
Case (\ref{caseaglue}): 
Jordan form of $X$ is   
 $\left[ \begin {array}{cc} 1&1\\ \noalign{\medskip}0&1\end {array}\right]\oplus \left[ \begin {array}{cc} 1&0\\ \noalign{\medskip}0&1\end {array}\right]$

\noindent 
Case (\ref{casefglue}): If $p\neq -q$ then $S=\pm R$ and so of course  $X=\pm I$.  
\\  { . } \hspace{1cm} 
If $p=-q$ then 
the Jordan form of $X$ is  
$\left[ \begin {array}{cccc} 1&1&0&0\\ \noalign{\medskip}0&1&1&0
\\ \noalign{\medskip}0&0&1&0\\ \noalign{\medskip}0&0&0&1\end {array}
 \right].$

\noindent 
Case (\ref{case5}): If $p^2 \neq 1$ then  $X$ is diagonal with entries $[1,p/s,s/p,\mp 1]$. 
\\ { . } \hspace{1.0095cm} 
If $p= \pm 1$ then 
$X$ has Jordan form 
$\left[ \begin {array}{cc} 1&1\\ \noalign{\medskip}0&1\end {array}\right]\oplus \left[ \begin {array}{cc} 1&0\\ \noalign{\medskip}0&1\end {array}\right].$

\noindent 
Case (\ref{casefslash}):  If $p^2 \neq 1$ then 
$X$ is diagonal, with eigenvalues $[1,p/s,s/p,\pm 1]$.
\\
{ . } \hspace{1cm}  If $p=-1$ this 
yields 
$ X= \left[ \begin {array}{cccc} 0&0&0&s\\ \noalign{\medskip}0&0&\pm 1&0
\\ \noalign{\medskip}0&\pm1&0&0\\ \noalign{\medskip}\frac{1}{s}&0&0&0
\end {array} \right] $,
having eigenvalues: $[\pm 1,\pm 1]$.

\newcommand{\zz}{{ . } \hspace{1.1cm}}  

\noindent 
Case (\ref{caseantislash}): (a)  $Spec(X)=[\epsilon,\epsilon,\pm2x + \epsilon - 2z]$. 
\\ \zz (b) $Spec(X)=[\pm1,\pm 2y + \epsilon - 2w]$.
\\ \zz (c) $Spec(X)=[\pm 1,\pm 1]$ and $X$ is diagonalisable.

Case (\ref{caseFlip}): {These cases are covered by the above by $R\leftrightarrow S$ symmetry, but note that $X$ is diagonalisable in all case.}
\newcommand{\AAAA}{\mathscr{A}}
\newcommand{\BBBB}{\mathscr{B}}

\mdef   Let $\AAAA$ be an algebra acting faithfully on a module $Y$ 
(one can think of any algebra $\BBBB$ acting on $V$ and then consider $\AAAA$ as the quotient by the annihilator of the action). 
Thus $\AAAA':=End_{\AAAA}(V)$ is the algebra of endomorphisms of $V$ that commute with the action of $\AAAA$. 

\ignore{{
Here we are interested in cases where $V$ is a finite-dimensional vector space over the ground field $\C$, so $\AAAA$ is finite-dimensional, even if $\BBBB$ is not. 
(Indeed for us $\BBBB$ is typically the group algebra of an infinite group.) 
Since $\AAAA$ is an fd $\C$-algebra it has a unique decomposition of 1 into primitive central idempotents (PCIs) - see e.g. \cite[\S9B]{CurtisReiner}. 
Since $Y$ is faithful the image of every PCI is a non-zero central idempotent, 
and induces a direct summand of $Y$. Thus $Y$ decomposes into at least as many summands as 1 does. In particular if there is a central idempotent different from 1 then $Y$ is not indecomposable. 

\ecr{[to not belabour the point, will try to be a little less tied to our examples.]} 

\mdef  \label{de:stan-comm} 
In practice we fix a basis of $Y$, so $End(Y)$ is the corresponding matrix algebra. 
If we are given a finite collection of matrices $A_1,\ldots, A_k$ generating $\AAAA$, finding PCIs can be computationally expensive. Na\"ively one must:
\begin{enumerate}
    \item Solve the (homogenous) linear equations in the entries of 
    general matrix $T$ 
    \ppmm{$\in End(Y)$} 
    given by $TA_i=A_iT$ to obtain a basis for $\AAAA'$,
    \item find a basis for $\AAAA$ itself 
    \item obtain $Z(\AAAA)=\AAAA\cap\AAAA'$ by finding a basis for the intersection,
    \item solve $T^2=T$ in $Z(\AAAA)$ to find central idempotents, 
    \item decompose the idempotents into primitives.
\end{enumerate}
Any of these steps can be troublesome in practice. Step (1) is perhaps the least onerous, especially if the $A_i$ are sparse.   If $\AAAA$ is semisimple, then step (2) may be computed from step (1) using the double commutant theorem.

An alternative practical approach to decomposing $Y$ as an $\AAAA$-module is the following 'divide and conquer' idea.  First find a non-trivial (i.e. neither $0$ nor $I$)  $M\in\AAAA'$.   The generalized eigenspace decomposition of $Y$ with respect to $M$ provides a decomposition of $\AAAA$, since any $M$-invariant decomposition is also $\AAAA$ invariant.   One can then repeat this on the these invariant subspaces.  

The difficulty to step (1) above is that the general case has as many variables as the matrix $T$.  Tuning some of these to $0$ can be an efficient way to find a suitable $M\in\AAAA'$.  Moreover, in our situation the generating matrices are fairly sparse, so methods for finding a single solution are typically more efficient than finding a basis for solutions.  

As we noted above, the action of $\AAAA$ on $Y$ is faithful.  Thus if we can prove $Y$ is indecomposable as a module then $\AAAA$ is indecomposable as an algebra.
}}

Now $V$ is indecomposable if an only if $\AAAA'$ has no non-trivial idempotents. {On the other hand, if $\AAAA$ has a non-trivial decomposition, there is a $T\in\AAAA'$ whose eigenspace decomposition yields the decompostion.} We will use this below.

\medskip


\subsection{The MD representation in the f-glue case}

\medskip

We consider Theorem \ref{thm: big theorem} case \eqref{casefglue} where
\[
[R,S]= \left[\left[ \begin {array}{cccc} 1&q&-q&-{q}^{2}\\ \noalign{\medskip}0&0&1
&q\\ \noalign{\medskip}0&1&0&-q\\ \noalign{\medskip}0&0&0&1
\end {array} \right] ,
 \left[ \begin {array}{cccc} 1&p&-p&-{p}^{2}\\ \noalign{\medskip}0&0&1
&p\\ \noalign{\medskip}0&1&0&-p\\ \noalign{\medskip}0&0&0&1
\end {array} \right] \right].
\]
As for a-glue we can see that this case is Wangian when $q=p$, where it is isomorphic to `flip'. 
In other words,  
the representation theory is as for the classical $sl_2$ spin chain.
And again generically (for $q \neq p$) then $R-S$ lies in the radical.
This time however, unlike the a-glue case, the radical-squared does not vanish. 

\mdef By the HoG Theorem \cite{AlmateariMartinRowell}  
we thus have the irreducible factors in all cases. 
Indeed the Bratelli diagram of the tower of semisimple quotients (a convenient summary of the irreducibles and their dimensions) is the usual {\em (left-)truncated} Pascal triangle.  

\medskip

\mdef 
Again, 
$$
X= \left[ \begin {array}{cccc} 1&p-q&q-p&- \left( q-p \right) ^{2}
\\ \noalign{\medskip}0&1&0&q-p\\ \noalign{\medskip}0&0&1&p-q
\\ \noalign{\medskip}0&0&0&1\end {array} \right]
$$ 
is not diagonalizable for $q\neq p$ 
(cf. \S\ref{ss:imageX}), 
so the image is infinite and non-semisimple. 

\mdef \label{prop: f-glue rep} {\bf{Proposition}}. 
The $\MD_n$ representations coming from the f-glue variety of MD representations are indecomposable, {except when $q=p$}. 

\medskip 

\noindent 
{\em{Proof}}. 
For $n=3$, by direct calculation we find that the centraliser 
$\AAAA'_3$
of $ \AAAA_3 = \langle R_1,S_1,R_2,S_2\rangle$
consists,
irrespective of the values of $q\neq p$, 
of matrices of the form:
\[A=\left[\begin{array}{cccc|cccc}
f  & e  & e  & a  & e  & a  & a  & b  
\\
 0 & f  & 0 & c  & 0 & c  & c -e  & d  
\\
 0 & 0 & f  & c  & 0 & c -e  & c  & d  
\\
 0 & 0 & 0 & f  & 0 & 0 & 0 & e  
\\ \hline 
 0 & 0 & 0 & c -e  & f  & c  & c  & d  
\\
 0 & 0 & 0 & 0 & 0 & f  & 0 & e  
\\
 0 & 0 & 0 & 0 & 0 & 0 & f  & e  
\\
 0 & 0 & 0 & 0 & 0 & 0 & 0 & f  
\end{array}\right]
,\]
which is never 
at any point in the variety
a non-trivial idempotent: this can be see by solving $A^2=A$. Note that the $(1,1)$ entry of $A^2=A$ implies $f^2=f$, and the $(5,4)$ entry implies $(c-e)=2f(c-e)$, so that $(c-e)=0$ hence $A$ is either full rank or nilpotent idempotents, i.e., trivial).

The general $n\geq 3$ case is as follows. 
Write $\AAAA_n=\langle R_1,S_1,\ldots,R_{n-1},S_{n-1}\rangle \subset  \Mat(2^n,2^n)$ 
for our image algebra. 
We will prove below that any $A\in\Mat(2^n,2^n)$ with $A\in \AAAA'_n$ 
has the following $2\times 2$ block structure:
\(A=\mat D & B\\ C&D' \tam\) where 
\begin{enumerate}
    \item $D$ and $D'$ are upper triangular with the same constant diagonal, and
    \item the only possibly non-zero entry in $C$ is the upper right entry.
\end{enumerate} 
To see that there are no non-trivial idempotents of this form first observe that $A$  has constant diagonal, and the only potentially non-zero entry below the diagonal is in the $(2^{n-1}+1, 2^{n-1})$ position.  Let us define $\alpha$ to be the diagonal entry of $A$ and $\gamma$ the entry in the $(2^{n-1}+1, 2^{n-1})$ position. Now if $A$ is an idempotent $A^2=A$, the $(1,1)$ entry yields $\alpha^2=\alpha$ and the $(2^{n-1}+1, 2^{n-1})$ entry yields $2\alpha\gamma=\gamma$.  Combining, we find $\gamma=0$.  But then $A=\alpha\cdot Id+N$ where $N$ is nilpotent.  So if $\alpha=0$, $A$ is nilpotent and an idempotent, hence $A=0$.  If $\alpha=1$, $A$ has full rank, so $A=Id$.

\smallskip
To prove  \ppmm{every} $A$ has the above form, we proceed by induction on the rank $n$.  
The base case $n=3$ is as above, and indeed has the form described. Suppose the statement holds for all ranks $3\leq k\leq n-1$, and consider rank $n$.  Since $A$ commutes with $\AAAA_n$ it commutes with the subalgebra $I\otimes \AAAA_{n-1}$. By the induction hypothesis 
this gives us the following $4\times 4$ block structure with $m=2^{n-2}$: 
\[A=\begin{bmatrix}
x\,I_m + N_1 & U_1            & y\,I_m + N_2 & U_2 \\
C_1          & x\,I_m + N_3   & C_2          & y\,I_m + N_4 \\
z\,I_m + N_5 & U_3            & w\,I_m + N_6 & U_4 \\
C_3          & z\,I_m + N_7   & C_4          & w\,I_m + N_8
\end{bmatrix}\] where: $x,y,z,w\in\C$, the $N_i$ are strictly upper triangular, the $C_i$ are $0$ except possibly for the  $(1,m)$ entry, and the $U_i$ are arbitrary.  The block structure of $R_1$ is 
\[R_1=\begin{bmatrix}
I_m & q\,I_m & -q\,I_m & -q^{2} I_m \\
0   & 0      & I_m     & q\,I_m      \\
0   & I_m    & 0       & -q\,I_m     \\
0   & 0      & 0       & I_m
\end{bmatrix}\] and similarly for $S_1$, with $q$ replaced by $p$.  Next we use $[R_1,A]=[S_1,A]=0$ to deduce further structure.  Explicitly, we have, for $R_1$ (entries in \textbf{bold} are used):

\begin{enumerate}
\item[\textbf{(1,1)}] \(q\,C_1 - q\,(z I_m+N_5) - q^{2} C_3 = 0.\)
\item[(1,2)] \(q\,(x I_m+N_1) + (y I_m+N_2) - U_1 - q\,(x I_m+N_3) + q\,U_3 + q^{2}(z I_m+N_7) = 0.\)
\item[(1,3)] \(-\,q\,(x I_m+N_1) + U_1 - (y I_m+N_2) - q\,C_2 + q\,(w I_m+N_6) + q^{2} C_4 = 0.\)
\item[(1,4)] \(-\,q^{2}(x I_m+N_1) + q\,U_1 - q\,(y I_m+N_2) - q\,(y I_m+N_4) + q\,U_4 + q^{2}(w I_m+N_8) = 0.\)

\item[(2,1)] \(C_1 - (z I_m+N_5) - q\,C_3 = 0.\)
\item[\textbf{(2,2)}] \(q\,C_1 + C_2 - U_3 - q\,(z I_m+N_7) = 0.\)
\item[(2,3)] \(-\,q\,C_1 + (x I_m+N_3) - (w I_m+N_6) - q\,C_4 = 0.\)
\item[(2,4)] \(-\,q^{2} C_1 + q\,(x I_m+N_3) - q\,C_2 + (y I_m+N_4) - U_4 - q\,(w I_m+N_8) = 0.\)

\item[(3,1)] \(z I_m+N_5 - C_1 + q\,C_3 = 0.\)
\item[\textbf{(3,2)}] \(q\,(z I_m+N_5) + (w I_m+N_6) - (x I_m+N_3) + q\,(z I_m+N_7) = 0.\)
\item[(3,3)] \(-\,q\,(z I_m+N_5) + U_3 - C_2 + q\,C_4 = 0.\)
\item[(3,4)] \(-\,q^{2}(z I_m+N_5) + q\,U_3 - q\,(w I_m+N_6) + U_4 - (y I_m+N_4) + q\,(w I_m+N_8) = 0.\)

\item[(4,1)] \(C_3 - C_3 = 0.\)  
\item[(4,2)] \(q\,C_3 + C_4 - (z I_m+N_7) = 0.\)
\item[\textbf{(4,3)}] \(-\,q\,C_3 + (z I_m+N_7) - C_4 = 0.\)
\item[(4,4)] \(-\,q^{2} C_3 + q\,(z I_m+N_7) - q\,C_4 = 0.\)
\end{enumerate}

and similarly for $S_1$, replacing $q$ by $p$.

\medskip

The $(1,1)$ entries give us: \[
    -\,q\,C_1 \;+\; q\,(zI_m+N_5) \;+\; q^2\,C_3 \;=\; 0, \qquad -\,p\,C_1 \;+\; p\,(zI_m+N_5) \;+\; p^2\,C_3 \;=\; 0.
\]
Recall that $q\neq p$ and neither is $0$. This implies that $z=0$ and the only possibly non-zero entry of $N_5$ is the upper right entry, say $n_5$.  Denote the corresponding entries of $C_1$ and $C_3$ by $c_1$ and $c_3$. Since $q\neq p$ we conclude $c_3=0$ so $\boxed{C_3=0}$ and $c_1=n_5$ i.e., $\boxed{N_5=C_1}$.

Using the above on the $(2,2)$ entry we get:
\[
q\,C_1 + C_2 - U_3 - q\,N_7 = 0,\qquad
p\,C_1 + C_2 - U_3 - p\,N_7 = 0.
\] Subtracting we conclude $\boxed{C_1=N_7}$ and hence $\boxed{U_3=C_2}$.    The $(3,2)$ entry yields (since $z=0$):
\[
qN_5 \;+\; (wI_m+N_6) \;-\; (xI_m+N_3) \;+\; qN_7 \;=\; 0.
\]
Using $N_5=C_1$, this becomes
\[
(wI_m+N_6) \;-\; (xI_m+N_1) \;+\; 2q\,C_1 \;=\; 0,
\] and similarly with $q\mapsto p$. Subtract:
\[
2(q-p)\,C_1=0 \quad\Rightarrow\quad \boxed{C_1=N_5=N_7=0}\quad(\text{since }q\neq p).
\]
Plug back in:
\[
(w-x)I_m+(N_6-N_1)=0.
\]
Here $N_1,N_6$ are strictly upper-triangular, whence:
\(
\boxed{w=x}\) and \(\boxed{N_6=N_1}.
\)
 The (4,3) entry \[-\,q\,C_3 + (z I_m+N_7) - C_4 = 0\] yields $\boxed{C_4=0}$ and so we conclude:

\[A=\begin{bmatrix}
x\,I_m + N_1 & U_1            & y\,I_m + N_2 & U_2 \\
0         & x\,I_m + N_3   & C_2          & y\,I_m + N_4 \\
0 & C_2           &x\,I_m + N_6 & U_4 \\
0          & 0   & 0         & x\,I_m + N_8
\end{bmatrix}\]

 In particular $A$ has the required form and the induction is complete. \qed

\medskip
\subsection{The MD representation in the a-glue case}



Here we 
consider the local representations $\rho_n$ of $\MD_n$
as in Theorem \ref{thm: big theorem} case \eqref{caseaglue} associated with
\[
S= \left[ \begin {array}{cccc} 1&0&0&p\\ \noalign{\medskip}0&0&1&0
\\ \noalign{\medskip}0&1&0&0\\ \noalign{\medskip}0&0&0&-1\end {array}
 \right],   \;\;
 R= \left[ \begin {array}{cccc} 1&0&0&q\\ \noalign{\medskip}0&0&1&0
\\ \noalign{\medskip}0&1&0&0\\ \noalign{\medskip}0&0&0&-1\end {array}
 \right]. 
 \] 
As usual it will be convenient to linearise - to consider $\rho_n$ as a representation of  the group algebra, and the algebra image, rather than just the group.  

The clue for the analysis is in the case name. 
This case is glue in the classification sense, but at $n=2$ it 
clearly also (generically) has glue in the Jacobson radical sense. 
At the points $p=q$ 
(a closed and hence non-generic condition)
then it is Wangian, as in Prop.\ref{pr:Wangian}, and hence of course semisimple (it is also a representation of the $\Sym_2$ quotient; and continuous deformation through a path of such reps cannot change the structure, so we can determine the structure, on the nose, from the $p=q=0$ case), with a structure that we will recall for all $n$ shortly. 
When $p \neq q$ observe that $S-R$ is in the radical, and the quotient by the radical is Wangian - indeed having the same structure as the Wangian case. 

Altogether for $n=2$ we have, in Loewy (socle series) notation: 
\[
\rho_2 \cong L_{2,0} \oplus L_{1^2} \oplus { L_{2,0} \atop L_{1^2}}
\]
-where we use the $\Sym_2$ partition labels for the irreducibles. 

\mdef  \label{lem:pascalc}
So, back to the Wangian case:
For general $n$ this is, 
up to isomorphism as above, 
the representation we get from Sutherland's $gl_q(1|1)$ spin chain 
\cite{Sutherland75} 
in the classical case (in the terminology of \cite{MartinRittenberg92}). That is, it is a sum  of all the two-row $(+-)$-signed Young modules. 
Thus the structure is a direct sum of two copies of each of the $\Sym_n$  irreps with a hook Young diagram label (see e.g. \cite{MartinRittenberg92,DeguchiMartin92}).  For example 
\beq  \label{eq:rho4-a} 
\rho_4 \cong  (L_{4,0} \oplus L_{3,1} \oplus L_{2,1^2} \oplus L_{1^3} )^{\oplus 2}
  = ( 1\oplus 3 \oplus 3' \oplus 1'  )^{\oplus 2}
\eq 
-where in the latter formulation we give the dimensions. 
Indeed the  dimensions are given by the appropriate row $n$ of Pascal's triangle. 

\mdef  {\bf{Proposition}}.    
The Pascal  
combinatoric as in (\ref{lem:pascalc}) 
also gives the structure of the semisimple quotient in the generic parameter case. 
 \medskip \\ 
{\em Proof}.  This follows immediately from  the higher-on-glue theorem \cite{AlmateariMartinRowell}.  
\qed 

\medskip 

So the next question is how the simple factors are glued together.

\mdef  
 Observe here, from \S\ref{ss:imageX}, that 
\[
RS = X = \left[ \begin {array}{cccc} 
1&0&0&p-q\\ \noalign{\medskip}
0&1&0&0
\\ \noalign{\medskip}
0&0&1&0\\ \noalign{\medskip}
0&0&0&1\end {array}
 \right]
\] 
which is not diagonalisable and has infinite order if $p\neq q$. 
Thus for all $n$ 
if $p\neq q$ the image of $MD_n$ is infinite, and non-semisimple since $X\otimes I^{\otimes (n-2)}$ is also not diagonalisable.

The $MD_3$ representation decomposes as a direct sum of $2$ 4-dimensional indecomposable representations, 
let us call them $\mu$ and $\nu$ respectively, 
with images for $s_1,r_1,s_2,r_2$:
  \[  \left[ \begin {array}{cccc} 1&0&0&p\\ \noalign{\medskip}0&0&1&0
\\ \noalign{\medskip}0&1&0&0\\ \noalign{\medskip}0&0&0&-1\end {array}
 \right], 
 \left[ \begin {array}{cccc} 1&0&0&q\\ \noalign{\medskip}0&0&1&0
\\ \noalign{\medskip}0&1&0&0\\ \noalign{\medskip}0&0&0&-1\end {array}
 \right], 
 \left[ \begin {array}{cccc} 1&p&0&0\\ \noalign{\medskip}0&-1&0&0
\\ \noalign{\medskip}0&0&0&1\\ \noalign{\medskip}0&0&1&0\end {array}
 \right], 
 \left[ \begin {array}{cccc} 1&q&0&0\\ \noalign{\medskip}0&-1&0&0
\\ \noalign{\medskip}0&0&0&1\\ \noalign{\medskip}0&0&1&0\end {array}
 \right] \]
and 
\[ \left[ \begin {array}{cccc} 1&0&0&p\\ \noalign{\medskip}0&0&1&0
\\ \noalign{\medskip}0&1&0&0\\ \noalign{\medskip}0&0&0&-1\end {array}
 \right], 
 \left[ \begin {array}{cccc} 1&0&0&q\\ \noalign{\medskip}0&0&1&0
\\ \noalign{\medskip}0&1&0&0\\ \noalign{\medskip}0&0&0&-1\end {array}
 \right] ,
 \left[ \begin {array}{cccc} 0&1&0&0\\ \noalign{\medskip}1&0&0&0
\\ \noalign{\medskip}0&0&1&p\\ \noalign{\medskip}0&0&0&-1\end {array}
 \right] ,
 \left[ \begin {array}{cccc} 0&1&0&0\\ \noalign{\medskip}1&0&0&0
\\ \noalign{\medskip}0&0&1&q\\ \noalign{\medskip}0&0&0&-1\end {array}
 \right] 
\]
\medskip 

\newcommand{\nuu}{\mu'}

It is immediately clear (from the first column) that the first rep, $\mu$, contains the trivial rep as a factor. 
Observe that if we quotient by this then we have a rep $\nuu$ with $\nuu(r_1)=\nuu(s_1)$ 
and $\nuu(r_2) = \nuu(s_2)$ - i.e. it is Wangian, in the sense of \ref{de:Wangian}. 
This is therefore also a  rep of $\Sym_3$ and has the same structure for either group. It is the rep $(2,1) \oplus (1^3)$ for $\Sym_3$. 
So altogether we have, as a $\MD_3$-module, the non-split structure 
\[
0 \rightarrow (3) \rightarrow \mu \rightarrow (2,1)\oplus (1^3) \rightarrow 0
\]
(or reverse the ses direction, depending on taste for left or right resolution) 
where the 4th term, and indeed the second term, is understood as Wangian, 
hence using the $\Sym_n$ label, but the whole module is of course not Wangian. 

\medskip 

The above analysis is immediate by sight. But not every such problem yields so easily. 
In (\ref{exa:a-n3-magma}) we address the same analysis using computer algebra, and learn a bit more about this indecomposable - in particular that the trivial is glued over {\em both} of the other factors. 

\medskip 

For the second indecomposable, it is immediately clear that the alternating Wangian rep is a 1d factor. And again the quotient is Wangian. 
It is the defining rep, as a $\Sym_3$ rep, with content $(2,1)\oplus (3)$.  Thus here we have 
the non-split structure 
\[
0 \rightarrow (1^3) \rightarrow \nu \rightarrow (2,1)\oplus (3) \rightarrow 0
\]
- to be understood in the same way as above. 

\mdef  \label{exa:a-n3-magma}
Consider  the summand  of $\rho_3$  
given by the rep $\mu$ (in the first row above). 
(\ppmm{By a Magma calculation})
The corresponding subalgebra of $\Mat(4,4)$ has dimension $9$, with Jacobson radical of dimension $3$.  The ss quotient has centre of dimension $3$, so is a sum of $3$ simples, of dimension $2,1,1$ as modules, as we already saw.   

In case $n=4$ then $\rho_4$ is a direct sum of two 8d indecomposables. 
We already know the total semisimple part by \eqref{eq:rho4-a}, 
and it turns out that each 8d rep has semisimple part of the form 
$1\oplus 3 \oplus 3' \oplus 1'$, thus with algebra dimension 20. 
(\ppmm{From Magma}) The full subalgebra has dimension 35, so the radical has dimension 15.  

With the insights gained above, we have the following:

\medskip
\mdef \label{prop:a-glue decomp} \textbf{Proposition}
For $n\geq 3$ the representation $\rho_n$ decomposes as the direct sum of two indecomposable representations of dimension $2^{n-1}$.
\medskip

\noindent 
{\em{Proof}}.  Define $D_1=\mat a&0\\0&b\tam$. For any matrix $M$ with entries in the function field $\C(a,b)$ define $M'$ to be the image of $M$ under the automorphism defined by $a\leftrightarrow b$. Thus, for example, $D_1'=\mat b&0\\0&a\tam$.  For $n\geq 2$ define $D_n=\mat D_{n-1}&0\\0&D_{n-1}'\tam$.
For $n\geq 3$ we can use induction to prove that:\begin{enumerate}
    \item if $T\in\AAAA_n'$ then there exists $N_1$ and $N_2$ strictly upper triangular and $B\in \AAAA'_{n-1}$ such that  \[T=\mat D_{n-1}+N_1 & B\\ 0& D'_{n-1}+N_2 \tam.\]
    \item \(\mat D_{n-1} & 0\\ 0&D'_{n-1} \tam\in \AAAA'_n.\)
\end{enumerate}
Direct calculation shows that any $T\in\AAAA'_3$ has the form: \[ \left[ \begin{array}{cccc|cccc} a&c&c&0&c&0&0&0\\ 0
&b&0&-d&0&-d&0&0\\ 0&0&b&d&0&0&-d&0
\\ 0&0&0&a&0&0&0&c\\ 
\hline0&0&0&0&b&d&d
&0\\ 0&0&0&0&0&a&0&-c\\ 0&0&0&0&0&0
&a&c\\ 0&0&0&0&0&0&0&b\end{array} \right] 
,\] so (both) statements hold for $n=3$.  
The induction proceeds as in the proof of Proposition \ref{prop: f-glue rep}.  For (1): if $T$ commutes with $I\otimes \AAAA_{n-1}\subset\AAAA_n$ then $T$ has a $2\times 2$ block structure with each block of the given form, by hypothesis.  Then commutation with $R_1$ and $S_1$ yields the desired form.  

For (2): the given diagonal matrix commutes with $I\otimes \AAAA_{n-1}$ by hypothesis, and by $R_1,S_1$ by direct verification of the block $4\times 4$ matrices.

From (1) we conclude that the representation $\rho_n$ decomposes into \emph{at most} two indecomposable summands, since any matrix in $\AAAA'_n$ has at most 2 distinct eigenvalues. From (2) we find it decomposes into \emph{at least} 2 indecomposable summands, each of dimension $2^{n-1}$ by observing that the eigenvalues $a,b$ each appear with multiplicity $2^{n-1}$. 
\qed
\medskip

\subsection{The MD representation in the antislash case}

The following semisimple case   
provides an example of unitary representations with infinite image.
Consider \[[R,S]=\left[\left[ \begin {array}{cccc} 0&0&0&1\\ \noalign{\medskip}0&1&0&0
\\ \noalign{\medskip}0&0&1&0\\ \noalign{\medskip}1&0&0&0
\end {array} \right],\left[ \begin {array}{cccc} z&-x&x&-z+1\\ \noalign{\medskip}x&-z+1&z&
-x\\ \noalign{\medskip}-x&z&-z+1&x\\ \noalign{\medskip}-z+1&x&-x&z
\end {array} \right]\right]\]  the anti-slash solution \eqref{caseantislash} of Theorem \ref{thm: big theorem} subcase (a) with $\varepsilon=-1$,
where $x^2 - z^2 + z=0$.  As above, we consider the corresponding representations of $\MD_n$.

\mdef As $X=RS$ is diagonalisable for all $x,z$ (see \S\ref{ss:imageX}), the representation is semisimple at every point.  For $S$ (and hence the representation of $\MD_n$) to be unitary we must have $S^{\dag}=S$, since $S$ is involutive.  This implies that $x=r\cdot \ii$ for some real number $r$, and that $z$ must be real.  Thus, since $z^2-z+r^2=0$, we further must constrain $-\frac{1}{2}\leq r\leq \frac{1}{2}$.  This gives spectrum of $X$  $[1,1,{ {e}\sqrt {1-4\,{r}^{2}}}\pm 2r\cdot \ii]$ where $e=\pm 1$ is the sign coming from the choice of $z$ for a fixed $r$.  Clearly these can be of  infinite order, but are finite for a countable dense subset of choices of $r$.

\mdef Again, let $\AAAA_n$ denote the algebra generated by the image of the corresponding representation of $\MD_n$. A full description of $\AAAA'_n$ is not straightforward to compute due to the form of $S$. We will consider the corresponding representation of $\MD_3$.  By direct computation we find that any $T\in\AAAA'_3$ has the form \[ \left[ \begin {array}{cccccccc} a&b&b&a+d-f&b&a+d-f&a+d-f&b+c-e
\\ \noalign{\medskip}e&f&d&e&d&e&c&d\\ \noalign{\medskip}e&d&f&e&d&c&e
&d\\ \noalign{\medskip}a+d-f&b&b&a&b+c-e&a+d-f&a+d-f&b
\\ \noalign{\medskip}e&d&d&c&f&e&e&d\\ \noalign{\medskip}a+d-f&b&b+c-e
&a+d-f&b&a&a+d-f&b\\ \noalign{\medskip}a+d-f&b+c-e&b&a+d-f&b&a+d-f&a&b
\\ \noalign{\medskip}c&d&d&e&d&e&e&f\end {array} \right] 
.\] Such matrices have at most 4 distinct eigenvalues, which implies that any decomposition has at most 4 irreducible factors.
An explicit example with exactly 4 eigenvalues is: \[Y= \left[ \begin {array}{cccccccc} 1&0&0&1&0&1&1&-2\\ \noalign{\medskip}
1&1&1&1&1&1&-1&1\\ \noalign{\medskip}1&1&1&1&1&-1&1&1
\\ \noalign{\medskip}1&0&0&1&-2&1&1&0\\ \noalign{\medskip}1&1&1&-1&1&1
&1&1\\ \noalign{\medskip}1&0&-2&1&0&1&1&0\\ \noalign{\medskip}1&-2&0&1
&0&1&1&0\\ \noalign{\medskip}-1&1&1&1&1&1&1&1\end {array} \right].
\] The eigenspace decomposition of $Y$ yields a decomposition of the representation of $\MD_3$ into $4$ \emph{irreducible} representations of dimensions $1$ (twice) and $3$ (twice).  The $1$-dimensional representations are trivial, and the $3$ dimensional representations are isomorphic copies of the same irreducible representation
with $X_{ij}$ eigenvalues $1,\lambda=-2z + 1 + 2x,$ and $\lambda^{-1}=-2z + 1 - 2x.$

  \mdef Note that the eigenvalues of $X$ are of the form $[1,\lambda,1/\lambda,1]$, which suggests there is a relationship with the  solution in Theorem \ref{thm: big theorem} case \eqref{caseFlip}: \[[R',S']=\left[ \left[ \begin {array}{cccc} 1&0&0&0
\\ \noalign{\medskip}0&0&1&0\\ \noalign{\medskip}0&1&0&0
\\ \noalign{\medskip}0&0&0&1\end {array} \right],\left[ \begin {array}{cccc} 1&0&0&0\\ \noalign{\medskip}0&0&\lambda&0
\\ \noalign{\medskip}0&\frac{1}{\lambda}&0&0\\ \noalign{\medskip}0&0&0&1
\end {array} \right]   \right].\]
Indeed, we find that these two representations are equivalent, with $\lambda$ as above.  It is clear that they are not locally equivalent: any matrix of the form $A\otimes A$ commutes with $R'$, so no local gauge choice can move $R'$.  Similarly, they are not DS-equivalent: there are no invertible matrices $T$ such that $I\otimes T$ conjugates the pair $(R',S')$ to $(R,S)$. On the other hand a combination of DS and local equivalences does yield an equivalence.  We find that $T\otimes T$ where $T= \left[ \begin {array}{cc} 0&1\\ \noalign{\medskip}1&0\end {array}
 \right]$ commutes with both $R$ and $S$.  Now take $A=\left[ \begin {array}{cc} 1&1\\ \noalign{\medskip}-1&1\end {array}
 \right]$ and define $M=(A\otimes A)(I\otimes T)= \left[ \begin {array}{cccc} 1&1&1&1\\ \noalign{\medskip}1&-1&1&-1
\\ \noalign{\medskip}-1&-1&1&1\\ \noalign{\medskip}-1&1&1&-1
\end {array} \right].$  Then $MRM^{-1}=R'$ and $MSM^{-1}=S'$, which is a composition of $\infty$-equivalences, the result follows.

\mdef A similar approach show that case (b) with $\epsilon=-1$  $S= \left[ \begin {array}{cccc} r&-y-1&y&-r\\ \noalign{\medskip}y&-r&r&-y
-1\\ \noalign{\medskip}-y-1&r&-r&y\\ \noalign{\medskip}-r&y&-y-1&r
\end {array} \right] $ is  $\infty$-equivalent to case \eqref{casefslash}
\[ [R',S']=\left[\left[ \begin {array}{cccc} 1&0&0&0\\ \noalign{\medskip}0&0&p&0
\\ \noalign{\medskip}0&1/p&0&0\\ \noalign{\medskip}0&0&0&1
\end {array} \right], 
 \left[ \begin {array}{cccc} 1&0&0&0\\ \noalign{\medskip}0&0& \left( -
2\,y-1-2\,r \right) p&0\\ \noalign{\medskip}0&{\frac {1}{ \left( -2\,y
-1-2\,r \right) p}}&0&0\\ \noalign{\medskip}0&0&0&-1\end {array}
 \right] \right].\]  In this case the same $A$ as above works, with $T= \left[ \begin {array}{cc} p-1&p+1\\ \noalign{\medskip}p+1&p-1
\end {array} \right].
$

\ignore{{

Set $W=\left[ \begin {array}{cc} 1&1\\ \noalign{\medskip}1&-1\end {array}
 \right]\otimes\left[ \begin {array}{cc} 1&-1\\ \noalign{\medskip}1&1
\end {array} \right] 
$ and take $z = -\frac{(\lambda^2 - 2\lambda + 1)}{4\lambda}, x = \frac{(\lambda^2 - 1)}{4\lambda}.$ One easily checks that these satisfy $z^2-z+r^2=0$.  Then we find that $WR'W^{-1}=R$ and $WS'W^{-1}=S$ so that the pairs $(R,S)$ and $R',S')$ are $2$-equivalent.  Decomposing the $\MD_3$ representations yields $3$-equivalence. Yet we have not ruled out $\infty$-equivalence.

On the other hand, one can show that corresponding functors are neither 1) locally equivalent or 2) DS-equivalent.  Indeed, the only invertible $2\times 2$ matrices $A$ such that $A\otimes A$ commutes with both $R$ and $S$ are of the form: $A=\left[ \begin {array}{cc} x&y\\ \noalign{\medskip}y&x\end {array}
 \right].$  But there are no such $A$ that map $(R,S)$ to $(R',S')$.  }}


\section{Discussion} 
Although we have dealt 
\ppmm{exclusively}  
with the rank $2$ setting of $4\times 4$ matrices $(R,S)$ yielding local $\MD_n$ representations, we have seen that these representations have some intriguing features.  From the paraparticles-to-paravortices perspective, having infinite image unitary representations is already a stark difference.  On the other hand, the semi-direct product structure of $\MD'_n\cong\MD_n$ precludes universality in the sense of dense images.  

There are at least two generalisations that our work suggests: firstly, one could move up in dimension--looking at the rank $3$ situation of pairs of involutive $9\times 9$ matrices satisfying the $\LB_n$ relations.  While no classification of Yang-Baxter operators in rank $3$ exists, the involutive solutions 
\ppmm{needed for our approach}
might be accessible.  Secondly, one might consider local $\LB$ representations for which $R^3=I$.  We note that the quotient of the braid groups $B_n$ by the relation $(\sigma_i)^3=1$ is only finite for $n\leq 5$ \cite{Coxeterbanff} in counter distinction with the $(\sigma_i)^2=1$ case, thus we anticipate this to be more complicated.

\ignore{{
\section{Towards Rank 3} 

...

\subsection{Comparison with the complete CC classification for loop braid}

...

\subsection{Towards classification of involutive braid reps in rank-3}

...

...
}}

\vspace{1cm}

\ignore{{
\appendix

\input{tex/linear-rep-thy}

\input{tex/macros-glue}

\input{tex/CCwithglue0}

\input{tex/CCwithgluev2}

}}

\vspace{1cm}

\bibliographystyle{abbrv}
\bibliography{local2,more}

\end{document}

%% file: macros.tex
\newcommand{\ppm}[1]{\textcolor{red}{#1}}
\newcommand{\ppmx}[1]{}
\newcommand{\ppmm}[1]{#1}
\newcommand{\ecr}[1]{\textcolor{orange}{#1}}
\newcommand{\ft}[1]{\textcolor{purple}{#1}}
\newcommand{\ftx}[1]{}
\newcommand{\ii}{\textrm{i}}

\newcommand{\PI}[1]{\Pi^{N}}

\newcommand{\An}{\mathfrak{A}}    

\newcommand{\Bcat}{\mathsf{B}}
\newcommand{\LB}{\mathbf{L}}
\newcommand{\Lcat}{\mathsf{L}}
\newcommand{\Mat}{\mathsf{Mat}} 

\newcommand{\Sym}{\Sigma} 
\newcommand{\N}{\mathbb{N}}   

\newcommand{\ignore}[1]{}

\newtheorem{theorem}{Theorem}[section]
\newtheorem{proposition}[theorem]{Proposition}
\newtheorem{lemma}[theorem]{Lemma}
\newtheorem{corollary}[theorem]{Corollary}
\newtheorem{conjecture}[theorem]{Conjecture}

\theoremstyle{definition}
\newtheorem{defin}[theorem]{Definition}
\newtheorem{example}[theorem]{Example}

\newtheorem{remark}[theorem]{Remark}

\newtheorem{para}[theorem]{}

\newcommand{\beq}{\begin{equation}}
\newcommand{\eq}{\end{equation}} 

\newcommand{\C}{{\mathbb C}}
\newcommand{\Z}{{\mathbb Z}}
\newcommand{\mat}{\begin{bmatrix}}
\newcommand{\tam}{\end{bmatrix}} 
\newcommand{\matt}[1]{\left( \begin{array}{#1}}
\newcommand{\ttam}{\end{array} \right)}

\newcommand{\fff}{{\mathsf{f}}}

\newcommand{\smat}{\left(\begin{smallmatrix}}
\newcommand{\stam}{\end{smallmatrix}\right)}

\newcommand{\footnotex}[1]{}  

\newcounter{minidef}[section]

\newcommand{\mdef}{\refstepcounter{theorem} 
\medskip \noindent ({\bf \thetheorem}) }
\newcounter{minicapt}

\newcommand{\catC}{{\mathcal C}}

\newcommand{\PP}{\mathtt{P}}    

\newcommand{\MD}{{\mathsf{MD}}}  
\newcommand{\MDD}{{\mathsf{MD}'}}  
\newcommand{\VB}{{\mathsf{VB}}} 
\newcommand{\VP}{{\mathsf{VP}}}

%% file: tex/wreaths01.tex
\subsection{Warm-up: wreaths, their conjugacy classes, and their representations}

$\;$ 

\newcommand{\ccc}{cc}  

Notation: Here if $G$ is a group then $\ccc(G)$ denotes the set of conjugacy classes of $G$. 

\mdef  Let $G$ be a group. Recall that the wreath $G\wr \Sym_n$ is the semidirect
product obtained by allowing $\Sym_n$ to act on $n$ copies of $G$ by permuting the order of the copies. 
Obviously elements of $G^n$ take the form of $n$-tuples of elements of $G$. 
Given an element $g\in G$ we write $g_i$ for the element 
$(1,1,...,1,g,1,1,...1) \in G^n$ (the $g$ in the $i$-th position). 
Then the action of $w \in \Sym_n$ is given by $\Psi_w (g_i) = g_{w(i)}$. 

\mdef Recall further that there is a natural `bird-track' calculus (cf. e.g. \cite{Cvitanovic,MartinWoodcockLevy00}) for  $G\wr \Sym_n$ which decorates the strings of the usual diagram calculus for $\Sym_n$ with elements of $G$. 
Writing $\tau \in G$ for a generating element (and $\sigma_1 \in \Sym_n$ as usual), then the calculus looks like this:
\[
\input{figs/levy4-pic}
\]
(all such figures here are taken directly from \cite{MartinWoodcockLevy00}). The calculus allows $\sigma_1 \tau_1 \sigma_1$ to be represented as a bead on the second track:
\beq  \label{eq:beadmove}
\input{figs/levy4-piccy}
\eq 
(diagrammatically this can be seen as sliding the bead up along its track, through the crossing, and then straightening the simple double crossing) whereupon a general element can be drawn in the form:
\[
\input{figs/levy4-piccyy}
\]
where $P$ is any undecorated permutation diagram. 
The `straightening' of a general element into the above format looks like:
\beq  \label{eq:straighten-eg}
\input{figs/levy4-piiccyy}
\eq 
- i.e. it can be done by `bead-sliding' along the track (the first step inserts the identity in the form of a double crossing so that (\ref{eq:beadmove}) can be applied).

\mdef 
Observe that using the diagrams in (\ref{eq:beadmove}-\ref{eq:straighten-eg}) it becomes easy to work out the conjugacy classes of the group $G\wr\Sym_n$ - and hence (at least if $G$ is finite) 
the size of an index set for irreducible representations. 
One consistently `visual' way to do this is with the diagram-trace graph, as illustrated in Fig.\ref{fig:diagramtrace}. 
In particular note that the cycle structure of the permutation factor 
in an element 
when expressed in the canonical form is an invariant of conjugation
(just as for the symmetric group); 
and that the $G$ elements associated to each cycle can be brought together 
(conjugating by a bead corresponds to sliding the bead around the trace loop)
and this product element is also invariant up to conjugacy class in $G$ 
in the sense that $ab \sim babb^{-1} = ba$ (in Figure~\ref{fig:diagramtrace} this is again moving a bead $b$ around the trace loop). 

Observe that we can organise the cycles into groupings with same-conjugacy class of $G$ elements on them, 
so the overall class can be described by a cycle structure for each $G$ conjugacy class - in other words the set of classes is the set of maps from $cc(G)$ to partitions such that the summed degrees of these partitions is $n$. 
(In our example in Figure~\ref{fig:diagramtrace} there are three cycles. If we suppose that $G$ is generated by $\tau$ only (for example $G=\Z/9\Z$) then   each cycle is here carrying a different $G$-class - the classes of $\tau^3$, $\tau^2$ and 1 respectively, so the class overall can be labelled by the function $3\mapsto (3)$, $2\mapsto (2) $ and $0\mapsto (1)$.)  
That is, an index set for conjugacy classes of $G\wr\Sym_n$ is $\hom(\ccc(G),\Lambda)_n$,
where  
$\Lambda$ is the set of all integer partitions; and the sub-$n$ indicates that we restrict to $f \in \hom(\ccc(G),\Lambda) $ such that $\sum_{c\in \ccc(G)} |f(c)| \; = n$. 
That is:
\beq
\ccc(G\wr\Sym_n) \cong \;   \hom(\ccc(G),\Lambda)_n
\eq 
For example if $G=C_2 \cong \Sym_2$, so that $cc(G) = \{ \{1\} , \{ a \}\}$ 
(taking $C_2 = \langle a | a^2=1\rangle$), 
then the index set consists of $\ccc(G)$ indexed pairs of partitions of total degree $n$. 
Again, a very neat way to see this is with the `diagram trace' as in Fig.\ref{fig:diagramtrace}.

\begin{figure}
\[ \hspace{-1cm}
\raisebox{.2in}{
\includegraphics[width=2.693cm]{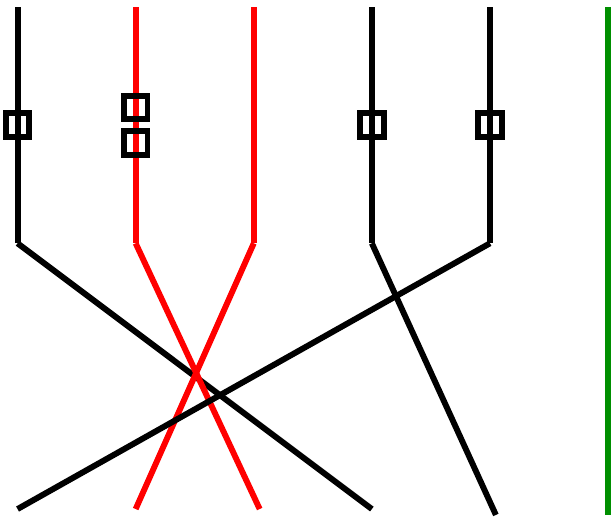}
}
\;\;\;\;\hspace{.1in} \raisebox{.431in}{$\mapsto$} \; \;\;\hspace{.1in} 
\includegraphics[width=5.3cm]{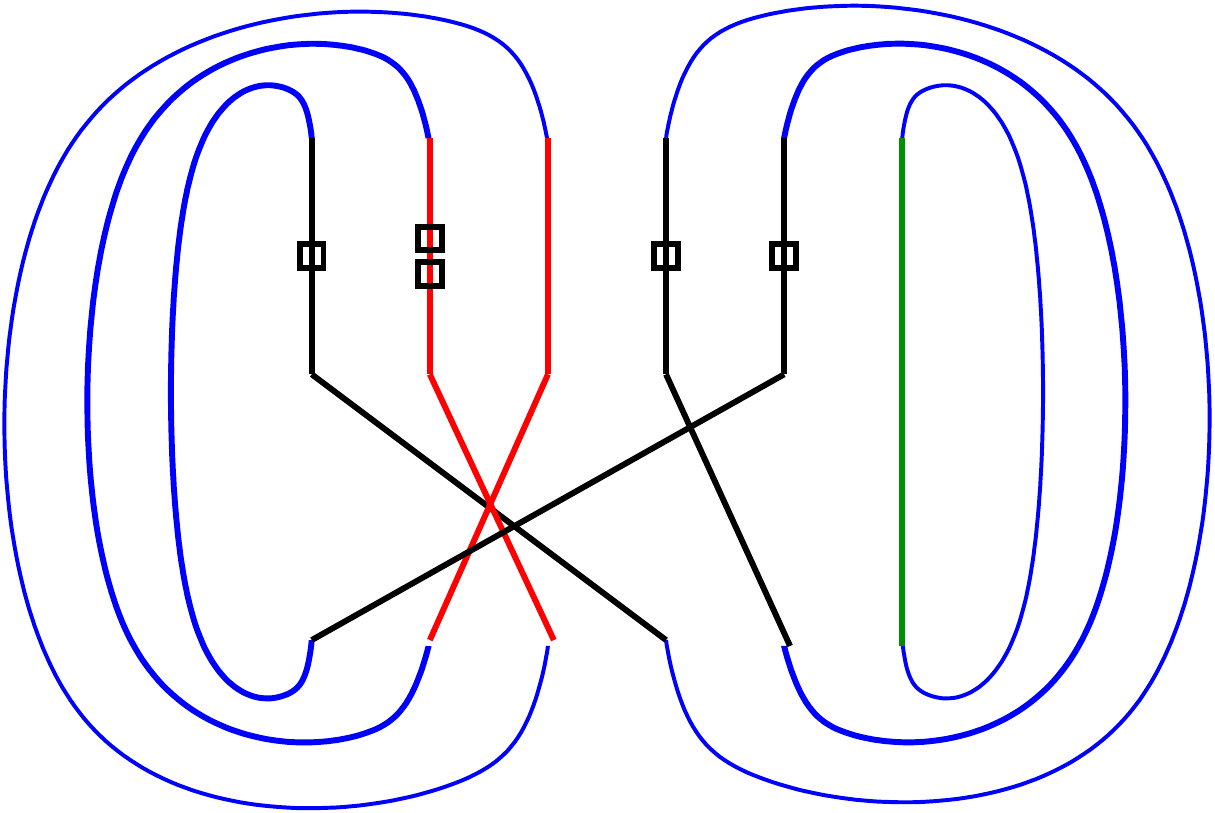}
\;\;\;\;\hspace{.1in} \raisebox{.431in}{$=$} \; \;\;\hspace{.1in} 
\includegraphics[width=5.2cm]{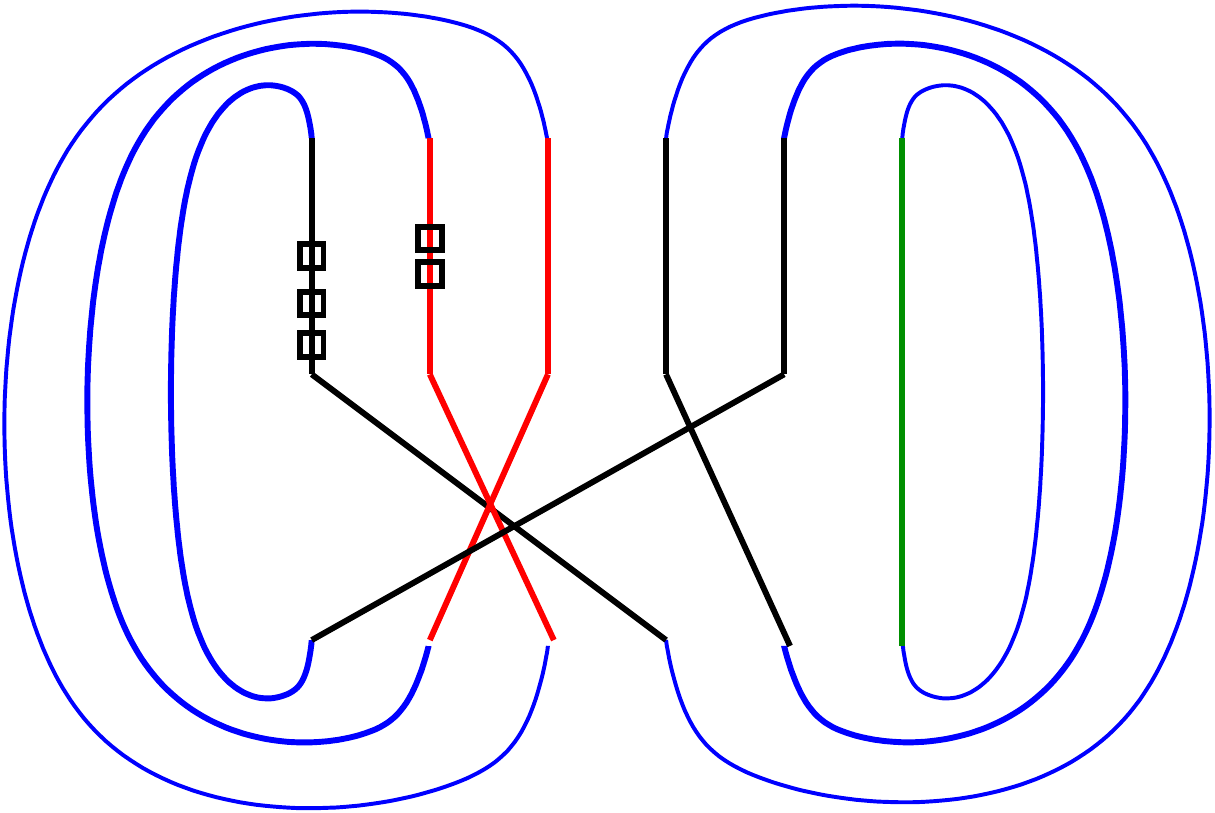}
\]
\caption{Passing from a diagram $d$ 
for an element of $G\wr \Sym_n$
to the diagram-trace graph of the diagram, which shows its conjugacy class. 
Here we have coloured the strands the same colour if they are in the same cycle in the perm in $d$ - the black cycle is the 3-cycle $(145)$ and so on. 
The final identity slides $G$ elements around their cycle until they are together - note that conjugating by, say $d \mapsto \tau_4^{-1} d \tau_4$, cancels the $\tau_4$ from the top and moves it to the bottom (diagrammatically this is as if the $\tau_4$ is  slid around the trace loop). 
Then from the bottom, straightening in the sense of (\ref{eq:straighten-eg}) then moves this factor to the first strand in the 3-cycle. 
\label{fig:diagramtrace} }
\end{figure}
\ignore{{
\mdef 
Representations of $G\wr \Sym_n$ can now be constructed using the outer product representations of $\Sym_n$. This can be seen as an explicit implementation of Clifford theory - a good warm-up for the cases we will need - so we will review briefly here. To keep it brief, we stay here with finite $G$, and indeed let us consider $G=C_d$, the cyclic group of order $d$ for some $d$.  
Recall that an outer product representation ... \ppm{[a bit more to add.]}
}}

Now with these ideas in mind, we are ready to move over to the generalisation that we need. 

\subsection{Generalised wreaths}

$\;$ 

\mdef  
For $n \in \N$
consider the group $\Z^{{n}\choose{2}}$, with generators $x_{ij}$ for $1\leq i<j\leq n$.  
That is,
$\Z^{{n}\choose{2}} 
 = \; \langle x_{ij} , i<j \; | x_{ij} x_{kl}=x_{kl} x_{ij} \rangle
 \; = \; \{ \prod_{i<j} x_{ij}^{a_{ij}} \; | \; a_{ij} \in \Z   \}$.
We may write $g \in  \Z^{{n}\choose{2}} $ as $(x_{12}^{a_{12}},x_{13}^{a_{13}},..)$.
In particular that is $x_{12} = (x_{12}^1,1,1,..,1)$ and $1=(1,1,..,1)$.

\mdef   \label{de:MDprime}
There is an action of the symmetric group $\Sym_n$ on this  
$\Z^{{n}\choose{2}} 
$
by 
$$
\psi_w ( x_{ij} ) = x_{w(i) w(j)} ,
$$
where $x_{ji} = x_{ij}^{-1}$. 
{Observe that this extends to a well-defined action on the whole group.}
We define $\MDD_n$ as the semidirect product   $\Z^{{n}\choose{2}} \rtimes \Sym_n$  with respect to this action. 
Thus as a set 
$\MDD_n$ 
is  $\Z^{{n}\choose{2}} \times \Sym_n$  and the multiplication is  (as for any semidirect product) 
\beq  \label{eq:defsd}
(X_1, w_1) \; (X_2, w_2) \;\; = \;\;  (X_1 \psi_{w_1} (X_2), \; w_1 w_2 ) .
\eq 
This can be visualised analogously to the bird-track calculus for wreaths mentioned above. For example, for $n=6$  we have:
\[
\includegraphics[width=3cm]{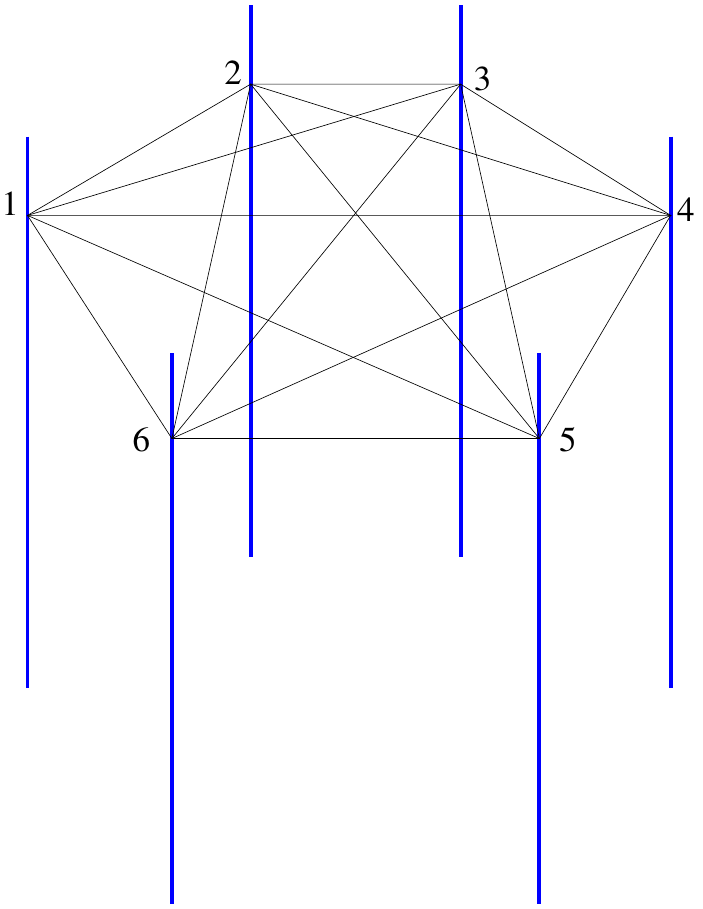}  
\qquad \hspace{2.1cm} 
\includegraphics[width=3cm]{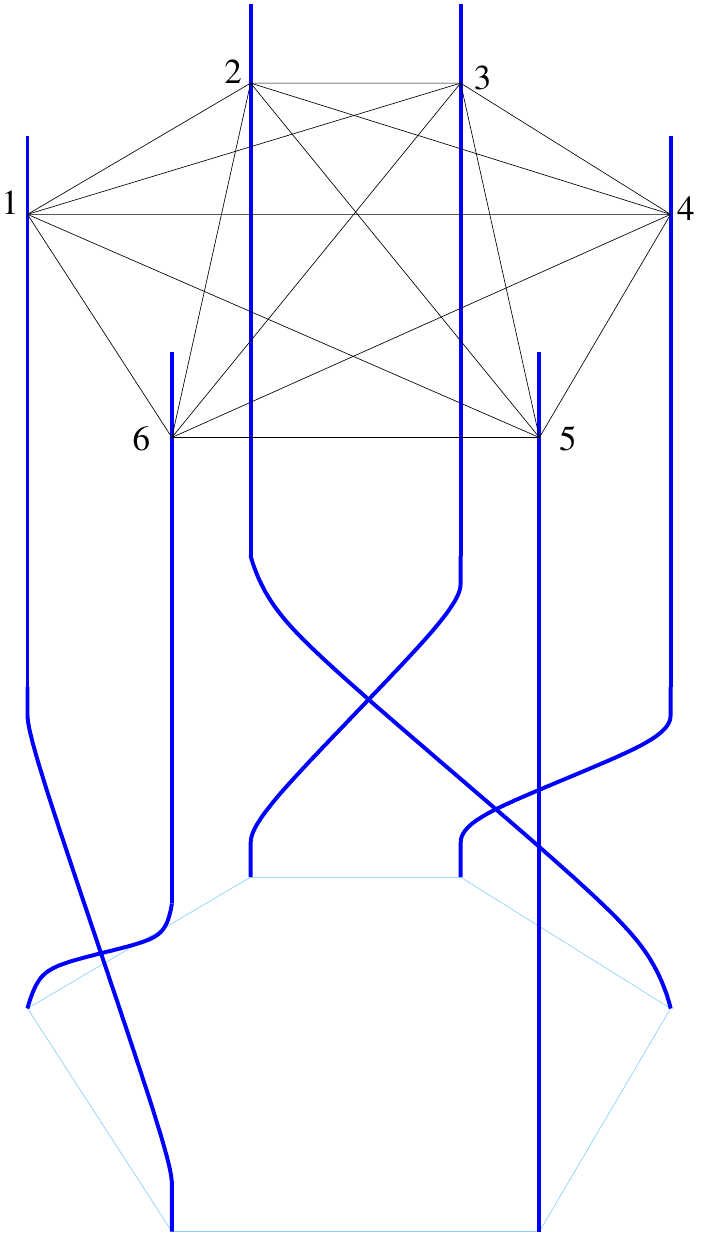}
\]
Here we direct the complete graph $K_n$ by the natural order $i+1 > i$, and so on, on vertices. The directed edge $(i,j)$ (with $i<j$) now carries $x_{ij}$. 
Composition is by stacking and `straightening'. For example locally we have 
\[
\raisebox{1.92cm}{$x_{ij}^n \;\;=\;\;\;$} 
\includegraphics[width=1.352cm]{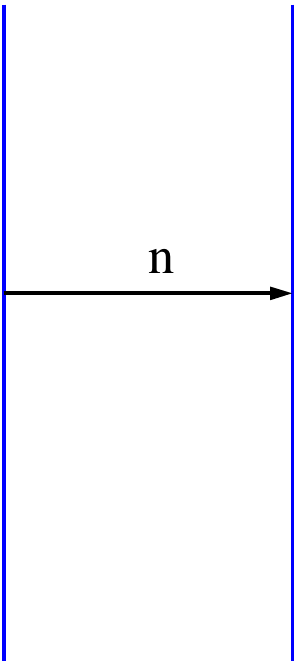}
\qquad \hspace{1.3cm} 
\includegraphics[width=4cm]{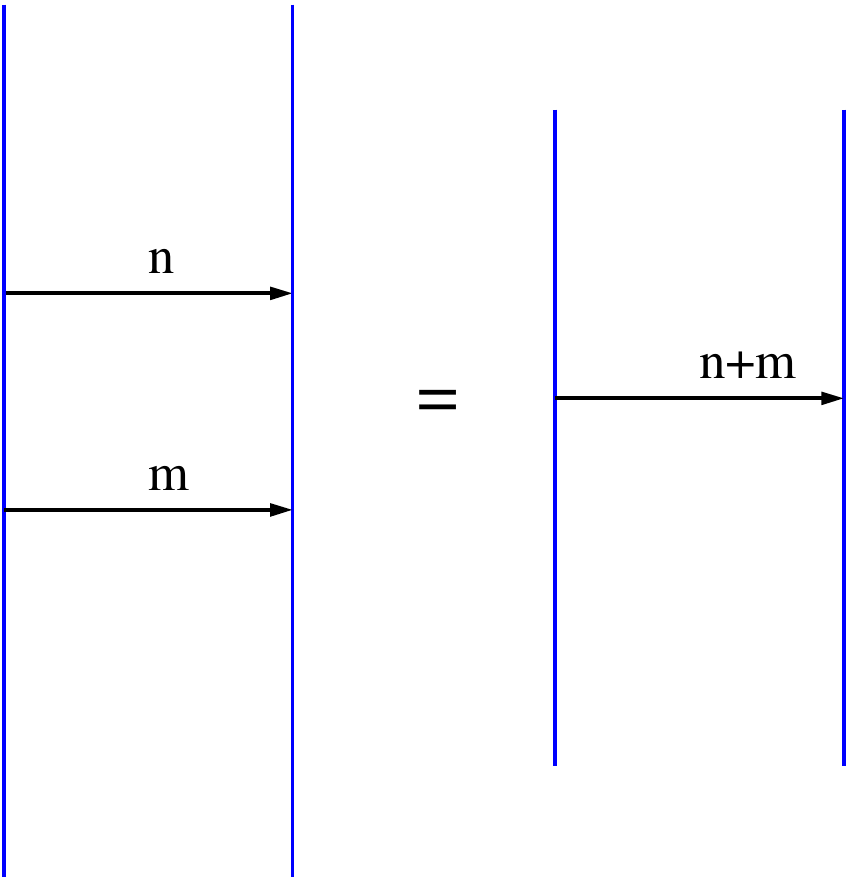}
\hspace{2.2cm}
\includegraphics[width=3.985cm]{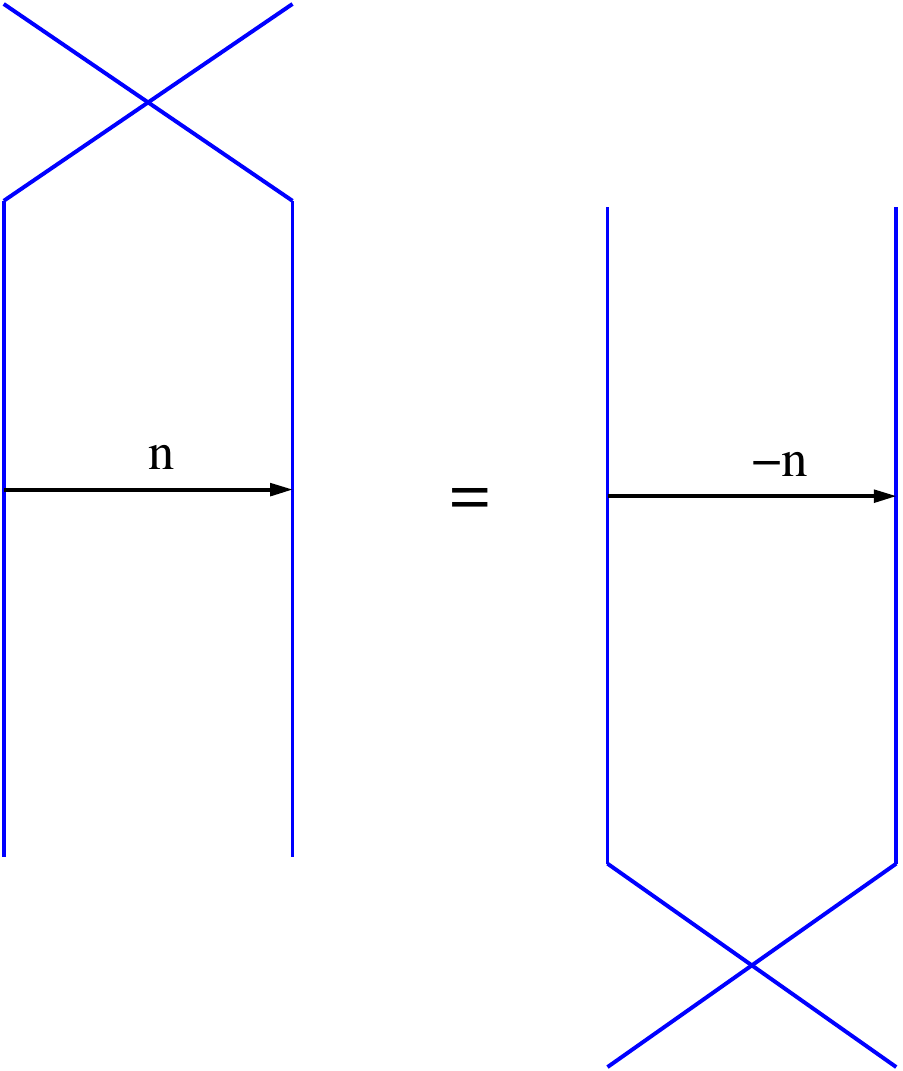}
\]
{where the left-hand side of the  
final equality is the product $w x_{ij}^n$ for $w$ the elementary transposition $(ij)$; and the right-hand side is this straightened into the canonical form 
- note that the reversed arrow is re-reversed by the negation. 
}

\mdef    \label{pres:MD'}
In $\MDD_n$ define $\varsigma_i$ as $(1,\sigma_i) = ((1,1,\ldots,1),\sigma_i)$ 
(where $\sigma_i$ is the elementary transposition in $\Sym_n$)
and $x_{ij} = (x_{ij} , 1)$. 

Observe that  $\MDD_n$ admits a presentation with generators $x_{12}$ and $\varsigma_i$, $i=1,2,\ldots,n-1$ together with the usual symmetric group relations, the relation 
\begin{align}\label{eq:Commute}
    x_{ij}x_{kl}=x_{kl}x_{ij}    
\end{align}
for all allowed $i,j,k,l$ which comes from the copy of $\Z^{{n}\choose{2}}$, and the conjugation relation
\begin{align}\label{eq:Conj}
\varsigma x_{ij}\varsigma^{-1} = x_{\varsigma(i)\varsigma(j)}
\end{align}
for $\varsigma=(1,\varsigma)\in \MDD_n$.

\ignore{
\mdef
More abstractly, 
a set of generating relations for $\MDD_n$ is... 
the standard relations defining $\Sigma_n$, i.e. the braid relations and $\varsigma_i^2=1$ as well as
... $(x_{12}\varsigma_1)^2=1$...\ecr{[probably need to define (somehow) the $x_{ij}$ and impose commutation.]} \ppm{[-why do we need this?]}\ecr{[for the homomorphism?  or so we can verify a rep via relations?]}
\ppm{[right. above we give the complete multiplication table, which is a presentation - albeit not necessarily an efficient one. shall we keep an eye on what we need once @Fiona has started her bit? if we end up -not- unpacking the homo/isomorphism then we'll just possibly need the rep-verify aspect - exactly what we need should hopefully come out in examples.]}}

%% file: figs/levy4-pic.tex
\setlength{\unitlength}{0.00525in}%
\begin{picture}(410,120)(50,330)            
\thicklines
\put(100,420){\line( 0,-1){ 80}}
\put(140,420){\line( 0,-1){ 80}}
\put(180,420){\line( 0,-1){ 80}}
\put(220,420){\line( 0,-1){ 80}}
\put( 95,375){\framebox(10,10){}}
\put( 20,365){\makebox(0,0)[lb]{\raisebox{0pt}[0pt][0pt]{$\tau_1= $}}}
\end{picture}
\setlength{\unitlength}{1647sp}
\begingroup\makeatletter\ifx\SetFigFont\undefined
\def\x#1#2#3#4#5#6#7\relax{\def\x{#1#2#3#4#5#6}}%
\expandafter\x\fmtname xxxxxx\relax \def\y{splain}%
\ifx\x\y   
\gdef\SetFigFont#1#2#3{%
  \ifnum #1<17\tiny\else \ifnum #1<20\small\else
  \ifnum #1<24\normalsize\else \ifnum #1<29\large\else
  \ifnum #1<34\Large\else \ifnum #1<41\LARGE\else
     \huge\fi\fi\fi\fi\fi\fi
  \csname #3\endcsname}%
\else
\gdef\SetFigFont#1#2#3{\begingroup
  \count@#1\relax \ifnum 25<\count@\count@25\fi
  \def\x{\endgroup\@setsize\SetFigFont{#2pt}}%
  \expandafter\x
    \csname \romannumeral\the\count@ pt\expandafter\endcsname
    \csname @\romannumeral\the\count@ pt\endcsname
  \csname #3\endcsname}%
\fi
\fi\endgroup
\begin{picture}(5112,2724)(601,-2173)
\thicklines 
\put(2101,539){\line( 0,-1){600}}
\put(2101,-61){\line( 1,-1){600}}
\put(2701,-661){\line( 0,-1){600}}
\put(2701,-1261){\line(-1,-1){600}}
\put(2101,-1861){\line( 0,-1){300}}
\put(2701,539){\line( 0,-1){600}}
\put(2701,-61){\line(-1,-1){600}}
\put(2101,-661){\line( 0,-1){600}}
\put(2101,-1261){\line( 1,-1){600}}
\put(2701,-1861){\line( 0,-1){300}}
\put(2026,164){\framebox(150,150){}}
\put(2026,-1036){\framebox(150,150){}}
\put(3301,539){\line( 0,-1){2700}}
\put(3901,539){\line( 0,-1){2700}}
\put(201,-811){\makebox(0,0)[lb] 
{\smash{\SetFigFont{10}{14.4}{rm}\hspace{-.21cm}$\tau_1\sigma_1\tau_1\sigma_1 = \;$}}}
\end{picture}

%% file: figs/levy4-piccy.tex
\setlength{\unitlength}{0.0055in}%
\begin{picture}(365,120)(55,660)
\thicklines
\put( 60,780){\line( 1,-1){ 40}}
\put(100,740){\line( 0,-1){ 40}}
\put(100,700){\line(-1,-1){ 40}}
\put(100,780){\line(-1,-1){ 40}}
\put( 60,740){\line( 0,-1){ 40}}
\put( 60,700){\line( 1,-1){ 40}}
\put(140,780){\line( 0,-1){120}}
\put(180,780){\line( 0, 1){  0}}
\put(180,780){\line( 0,-1){120}}
\put(300,780){\line( 0,-1){120}}
\put(340,780){\line( 0,-1){120}}
\put(380,780){\line( 0,-1){120}}
\put(420,780){\line( 0,-1){120}}
\put( 55,725){\line( 1, 0){ 10}}
\put( 65,725){\line( 0,-1){ 10}}
\put( 65,715){\line(-1, 0){ 10}}
\put( 55,715){\line( 0, 1){ 10}}
\put(335,725){\line( 1, 0){ 10}}
\put(345,725){\line( 0,-1){ 10}}
\put(345,715){\line(-1, 0){ 10}}
\put(335,715){\line( 0, 1){ 10}}
\put(240,720){\makebox(0,0)[lb]{\raisebox{0pt}[0pt][0pt]{ =}}}
\end{picture}

%% file: figs/levy4-piccyy.tex
\setlength{\unitlength}{0.0055in}%
\begin{picture}(240,180)(50,590)
\thicklines
\put(240,780){\line( 0,-1){100}}
\put(235,735){\line( 1, 0){ 10}}
\put(245,735){\line( 0,-1){ 10}}
\put(245,725){\line(-1, 0){ 10}}
\put(235,725){\line( 0, 1){ 10}}
\put(280,780){\line( 0,-1){100}}
\put(275,735){\line( 1, 0){ 10}}
\put(285,735){\line( 0,-1){ 10}}
\put(285,725){\line(-1, 0){ 10}}
\put(275,725){\line( 0, 1){ 10}}
\put(120,780){\line( 0,-1){100}}
\put(115,735){\line( 1, 0){ 10}}
\put(125,735){\line( 0,-1){ 10}}
\put(125,725){\line(-1, 0){ 10}}
\put(115,725){\line( 0, 1){ 10}}
\put( 75,735){\line( 1, 0){ 10}}
\put( 85,735){\line( 0,-1){ 10}}
\put( 85,725){\line(-1, 0){ 10}}
\put( 75,725){\line( 0, 1){ 10}}
\put( 75,755){\line( 1, 0){ 10}}
\put( 85,755){\line( 0,-1){ 10}}
\put( 85,745){\line(-1, 0){ 10}}
\put( 75,745){\line( 0, 1){ 10}}
\put(195,760){\line( 1, 0){ 10}}
\put(205,760){\line( 0,-1){ 10}}
\put(205,750){\line(-1, 0){ 10}}
\put(195,750){\line( 0, 1){ 10}}
\put(195,720){\line( 1, 0){ 10}}
\put(205,720){\line( 0,-1){ 10}}
\put(205,710){\line(-1, 0){ 10}}
\put(195,710){\line( 0, 1){ 10}}
\put( 60,680){\line( 1, 0){240}}
\put(300,680){\line( 0,-1){ 60}}
\put(300,620){\line(-1, 0){240}}
\put( 60,620){\line( 0, 1){ 60}}
\put( 80,620){\line( 0,-1){ 40}}
\put(120,620){\line( 0,-1){ 40}}
\put(160,620){\line( 0,-1){ 40}}
\put(200,620){\line( 0,-1){ 40}}
\put(240,620){\line( 0,-1){ 40}}
\put(280,620){\line( 0,-1){ 40}}
\put( 80,780){\line( 0,-1){100}}
\put(200,780){\line( 0,-1){100}}
\put(160,780){\line( 0,-1){100}}
\put(195,740){\line( 1, 0){ 10}}
\put(205,740){\line( 0,-1){ 10}}
\put(205,730){\line(-1, 0){ 10}}
\put(195,730){\line( 0, 1){ 10}}
\put(170,645){\makebox(0,0)[lb]{\raisebox{0pt}[0pt][0pt]{P}}}
\end{picture}

%% file: figs/levy4-piiccyy.tex
\setlength{\unitlength}{0.0045in}%
\begin{picture}(525,360)(75,420)
\thicklines
\put( 80,780){\line( 1,-1){200}}
\put(120,780){\line(-1,-1){ 40}}
\put( 80,740){\line( 0,-1){ 40}}
\put( 80,700){\line( 1,-1){ 40}}
\put(120,660){\line( 0,-1){ 40}}
\put(120,620){\line( 1,-1){ 40}}
\put(160,740){\line(-1,-1){ 80}}
\put( 80,660){\line( 0,-1){120}}
\put( 80,540){\line( 0, 1){  5}}
\put(160,740){\line( 1, 1){ 40}}
\put(160,780){\line( 1,-1){ 40}}
\put(240,780){\line( 0,-1){120}}
\put(240,660){\line(-1,-1){ 40}}
\put(200,620){\line( 0,-1){ 80}}
\put(200,740){\line( 0,-1){ 40}}
\put(200,700){\line(-1,-1){ 40}}
\put(160,660){\line( 0,-1){ 40}}
\put(160,620){\line(-1,-1){ 40}}
\put(120,580){\line( 0,-1){ 40}}
\put(160,580){\line( 0,-1){ 40}}
\put(280,780){\line( 0,-1){160}}
\put(280,620){\line(-1,-1){ 40}}
\put(240,580){\line( 0,-1){ 40}}
\put(240,540){\line( 0, 1){  5}}
\put(280,580){\line( 0,-1){ 40}}
\put( 75,725){\line( 1, 0){ 10}}
\put( 85,725){\line( 0,-1){ 10}}
\put( 85,715){\line(-1, 0){ 10}}
\put( 75,715){\line( 0, 1){ 10}}
\put( 75,645){\line( 1, 0){ 10}}
\put( 85,645){\line( 0,-1){ 10}}
\put( 85,635){\line(-1, 0){ 10}}
\put( 75,635){\line( 0, 1){ 10}}
\put( 75,565){\line( 1, 0){ 10}}
\put( 85,565){\line( 0,-1){ 10}}
\put( 85,555){\line(-1, 0){ 10}}
\put( 75,555){\line( 0, 1){ 10}}
\put(400,780){\line( 1,-1){ 40}}
\put(440,740){\line( 0,-1){ 40}}
\put(440,700){\line(-1,-1){ 40}}
\put(400,660){\line( 1,-1){ 80}}
\put(440,780){\line(-1,-1){ 40}}
\put(400,740){\line( 0,-1){ 40}}
\put(400,700){\line( 1,-1){ 40}}
\put(440,660){\line(-1,-1){ 40}}
\put(400,620){\line( 0,-1){ 40}}
\put(395,725){\line( 1, 0){ 10}}
\put(405,725){\line( 0,-1){ 10}}
\put(405,715){\line(-1, 0){ 10}}
\put(395,715){\line( 0, 1){ 10}}
\put(480,780){\line( 1,-1){ 40}}
\put(520,740){\line( 0,-1){120}}
\put(520,780){\line(-1,-1){ 40}}
\put(480,740){\line( 0,-1){120}}
\put(480,620){\line(-1,-1){ 80}}
\put(400,540){\line( 0,-1){ 80}}
\put(400,580){\line( 1,-1){ 80}}
\put(480,500){\line( 0,-1){ 40}}
\put(480,580){\line( 1,-1){120}}
\put(520,620){\line( 0,-1){ 40}}
\put(520,580){\line(-1,-1){ 80}}
\put(440,500){\line( 0,-1){ 40}}
\put(560,780){\line( 0,-1){240}}
\put(560,540){\line(-1,-1){ 40}}
\put(520,500){\line( 0,-1){ 40}}
\put(600,780){\line( 0,-1){280}}
\put(600,500){\line(-1,-1){ 40}}
\put(395,520){\line( 1, 0){ 10}}
\put(405,520){\line( 0,-1){ 10}}
\put(405,510){\line(-1, 0){ 10}}
\put(395,510){\line( 0, 1){ 10}}
\put(395,500){\line( 1, 0){ 10}}
\put(405,500){\line( 0,-1){ 10}}
\put(405,490){\line(-1, 0){ 10}}
\put(395,490){\line( 0, 1){ 10}}
\put(350,670){\makebox(0,0)[lb]{\raisebox{0pt}[0pt][0pt]{=}}}
\end{picture}
\begin{picture}(525,360)(75, 20)
\thicklines
\put(355,360){\line( 1, 0){ 10}}
\put(365,360){\line( 0,-1){ 10}}
\put(365,350){\line(-1, 0){ 10}}
\put(355,350){\line( 0, 1){ 10}}
\put(435,370){\line( 1, 0){ 10}}
\put(445,370){\line( 0,-1){ 10}}
\put(445,360){\line(-1, 0){ 10}}
\put(435,360){\line( 0, 1){ 10}}
\put(435,350){\line( 1, 0){ 10}}
\put(445,350){\line( 0,-1){ 10}}
\put(445,340){\line(-1, 0){ 10}}
\put(435,340){\line( 0, 1){ 10}}
\put(320,380){\line( 0,-1){ 60}}
\put(360,380){\line( 0,-1){ 60}}
\put(400,380){\line( 0,-1){ 60}}
\put(440,380){\line( 0,-1){ 60}}
\put(480,380){\line( 0,-1){ 60}}
\put(520,380){\line( 0,-1){ 60}}
\put(320,320){\line( 2,-1){200}}
\put(360,320){\line( 2,-5){ 40}}
\put(400,320){\line(-2,-5){ 40}}
\put(440,320){\line(-6,-5){120}}
\put(480,320){\line(-2,-5){ 40}}
\put(520,320){\line(-2,-5){ 40}}
\put(280,305){\makebox(0,0)[lb]{\raisebox{0pt}[0pt][0pt]{= }}}
\put(185,305){\makebox(0,0)[lb]{\raisebox{0pt}[0pt][0pt]{ = .....}}}
\end{picture}

%% file: tex/HietarintaN2.tex

\newcommand{\Hslashglue}{
\left[ \begin {array}{cccc} k&q&p&s\\ \noalign{\medskip}0&0&k&q
\\ \noalign{\medskip}0&k&0&p\\ \noalign{\medskip}0&0&0&k\end {array}
 \right] 
 }
\newcommand{\Hslash}{
 \left[ \begin {array}{cccc} k&0&0&0\\ \noalign{\medskip}0&0&p&0
\\ \noalign{\medskip}0&q&0&0\\ \noalign{\medskip}0&0&0&s\end {array}
 \right] 
 }

\newcommand{\Haglue}{
 \left[ \begin {array}{cccc} p&0&0&k\\ \noalign{\medskip}0&0&p&0
\\ \noalign{\medskip}0&q&p-q&0\\ \noalign{\medskip}0&0&0&-q\end {array}
 \right]
 }
 
\newcommand{\HslashDS}{
 \left[ \begin {array}{cccc} 0&0&0&p\\ \noalign{\medskip}0&k&0&0
\\ \noalign{\medskip}0&0&k&0\\ \noalign{\medskip}q&0&0&0\end {array}
 \right] 
}

\newcommand{\Hslashgluex}{
 \left[ \begin {array}{cccc} {k}^{2}&-kp&kp&pq\\ \noalign{\medskip}0&0
&{k}^{2}&kq\\ \noalign{\medskip}0&{k}^{2}&0&-kq\\ \noalign{\medskip}0&0
&0&{k}^{2}\end {array} \right]
}
\newcommand{\Hslashglu}{
\left[ \begin {array}{cccc} 1&0&0&1\\ \noalign{\medskip}0&0&-1&0
\\ \noalign{\medskip}0&-1&0&0\\ \noalign{\medskip}0&0&0&1\end {array}
 \right] 
}

\newcommand{\Hising}{
\left[ \begin {array}{cccc} 1&0&0&1\\ \noalign{\medskip}0&1&1&0
\\ \noalign{\medskip}0&-1&1&0\\ \noalign{\medskip}-1&0&0&1\end {array}
 \right]
}
\newcommand{\Height}{
 \left[ \begin {array}{cccc} {p}^{2}+2\,pq-{q}^{2}&0&0&{p}^{2}-{q}^{2}
\\ \noalign{\medskip}0&{p}^{2}-{q}^{2}&{p}^{2}+{q}^{2}&0
\\ \noalign{\medskip}0&{p}^{2}+{q}^{2}&{p}^{2}-{q}^{2}&0
\\ \noalign{\medskip}{p}^{2}-{q}^{2}&0&0&{p}^{2}-2\,pq-{q}^{2}
\end {array} \right]
}

\newcommand{\Hf}{
 \left[ \begin {array}{cccc} {k}^{2}&0&0&0\\ \noalign{\medskip}0&0&kq
&0\\ \noalign{\medskip}0&kp&{k}^{2}-pq&0\\ \noalign{\medskip}0&0&0&{k}
^{2}\end {array} \right]
}
\newcommand{\Ha}{
\left[ \begin {array}{cccc} {k}^{2}&0&0&0
\\ \noalign{\medskip}0&0&kq&0\\ \noalign{\medskip}0&kp&{k}^{2}-pq&0
\\ \noalign{\medskip}0&0&0&-pq\end {array} \right]
}

\[ \hspace{-.2in}
\Hslash , \HslashDS, \Hslashglu, \Hslashglue, \Hslashgluex, 
\]
\[
\Ha, \Haglue, \hspace{.2in} \Hf, 
\]
\[
\Hising, \;\; \Height 
\]

%% file: tex/mapleSRR1.tex
\newcommand{\ppp}{p}
\newcommand{\pppp}{s}
\newcommand{\qqqq}{t}
$$
\ignore{{
\left[\begin{array}{cccc}
a  & b  & c  & d  
\\
 e  & f  & g  & h  
\\
 j  & k  & l  & m  
\\
 n  & \pppp  & \qqqq  & r  
\end{array}\right],
}}
%
\An_{SRR} \; = \; 
\left[\begin{array}{cccccccc}
0 & -(e +j) \ppp  & 0 & (a  -f  -k) \ppp  & 0 & (a  -g  -l) \ppp  & 0 & (c-b -h -m)\ppp   
\\
 0 & n \ppp  & 0 & e \ppp +\pppp \ppp  & 0 & e \ppp +\qqqq \ppp  & 0 & (g-f  +r)\ppp   
\\
 0 & -n \ppp  & 0 & j \ppp -\pppp \ppp  & 0 & j \ppp -\qqqq \ppp  & 0 & (l-k  -r)\ppp   
\\
 0 & 0 & 0 & n \ppp  & 0 & n \ppp  & 0 & -\pppp \ppp +\qqqq \ppp  
\\
 0 & 0 & 0 & -2 b  & 0 & -2 c  & 0 & 0 
\\
 0 & 2 e  & 0 & 0 & 0 & 0 & 0 & 2 h  
\\
 0 & 2 j  & 0 & 0 & 0 & 0 & 0 & 2 m  
\\
 0 & 0 & 0 & -2 \pppp  & 0 & -2 \qqqq  & 0 & 0 
\end{array}\right]
$$

%% file: tex/mapleSRRc2subs.tex
\left[\begin{array}{cccccccc}
0 & 0 & 0 & p \left(a -f \right)+b  & 0 & \left(-a +f \right) p +c  & 0 & p \left(b +c -h \right)-c q +2 d  
\\
 0 & 0 & 0 & 0 & 0 & 0 & 0 & \left(-a +f \right) q +h  
\\
 0 & 0 & 0 & 0 & 0 & 0 & 0 & q \left(a -f \right)+c +b -h  
\\
 0 & 0 & 0 & 0 & 0 & 0 & 0 & 0 
\\
 0 & 0 & 0 & 0 & 0 & 0 & 0 & 0 
\\
 0 & 0 & 0 & 0 & 0 & 0 & 0 & 0 
\\
 0 & 0 & 0 & 0 & 0 & 0 & 0 & 0 
\\
 0 & 0 & 0 & 0 & 0 & 0 & 0 & 0 
\end{array}\right]

%% file: tex/mapleSRR4.tex
$$ { } \hspace{-1.8cm} 
\left[\begin{array}{cccccccc}
(-p n -e  -j) p  
& (-2 e  q +\left(p^{}- q \right) j -\left(p^{}  - q^{}\right) qn)p  
& (-p^{} s +a  -f  - k)p  
& 2 a p q +b p +b q -2 f p q +\left(p^{2}-p q \right) k -\left(p^{} q -q^{2}\right) ps  
& -p^{2} t +a p -p g -p l  
& -\left(a p +c \right) p +a p q +c q +2 p c -2 p q g -\left(-p^{2}+p q \right) l -
\left(p^{2} q -p \,q^{2}\right) t  
& \left(a p +c \right) p +b p -p^{2} r -p h -p m  
& \left(a p +c \right) p q +\left(b p +d \right) q -\left(a p q +c q \right) q +b p q +d q +2 p d 
-2 p q h -\left(-p^{2}+p q \right) m -\left(p^{2} q -p \,q^{2}\right) r  
\\
 -n p  & -2 e p -\left(q -p \right) e -2 n p q  & e p -s p  & 2 e p q -f p +f q -\left(q -p \right) f -2 s p q  & e p -p t  & -\left(e p +g \right) p +e p q +g q -\left(q -p \right) g -2 p q t  & \left(e p +g \right) p +f p -p r  & \left(e p +g \right) p q +\left(f p +h \right) q -\left(e p q +g q \right) q +f p q +h q -
\left(q -p \right) h -2 p q r  
\\
 -n p  & -2 j p -\left(q -p \right) j -\left(p q -q^{2}\right) n  
 & j p -s p  
 & 2 j p q  
    -\left(p q -q^{2}\right) s  
 & j p -p t  
 & -\left(j p +l \right) p +j p q +l q -\left(q -p \right) l -\left(p q -q^{2}\right) t  & \left(j p +l \right) p +p k -p r  & \left(j p +l \right) p q +\left(p k +m \right) q -\left(j p q +l q \right) q +k p q +m q -
\left(q -p \right) m -\left(p q -q^{2}\right) r  
\\
 0 & -2 n p -2 n q  & n p  & 2 n p q -s p -s q  & n p  & -\left(n p +t \right) p +n p q -t q  & \left(n p +t \right) p +s p  & \left(n p +t \right) p q +\left(s p +r \right) q -\left(n p q +t q \right) q +s p q -r q  
\\
 0 & -n \,q^{2}+e q +j q  & 0 & -q^{2} s -a q +f q +q k  & 0 & -q^{2} t -a q +g q +l q  & 0 & -\left(-a q +c \right) q -b q -q^{2} r +h q +m q  
\\
 0 & n q  & 0 & -e q +s q  & 0 & -e q +t q  & 0 & -\left(-e q +g \right) q -f q +r q  
\\
 0 & n q  & 0 & -j q +s q  & 0 & -j q +t q  & 0 & -\left(-j q +l \right) q -q k +r q  
\\
 0 & 0 & 0 & -n q  & 0 & -n q  & 0 & -\left(-n q +t \right) q -s q  
\end{array}\right]
$$

%% file: tex/mapleSRR4n0.tex
%
%
%
%
%
%
%
%
$$
(p+q)
\left[\begin{array}{cccccccc}
 0  
 & 0    
 &  0  
 &   (p  +b  ) 
 & 0  
 & (c   - p  )  
 & 0    
 & pm-qc  
 +2d  
\\
 0  & 0   &  0  &  0   & 0  & 0   & 0  
 &  (h-q)  
\\
 0  
 & 0  
 &  0  
 &  0   
 &  0  
 & 0   
 & 0  
 & (m+q)  
\\
 0  & 0 & 0   &  0   & 0  
 & 0   & 0    &  0 
\\
 0 & 0  & 0 & 0  & 0 &  0  & 0 & 0  
\\
 0 & 0  & 0 &  0  & 0 &  0  & 0 & 0  
\\
 0 & 0  & 0 &  0  & 0 &  0  & 0 & 0  
\\
 0 & 0 & 0 & 0  & 0 & 0  & 0 & 0  
\end{array}\right]
$$
%

%% file: tex/mapleSSR5.tex
\[
\left[
\begin{array}{cccccccc}
0 & 0 & b (\ppp -1)  & 0 & c (\ppp -1)  & 0 & d (\ppp^{2}-1)  & 0 
\\
 -e (\ppp -1)  & 0 & 0 & 0 & 0 & 0 & h\ppp(\ppp^{}  -1)  & 0 
\\
 -j (\ppp -1)  & 0 & 0 & 0 & 0 & 0 & m\ppp(\ppp^{}  -1)  & 0 
\\
 -n(\ppp^{2} -1)  & 0 & -s\ppp(\ppp^{} -1)   & 0 & -t\ppp(\ppp^{} -1)  & 0 & 0 & 0 
\\
 0 & 0 & 0 & -b\frac{\ppp+1}{\ppp^2} & 0 & -c\frac{\ppp+1}{\ppp^{2}} & 0 & d \frac{\ppp^2 -1}{\ppp^{2}} 
\\
 0 & e\frac{1+\ppp}{\ppp^{2}} & 0 & 0 & 0 & 0 & 0 & h \frac{\ppp+1}{\ppp} 
\\
 0 & j\frac{1+\ppp}{\ppp^{2}} & 0 & 0 & 0 & 0 & 0 & m \frac{\ppp+1}{\ppp} 
\\
 0 & n\frac{1-\ppp^2}{\ppp^{2}}  & 0 & -s\frac{1+\ppp}{\ppp}  & 0 & -t\frac{1+\ppp}{\ppp}  & 0 & 0 
\end{array}
\right]
\]

%% file: tex/mapleSSR4f-bcejhmst0.tex
\left[\begin{array}{cccccccc}
0 & 0 & 0 & 0 & 0 & \frac{a d}{ p }-d g  & d g  p  -a d  & 0 
\\
 0 & \frac{g f}{ p }-f a  & f(a  p  -g) & 0 & 0 & 0 & 0 & 0 
\\
 0 & \frac{l f}{ p }-d  p  n  & 0 & 0 & l(g-a  p )   & 0 & 0 & d(k-  p  r)   
\\
n(a -g  p )  & 0 & 0 & f(r-k  p )  & 0 & \frac{n d}{ p }-f  p  l  & 0 & 0 
\\
 0 & 0 & d  p  n -\frac{l f}{ p } & 0 & a l -\frac{l g}{ p } & 0 & 0 & d r -\frac{k d}{ p } 
\\
 n(g -\frac{a}{ p }) & 0 & 0 & f k -\frac{r f}{ p } & 0 & 0 & f  p  l -\frac{n d}{ p } & 0 
\\
 0 & 0 & 0 & 0 & 0 & \frac{l r}{ p }-k l  & k l  p  -l r  & 0 
\\
 0 & \frac{n k}{ p }-r n  & r n  p  -n k  & 0 & 0 & 0 & 0 & 0 
\end{array}\right]

%% file: tex/linear-rep-thy.tex
\section{Linear representation theory}   \label{ss:linRT}

\newcommand{\A}{\mathfrak{V}}


This paper is about MD representations, which are higher representations - functors. 
But as soon as one has an MD rep one has an ordinary representation for each $\MD_n$, so 
understanding an MD rep is closely tied to understanding the contained ordinary reps;
which is in turn tied to understanding their irreducible and indecomposable factors.
Thus we are motivated to review the ordinary linear representation theory.

There are a few useful approaches to the ordinary linear rep theory of $\MD_n\cong\MDD_n$. But from the $\MDD_n$ perspective the key tool is Clifford theory, so we look at this next.

\subsection{Some general Clifford theory}

\newcommand{\widehate}[1]{#1^{*}}

For notational convenience we will denote the normal abelian subgroup of $\MDD_n$ by $\A_n\cong\Z^{{n}\choose{2}}$, and the (dual) group of characters 
$\A^*_n$.  
Indeed for any abelian group $A$ we  
 denote the `full' dual by 
\beq 
{A^*} \; =\;  \hom_{Grp}(A,\C^{\times})
\eq 
- a group by pointwise multiplication. 
And 
for any locally compact abelian group $A$ we write the {\em unitary} dual 
\beq 
\widehat{A} \; =\;  \hom_{TopGrp}(A,S^1)
\eq  
where $S^1$ (or $U(1)$) is the circle subgroup of $\C^\times$, with the obvious topology - noting that $ \widehat{A} $ is a group by pointwise multiplication; and a topological group by the compact open topology. 

Let $\rho$ be an $A$-rep - not necessarily irreducible. Then we have the trace map to $\C$. This is a character - but not a simple character unless $\rho$ is simple. 
Thus the elements of the dual group are actually {\em simple} characters. 

Now continuing with $A$ an abelian group, let $H$ be a group acting on $A$ by automorphisms as in (\ref{eq:defsd0}), and $A \rtimes H$ the semidirect product. 
Evidently $H$ acts on $A^*$ by $\chi \stackrel{h \in H}{\mapsto} \chi^h$ where 
$\chi^h(a) = \chi(h^{-1} a h)$. 
For $\chi \in A^*$ we thus define the {\em stabiliser} $H$-subgroup 
$H_\chi = \{ h \in H \; | \; \chi^h = \chi \}$.

\ignore{{
\begin{example}
    One instructive case is for 2 dimensional irreducible representations of $\MD_3$.  Suppose that $\rho$ is such a representation.  Then the restriction to a subgroup isomorphic to $\Sigma_3$, say generated by $s_1,s_2$, has one of 4 forms:
    \begin{enumerate}
        \item two copies of the trivial representation,
        \item two copies of the sign representation,
        \item  one copy of each of the trivial and sign representations or
        \item one copy of the irreducible $2$-dimensional $\Sigma_3$ representation
    \end{enumerate} In any of the first three cases we have $\rho(s_1)=\rho(s_2)$, which implies $\rho(r_1)=\rho(r_2)$.  In the first two cases we have $\rho(s_i)$ commutes with $\rho(r_i)$, yielding a reducible representation. 
    
    In the third case we may assume $\rho(s_1)=\rho(s_2)=\left[ \begin {array}
{cc} 1&0\\ \noalign{\medskip}0&-1\end {array} \right]$ and $\rho(r_1)=\rho(r_2)=B$ with $B$ involutive with $AB\neq BA$.  Up to equivalence, we find that such a $B= \left[ \begin {array}{cc} z&1\\ \noalign{\medskip}-{z}^{2}+1&-z
\end {array} \right] $, 
where $z^2\neq 1$ to ensure irreducibility.  

In the last case we assume, without loss of generality, that \[(\rho(s_1),\rho(s_2))=\left(\left[ \begin {array}
{cc} 1&0\\ \noalign{\medskip}0&-1\end {array} \right] , \left[ 
\begin {array}{cc} -{\frac{1}{2}}&{\frac {\sqrt {3}}{2}}
\\ \noalign{\medskip}{\frac {\sqrt {3}}{2}}&{\frac{1}{2}}\end {array}
 \right]
\right),\] with breaks the symmetry afforded by the automorphisms described above.  In this case one may solve explicitly to obtain 

\[(\rho(r_1),\rho(r_2))=\left(\left[ \begin {array}{cc} {\frac {x\sqrt {3}}{2}}+{\frac {y}{2}}&-{
\frac {x}{2}}+{\frac {y\sqrt {3}}{2}}\\ \noalign{\medskip}-{\frac {x}{
2}}+{\frac {y\sqrt {3}}{2}}&-{\frac {x\sqrt {3}}{2}}-{\frac {y}{2}}
\end {array} \right] , \left[ \begin {array}{cc} -y&x
\\ \noalign{\medskip}x&y\end {array} \right]\right)\]
where $x^2+y^2=1$.  
\end{example}
}}


The first key is Clifford's Theorem \cite{Clifford37}, or rather the following generalisation to infinite groups:

\begin{theorem}\label{thm:infinite clifford}
    Let $H$ be a (not necessarily finite) group acting on an abelian group $A$, and $G:=A\rtimes H$ the corresponding semidirect product.  Suppose $(\pi,V)$ is an irreducible representation of $G$.
    Then 
    \begin{enumerate}
        \item[(I)] The restriction of $\pi$ to $A$ decomposes as a direct sum of irreducible characters of $A$.  
        \item[(I')] The irreducible characters in $\pi|_A$ all belong to the same $H$-orbit. 
    \item[(II)] Fix an irreducible  $A$-character $\chi$ appearing in $\pi|_A$. Let $H_\chi$ be its stabiliser as above. 
    Then 
    $\pi|_A\cong \bigoplus_{t\in T}(\chi^{t})^{\oplus m}$ 
    where $T$ is a transversal for 
    the set {$H/H_\chi$} \ppmm{of cosets}; 
    and $m$ is some integer  independent of $t$.
    \item[(III)] Let $V_\chi=\chi^{\oplus m}$ be the ($m$-dimensional) $\chi$-isotypic component of $\pi|_A$.  Then $V_\chi$ lifts to an irreducible $H_\chi$ representation, or indeed for any $\chi^t$ in the orbit of $\chi$.

    \end{enumerate}
\end{theorem}

\begin{proof}
(I) 
As $\pi(A)$ is a collection of commuting matrices, 
we may find a common eigenvector 
$v \in V$ for $\pi(A)$, 
{i.e. a 1-dimensional $A$-subrepresentation of $V$}.
{Fixing such a $v$, define $\lambda: A \rightarrow \C^\times$ by $av =\lambda_a v$.}
Let $W$ be the span of all $\pi(h)v$ for $h\in H$ 
- i.e. the $H$-representation induced from $v$. 
Consider any $a\in A$, and $h \in H$. Since 
$ah= h h^{-1} ah = h \psi_h(a)$ by (\ref{eq:defsd0}) and (\ref{pa:conj-auto}) we have
\beq \label{eq:piapihv} 
\pi(a)[\pi(h)v]= \pi(ah)v = \pi(h\psi_{h}(a))v = \pi(h)\pi(a')v=\lambda_{a'}\pi(h)v
\eq  
for  $a' = \psi_h(a)\in A$. 
Varying $a$ and $h$, we see from this that $W$ 
is also closed under the $A$ action, which is indeed diagonal here, and hence $W$  
is a subrepresentation of $V$.  
But since $V$ is irreducible, $W=V$. 
Let $\mathcal{B}$ be a basis of $W$ of the form $\{\pi(h_i)v\}_{i=1}^k$.  
Now note that $\pi(A)$ acts diagonally on $\mathcal{B}$. 
We have shown (I).

(I') Observe that $\lambda$ is one of our irreducible $A$-characters, with $\lambda(a) = \lambda_a$. 
And observe from (\ref{eq:piapihv}) that for any $a,h$ we have  
$a \, (\pi(h)v ) = \lambda^h(a)  \, (\pi(h)v)$, 
so the 1-dimensional $A$-rep spanned by $\pi(h)v$ is in the $H$-orbit of $\lambda$. This proves (I').

(II) 
Write 
$m_\chi$ for the multiplicity of $\chi$ as a summand in $\pi|_A$. In other words we 
{decompose $\pi|_A=\bigoplus_{\chi\in\widehate{A}}m_\chi\chi$.  
Fix one such $\chi$ with $m_\chi\geq 1$.
By the normality of $A$, for any $h\in H$ the $A$-representation given by $\pi^h(a)=\pi(h^{-1}ah)$ is isomorphic to $\pi|_A$ so
we find that 
for each $h \in H$ 
the $H$-twisted character $\chi^h(a):=\chi(h^{-1}ah)$ also appears with multiplicity $m_\chi$.  Thus the multiplicity $m_\chi$ is constant on $H$-orbits.
On the other hand, by the irreducibility of $\pi$ 
the characters appearing in $\pi|_A$ 
must be permuted transitively by $H$: the sum of the $\chi^h$-isotypic subspaces over all $h\in H$ is $H$-invariant and hence $G$-invariant. Thus, there is only one orbit, and we set $m:=m_\chi$. This proves (II).

(III)  The $\chi$-isotypic component $(\chi^{\oplus m},V_\chi)$ is $H_\chi$-stable 

and thus lifts to an $H_\chi$-representation $\tau$ of dimension $m$.  If $\tau$ had a non-trivial subrepresentation then taking the $H$-orbit of $\tau$ would yield a subrepresentation of $G$, contradicting irreducibility.}

\end{proof}

We now recall the Mackey--Clifford (``little groups'') machine in this setting, see e.g. \cite[Chapters 7 and 8]{Serre} for the finite group version, which we adapt using Theorem \ref{thm:infinite clifford}.  Indeed, the construction and proof of the proposition below are 
otherwise 
identical to the finite group case found in \cite[Section 8.2]{Serre}.

\mdef  \label{de:algorithm}
{In order to construct all irreducible representations of $G=A\rtimes H$ we proceed as follows: }

\begin{enumerate}
\item Fix any irreducible character \(\chi\in{A^*}\).

\item 

Determine $H_\chi$; and determine its irreducible representations. 
\item 
Form \(G_\chi:=\A\rtimes H_\chi\). There is a canonical extension of \(\chi\) to a one-dimensional representation of \(G_\chi\) given by
\[
(x,h)\longmapsto \chi(x)\qquad (x\in \A_n,\; h\in H_\chi),
\]
since \(h\) fixes \(\chi\) by definition. 
Fix any irreducible representation $\tau$ of \(H_\chi\). 
We obtain a representation
\[
\chi\otimes \tau\quad \text{of }G_\chi,\qquad (x,h)\mapsto \chi(x)\,\tau(h).
\]

\item Induce to \(G\):
\[
\rho(\chi,\tau):=\mathrm{Ind}_{G_\chi}^{G}(\chi\otimes \tau).
\]
\end{enumerate}

\mdef  Obviously steps (2) and (4) may be hard in general. We will discuss feasibility when we treat examples shortly.  

Then

\begin{theorem}
\label{th:allreps}
Let $G = A \rtimes H$ as in \ref{thm:infinite clifford} 
above. 
Let $\chi$ be an irreducible $A$-character; and $\tau$ be an irrep of 
$H_\chi = Stab_H(\chi)$. Then
\\
(I) The rep 
\(\rho(\chi,\tau)\) is irreducible; and 
\\ 
(II) every irreducible representation of \(G\) is obtained this way. Moreover,
\[
\dim \rho(\chi,\tau)= [H:H_\chi]\cdot \dim \tau .
\]
(III) If $\rho(\chi_1,\tau_1)$ and $\rho(\chi_2,\tau_2)$ are isomorphic then $\chi_2=(\chi_1)^h$ 
\ppmm{for some $h \in H$}
and $\tau_2\cong \tau_1$.
\end{theorem}

\begin{proof}
    See \cite{Serre}, \ppmm{deploying Thm.\ref{thm:infinite clifford} at the appropriate moment}.
\end{proof}

\medskip 

\subsection{Application to $\MDD_n$ - generalities}

\newcommand{\chill}{\underline{\chi}}

We now apply Thm.\ref{th:allreps} to our case of $\MDD_n=\A_n\rtimes \Sym_n$.  Fix an irreducible character \(\chi\in{\A_n^*}\). 
Note that $\A_n$ is free-abelian on its generators $x_{ij}$ so an irreducible character is given by an element $v$ of $(\C^\times)^{{{n}\choose{2}}}$: $\chi_v (x_{ij} ) \; = v_{ij}$ and \(\chi_v(x_{ji})=v_{ij}^{-1}\). 

In order to give a specific $\chi$ in case  $n>2$ let us abuse notation slightly and write 
$$
\chill = (\chi(x_{12}),\chi(x_{13}),...,\chi(x_{1n}),\chi(x_{23}),...,\chi(x_{2n}),...,\chi(x_{n-1,n}))
$$
- for example 
$\chill =( \chi(x_{12}),\chi(x_{13}),\chi(x_{23}) )$ 
in case $n=3$. 

Because the set of $\A_n$-characters is very large, but the set of possible stablisers is finite, it is convenient to organise the characters according to the stabiliser. To see how to do this, we first do some examples with low $n$.

\medskip 

 The way that we will need to use the irrep construction here in practice is slightly indirect. What follows are first a couple of direct examples, and then some illustrative computations in the format that will be useful. 

\mdef Example. Consider $n=2$. 
\\ 
Step \ref{de:algorithm}(1-2): for $\chi$ we choose a single $v_{12}$. 
Here $H=\Sym_2$ which is abelian, but $\sigma_1$ acts on $\chi$ by $v_{12} \mapsto v_{21} = v_{12}^{-1}$, 
so either (i) $v_{12}=\pm 1$ and $H_\chi =H$;
or (ii) $H_\chi$ is trivial. 
\\ 
Step \ref{de:algorithm}(3):
In case (i)  then $H_\chi = H$ has two irreps. In both cases the Kronecker is 1d 
and Step~(4) is trivial, so the reps will be clear.
\\
In case (ii) there is nothing to do for Step (3). 
For Step (4) the cosets are the elements of $H=\Sym_2$ so we will induce to a 2d rep.
The simple characters contained are $\chi$ and $\chi^{\sigma_1}$. 
We have 
$ \rho(\chi,1)(x_{12}) = diag(v_{12}, v_{12}^{-1})$ and 
$ \rho(\chi,1)(\sigma_1) =  \begin{bmatrix}0&1\\[2pt]1&0\end{bmatrix}$.

\mdef Example. Consider $n=3$. 
\\
Step (1-2): For $\chi$ we choose $v_{12},v_{13},v_{23} \in \C^\times$. 
Here $H=\Sym_3$, so there are subgroups of order 1,2,3 and 6. 
\\
(i) To obtain $H_\chi = H$ we need $\chi^{w}=\chi$ for $w \in \Sym_3$, so in 
particular 
$\chi(x_{12}) = v_{12} = \chi^{\sigma_1}(x_{12}) = \sigma_1 v_{12} = v_{12}^{-1}$ so $v_{12}= \pm 1$;
and similarly for $v_{ij}$; furthermore $v_{12} = \sigma_2 v_{12} = v_{13}$ and so on.
Thus there are two characters with this property: $\chill=\pm(1,1,1)$.
\\
(ii) To obtain $H_\chi = \{1,\sigma_1 \}$ (a representative of the subgroups of order 2) we need $\chi^{\sigma_1}=\chi$ so again $v_{12} = \pm 1$; furthermore $v_{23} = \sigma_1 v_{23} = v_{13}$. 
Overall $\chill=(\pm 1,v_{13},v_{13})$, with $v_{13}$ generic. 
\\
(iii) To obtain $H_\chi = \{1,\sigma_1\sigma_2,\sigma_2\sigma_1 \}$ we need $\chi^{\sigma_1\sigma_2}=\chi$, so $v_{12} = \sigma_1 \sigma_2 v_{12} = v_{23}$
and similarly $v_{23} = v_{31} = v_{13}^{-1}$. 
Thus $\chill=(v_{12},v_{12}^{-1},v_{12})$.
\\ 
(iv) If we choose $\chi$ generically then $H_\chi$ is trivial.
\medskip 
\\
Step (3-4): In case (i) then we obtain $\rho(\chi=\pm1,\lambda)$ where $\lambda\in irreps(\Sym_3) = \{ (3), (2,1), (1^3) \}$ in the usual notation. 
\\
In case (ii) we have $H_\chi \cong \Sym_2$ with two 1d irreps. 
The induction cosets are 
$\{ \{1,\sigma_1 \}, \{ \sigma_2, \sigma_2\sigma_1 \}, \{  \sigma_1\sigma_2 , \sigma_1 \sigma_2 \sigma_1  \} \}$. 
\\
In case (iii) we have $H_\chi \cong C_3$ (i.e. the cyclic group $\Z/3\Z$ taken multiplicatively) with three 1d irreps. 
\\
In case (iv) 
then $\rho(\chi,1)$ is 6d. In particular writing $\Sym_3$ in the order 
$\{ 1, \sigma_1 , \sigma_2, \sigma_1 \sigma_2 , \sigma_2 \sigma_1, \sigma_1 \sigma_2 \sigma_1 \}$ then 
$\rho(\chi,1)(x_{12}) = diag(v_{12}, v_{12}^{-1}, v_{13}, \ldots,v_{23}^{-1})$. 
\\

\mdef Observe that for general $n$ there are some patterns in the general construction, but also components of increasing complexity. We will not need these here, since our objective is to provide tools to analyse MD representations in rank 2,
and for this the complexity can be capped, as we will see.

\mdef 
From the isomorphism \eqref{iso} 
\ppmm{
we can see} 
that for any representation $\rho$ of $\MDD_n$ we have 
\beq  \label{eq:det1}
\det(\rho(r_is_i))=\det(\rho(x_{i,i+1}))=\pm 1=\det(\rho(x_{jk})).
\eq for all $j,k$.  Indeed, 
since the $r_i,s_i$ are involutions the determinant of any of their images is $\pm 1$, and hence the same holds for their products, such as $x_{i,i+1}$.  Then since the $x_{jk}$ are all conjugate we deduce the result.

\subsection{Complete classification of one and two dimensional irreducibles} $\;$ 

Observe from \ref{th:allreps}(II) that for a 1d rep we require $H_\chi =H$. 
This allows only for the two choices for $\chi$, with $\chi(x_{ij})=\pm 1$. 
We must further choose a 1d rep $\tau$ of $H_\chi = H$. 

In our case, for any $n$, we have $H= \Sym_n$, so $\tau$ is either the trivial or the alternating. 
Altogether we have the four 1d reps starting with the trivial and simply varying the signs of the images of $r_1$ and $s_1$.

\begin{example}\label{ex: 2d irreps MD_3}[Two-dimensional irreducibles]
Here we look specifically for 2d irreps of \(\MDD_n\).

For \(n=2\), \(\MDD_2\) is the infinite dihedral group; 
it has 
\(2\)-dimensional irreps where, in a diagonal basis,
\(X_{12}=\mathrm{diag}(a,a^{-1})\) with \(a\in\mathbb{C}^\times\) and \(S_1=\begin{bmatrix}0&1\\[2pt]1&0\end{bmatrix}\).

Assume \(n\ge 3\). By \eqref{eq:det1}, we may take either \[X_{12}=\begin{bmatrix} -a&0\\[2pt]0&a^{-1}\end{bmatrix}\quad \textrm{or}\quad
X_{12}=\begin{bmatrix} a&0\\[2pt]0&a^{-1}\end{bmatrix}.
\]  Since $X_{12}$ is conjugate to its inverse, in the first case we must have $\{-a,a^{-1}\}=\{-a^{-1},a\}$, thus $a=a^{-1}$, in which case $X_{12}=X_{12}^{-1}$ with $a^2=1$.  But then $X_{12}$ and $S_1$ commute and hence are simultaneously diagonalisable. But then each $X_{ij}=\pm X_{12}$ and so each $S_i$ commutes with the $X_{ij}$.  This implies that we have a $2$-dimensional irreducible representation of $\Sym_n$, implying $n=3$ or $4$.

In the second case if \(a^2=1\) then \(X_{12}=\pm I\), hence all \(X_{ij}=\pm I\), so the \(S_i\) commute with \(\A_n\) and \(\rho_n|_{\Sym_n}\) is an irreducible \(2\)-dimensional representation of \(\Sym_n\). This forces \(n=3\) or \(4\) (the only \(\Sym_n\) with \(2\)-dimensional irreps).

Now assume \(a^2\neq 1\). We claim \(X_{13}=X_{12}^{-1}\). Indeed, if \(X_{13}=X_{12}\), then
\(
S_2 X_{12} S_2^{-1}=X_{13}=X_{12},
\)
so \(S_2\) commutes with \(X_{12}\) and hence with \(X_{12}^{-1}\). But then \(S_2 X_{23} S_2^{-1}=X_{23}\), contradicting the relation that \(X_{23}\) is conjugate to its inverse via \(S_2\), unless \(a^2=1\). Thus \(X_{13}=X_{12}^{-1}\), and consequently \(X_{23}=X_{12}\).

Up to swapping the basis and scaling, this forces
\[
S_1=\begin{bmatrix}0&1\\[2pt]1&0\end{bmatrix},\qquad
S_2=\begin{bmatrix}0&\omega\\[2pt]\omega^{-1}&0\end{bmatrix},\qquad \omega=e^{2\pi i/3},
\]
which is equivalent to the standard \(2\)-dimensional irrep of \(\Sym_3\).

\paragraph{Second approach (direct Mackey induction).}
Fix $a\in\C$ with $a^2\neq 1$ and define a character $\chi^a$ of $\A_3\cong \Z^3$ on generators by
\[
(x_{12},x_{13},x_{23})\longmapsto (a,a^{-1},a).
\]
Let $C_3:=\langle(1\,2\,3)\rangle<\Sym_3$ be the stabiliser of $\chi^a$. For $i=0,1,2$ let
$\tau_i$ be the irreducible character of $C_3$ with $\tau_i((1\,2\,3))=\omega^i$, where $\omega=e^{2\pi i/3}$.
Then we get $1$-dimensional representations of $\A_3\rtimes C_3$ by
\[
\pi_i^a(xh)=\chi^a(x)\,\tau_i(h)\qquad (x\in\A_3,\ h\in C_3).
\]
Induce to $\A_3\rtimes \Sym_3$:
\[
\rho_i^a:=Ind_{\A_3\rtimes C_3}^{\A_3\rtimes\Sym_3}(\pi_i^a).
\]
Since $[\Sym_3:C_3]=2$, each $\rho_i^a$ is $2$-dimensional.

Choose left coset representatives $e,(1\,2)$ and let $v_1,v_2$ be the corresponding basis of
$eV\oplus(1\,2)V$, where $V$ is the $1$-dimensional space of $\pi_i^a$.
Then, using $x_{12}(1\,2)=(1\,2)x_{12}^{-1}$ etc., we get
\[
\rho_i^a(x_{12})=diag(a,a^{-1}),\qquad
\rho_i^a(x_{13})=diag(a^{-1},a),\qquad
\rho_i^a(x_{23})=diag(a,a^{-1}).
\]
By our coset choice,
\[
\rho_i^a((1\,2))=\begin{bmatrix}0&1\\[2pt]1&0\end{bmatrix}.
\]
Using $(2\,3)(1\,2)=(1\,3\,2)e$ and $(2\,3)e=(1\,2)(1\,2\,3)$,
\[
\rho_i^a((2\,3))=
\begin{bmatrix}
0 & \tau_i(1\,2\,3)\\[2pt]
\tau_i(1\,3\,2) & 0
\end{bmatrix}
=
\begin{bmatrix}
0 & \omega^i\\[2pt]
\omega^{-i} & 0
\end{bmatrix}.
\]
For $i=1$ this reproduces exactly the matrices in the first $2$-dimensional construction above.

\medskip
This gives a family of $\MD_3$ representations; it also occurs as the $\Sym_3$-block inside any $\MD_n$ representation with $n\ge 3$ obtained from the same orbit data.
For $n\ge 4$, restrict to $\Sym_4$. The only $2$-dimensional irrep of $\Sym_4$ that restricts irreducibly to $\Sym_3$ satisfies $S_3=S_1$ on the $\Sym_3$ block. In our family,
\[
S_1X_{12}S_1^{-1}=X_{12}^{-1},\qquad S_3X_{12}S_3^{-1}=X_{12},
\]
so a lift to $\Sym_4$ exists only in the degenerate case $a^2=1$.

\medskip\noindent
\emph{Conclusion.} $\MD_n$ has irreducible $2$-dimensional representations only for $n=3$ and $n=4$.
For $n=3$ there is a $1$-parameter family (parameter $a$ with $a^2\ne 1$, and $i\in\{0,1,2\}$ indexing the $C_3$ character); for $n=4$ there are exactly two, corresponding to $X_{ij}=\pm I$ (the $a^2=1$ case).

\end{example}

\begin{example}[Three-dimensional irreps with order-2 stabilizer]\label{ex: 3d irreps of MD3}
Take \(\chi=\eta_a^\pm\in{\A_3}^*\) given by
\[
(x_{12},x_{13},x_{23})\mapsto (a,a,\pm 1).
\]
Its stabilizer in \(\Sym_3\) is \(H=\langle(2\,3)\rangle\cong \Sym_2\). Let \(\varepsilon\) be either the trivial or sign character of \(H\), and set
\(\pi_a^{\pm,\varepsilon}:=\chi\boxtimes \varepsilon\) on \(\A_3\rtimes H\).
Inducing to \(\A_3\rtimes \Sym_3\) with coset reps \(e,(1\,2),(1\,3)\) yields a \(3\)-dimensional irrep \(\rho_a^{\pm,\varepsilon}\) with
\[
\rho_a^{\pm,\varepsilon}(x_{12})=\mathrm{diag}(a,a^{-1},\pm 1),\quad
\rho_a^{\pm,\varepsilon}(x_{23})=\mathrm{diag}(\pm 1,a,a^{-1}),\quad
\rho_a^{\pm,\varepsilon}(x_{13})=\mathrm{diag}(a,\pm 1,a^{-1}),
\]
and
\[
\rho_a^{\pm,\varepsilon}((1\,2))=\begin{bmatrix}0&1&0\\[2pt]1&0&0\\[2pt]0&0&\theta\end{bmatrix},\qquad
\rho_a^{\pm,\varepsilon}((2\,3))=\theta\begin{bmatrix}1&0&0\\[2pt]0&0&1\\[2pt]0&1&0\end{bmatrix},
\quad \theta:=\varepsilon((2\,3))\in\{\pm 1\}.
\]
\end{example}

\ignore{{
\begin{example}[Generic six-dimensional case\ecr{[merge with the above?]}]
For \(\chi\) with parameters \((a,b,c)\) such that the multiset \(\{a^{\pm1},b^{\pm1},c^{\pm1}\}\) has six distinct entries, the stabilizer is trivial. Then
\(\rho(\chi,\mathbf{1})\) has dimension \(6\); restricted to \(\Sym_3\) it is the left-regular representation, and the diagonal entries for \(x_{12},x_{13},x_{23}\) on the induced basis are respectively
\[
(a,a^{-1},c^{-1},b,b^{-1},c),\quad
(b,c,b^{-1},a,a^{-1},c^{-1}),\quad
(c,b,a^{-1},c^{-1},b^{-1},a).
\]\ecr{[typo here...?]}
\end{example}
}}

\begin{example}[Special characters \(\pm(1,1,1)\)]
If \(\chi(x_{12},x_{13},x_{23})=\pm(1,1,1)\), then all \(X_{ij}=\pm I\) and the \(\Sym_3\)-subgroup may act by any irreducible representation of \(\Sym_3\).  This produces irreps of dimensions $1$ and $2$.  Notice that the $2$ dimensional case is already covered by Example \ref{ex:2d irreps MD_3}
\end{example}

\mdef While the classification problem for general $n$ is likely a bit tedious there are some general observations one can make.  
\begin{enumerate}
    \item The subgroups of $\Sym_n$ that can appear as $H_\chi$ for some character $\chi=(a_{1,2},\ldots, a_{n-1,n})$ of $\A_n$ may not be as plenteous as one might naively imagine. 
    If $H$ acts $2$-transitively then $a_{12}=a_{21}$ and $a_{ij}=a_{12}^{\pm 1}$ for all $i,j$ so we have $a_{ij}=\pm 1$, for all $i<j$ and hence $\chi$ is given by $\pm(1,\ldots,1)$.  But this has stabiliser $S_n$.
    
    \item If $H$ acts $2$-homogeneously, i.e., if it acts transitively on the sets $\{i,j\}$ for $i\neq j$, then we must have $a_{ij}=a_{km}^{\pm 1}$ for all quadruples $(i,j,k,m)$.  If it is not a $2$-transitive action then no $(i,j)$ has $(j,i)$ in its orbit, so there are exactly two orbits. Thus  the character takes value $a$ on one orbit and $a^{-1}$ on the other.  For example if we take the subgroup $C_3$ of $S_3$ then we have $(1,2),(2,3)$ in one orbit and $(1,3)$ in the other.
    \item A convenient geometric description of the action of $\Sym_n$ on $\A_n$ is to label the edges of $K_n$ by elements of $\A_n$.  
One may partition the edges and $H$ be the subgroup of $\Sym_n$ that preserves the partition.  One can define a character by assigning a complex value $a_k$ for each member of the partition $P_k$, with the rule that the if some member of $H$ flips some $(i,j)\mapsto (j,i)$ in $P_k$ then $a_k=\pm 1$.  For example if we partition the edges of $K_3$ into $[\{1,2\},\{1,3\}]$ and $[\{2,3\}]$ then the group that preserves this partition is $\langle (2\; 3)\rangle$.  Since $(2\; 3)$ doesn't flip $(1,2)$ or $(1,3)$, but does flip $(2,3)$ we get the character $(a,a,\pm 1)$.
\item Note that a representation $\pi$ of $\MD_n$ with $\pi(r_i)=\pm\pi(s_i)$ corresponds to a representation of $\MDD_n$ with $\pi(x_{ij})=\pm I$.  Such a representation is clearly semisimple, with irreducible constituents obtained from the characters $\pm(1,\ldots,1)$.  This is a generalisation of the \emph{Wangian} local representations described below.

\end{enumerate}

%% file: tex/macros-glue.tex


\newcommand{\soutx}[1]{} 

\newcommand{\codt}{{\mathbf Dot product}} 

\newcommand{\SSS}{{\mathbb{S}}}
\newcommand{\B}{{\mathbb{B}}}



\newcommand{\MATCHA}{\MAT_{ACC}}
\newcommand{\botimes}{ \bar{\tens} }
\newcommand{\bartens}{ \bar{\tens} }
\newcommand{\ssss}[1]{\noindent {\bf{#1}} } 

\newcommand{\MATCH}{\mathbf Match}
\newcommand{\MAT}{{\mathbf Mat}} 
\newcommand{\VEC}{\mathbf Vec}
\newcommand{\mqxx}[1]{}
\newcommand\myeq{\mathrel{\overset{\makebox[0pt]{\mbox{\normalfont\tiny\sffamily Mult A,B}}}{=}}}
\newcommand\mmyeq{\mathrel{\overset{\makebox[0pt]{\mbox{\normalfont\tiny\sffamily Mult A,B,C}}}{=}}}
\newcommand\mmmyeq{\mathrel{\overset{\makebox[0pt]{\mbox{\normalfont\tiny\sffamily Mult assoc}}}{=}}}
\newcommand\mmmmyeq{\mathrel{\overset{\makebox[0pt]{\mbox{\normalfont\tiny\sffamily ret comp}}}{=}}}

\newcommand{\U}[1]{\underline{#1}}
\newcommand{\NN}{\U{N}}

\newcommand{\Lamm}[2]{\Lambda^{#1}_{#2}}  
\newcommand{\lamm}{{\bf m}}

\newcommand{\Symm}[1]{\Sigma_{#1}}
\newcommand\inner[2]{\langle #1, #2 \rangle}
\newcommand{\ff}{{\mathsf f}}

\renewcommand{\fff}[1]{{\bf{#1}}}

%% file: tex/CCwithglue0.tex
\newcommand{\lav}{\lambda}
\newcommand{\mw}{\mu}
\newcommand{\off}{\overline{\ff}}
\newcommand{\unk}[2]{\raisebox{-5pt}{${#1 \atop #2}$}}
\newcommand{\unkk}[2]{\raisebox{-5pt}{$\overset{\scalebox{1.02}{$#1$}}{\scriptscriptstyle{#2}}$}} 
\newcommand{\unkkk}[2]{\raisebox{-4pt}{\begin{array}{c} #1 \\ {}^{#2}\end{array}}}

\section{Charge-conservation with glue: Appendix extracted from \cite{tbd}}

In this section we use Almateari's methods 
\cite{AlmateariThesis,Almateari24}
to introduce a new strict monoidal 
category, again intended as a target in representation theory;
and deduce some general representation theoretic properties.

Note that 
the charge-conserving (CC) category 
$\MATCH^N$ and 
the additive charge-conserving (ACC) category 
$\MATCHA^N$ coincide for $N=2$, and that matrices of this type
do not exhaust the types seen in Hietarinta's rank-2 braid representation classification \cite{Hietarinta1992}. 
A type much closer to including all the varieties in the classification is given by the `CCwg' matrices to be described below. 
As shown in \cite{Almateari24}, the generic ansatz for a braid representation in ACC has a lot more parameters than CC. In this sense ACC is a big step up in difficulty from the solved CC case, and indeed the advantageous calculus of CC does not have a directly similarly strong lift to ACC, so that after $N=3$ the all-higher-ranks ACC braid rep classification problem remains open. Like ACC, CCwg also extends CC, but a vestige of the CC calculus is retained, so it offers the twin advantages of both containing a more-complete transversal in rank-2; and more tools towards a potential all-rank solution. 
The analogues of non-semisimplicity are not direct, or even canonical, in higher representation theory, but CCwg can also be seen as a kind of generalisation to allow a form of higher non-semisimplicity, as we shall see.

\medskip 

We first recall some notation from 
\cite{MR1X,AlmateariThesis}.

\noindent 
For $N,n \in \N$, 

$\NN = \{1,2,...,N \}$

$
\Lamm{N}{n}  = \{ \lamm \in \N_0^N \; : \;  \lamm.(1,1,...,1)\;  =n   \}
$. 

The function $\#_i : \U{N}^n \rightarrow \N$ counts the number of copies of symbol $i$ in a word. 

The function $\ff: \U{N}^n \rightarrow \Lamm{N}{n}$ is given 
by $\ff(v)_i = \#_i(v)$.

The Kronecker product (with Ab and revlex conventions) is given, 
for $A \in \MAT^N(n,n)$ and $B \in \MAT^N(m,m)$, by 
\beq \label{eq:MATNKronecker}  
\langle w | A \otimes B | v \rangle \; = \langle \unkk{w}{n-} | A |  \unkk{v}{n-} \rangle \langle \unkk{w}{-m} |B| \unkk{v}{-m} \rangle  
\eq 
where  $w,v \in \NN^{n+m}$ and 
$\unkk{w}{n-}$ denotes $w_1 w_2 ... w_n$ and $\unkk{w}{-m} = w_{n+1}...w_{n+m}$.

%% file: tex/CCwithgluev2.tex


\ignore{{
In this section we use Almateari's methods to introduce a new strict monoidal 
category, again intended as a target in representation theory.
(Note that $\MATCH^N$ and $\MATCHA^N$ coincide for $N=2$, and that matrices of this type
do not exhaust the types seen in Hietarinta's rank-2 braid representation classification \cite{Hietarinta92}. 
A type much closer to including all the varieties in the classification is given by the `CCwg' matrices to be described below. As shown in this paper, the generic ansatz for a braid representation in ACC has a lot more parameters than CC. In this sense ACC is a big step up in difficulty from the solved CC case, and indeed the advantageous calculus of CC does not have a directly similarly strong lift to ACC, so that after $N=3$ the all-higher-ranks ACC braid rep classification problem remains open. Like ACC, CCwg also extends CC, but a vestige of the CC calculus is retained, so it offers the twin advantages of both containing a more-complete transversal in rank-2; and more tools towards a potential all-rank solution. 
The analogues of non-semisimplicity are not direct, or even canonical, in higher representation theory, but CCwg can also be seen as a kind of generalisation to allow a form of higher non-semisimplicity, as we shall see.)

\medskip 

We first recall some notation from the paper.

For $N,n \in \N$, 

$\NN = \{1,2,...,N \}$

$
\Lamm{N}{n}  = \{ \lamm \in \N_0^N \; : \;  \lamm.(1,1,...,1)\;  =n   \}
$ 
as in \ref{de:Lamm}. 

The function $\ff: \U{N}^n \rightarrow \Lamm{N}{n}$ is given in \ref{de:ff}.
}} 

\newcommand{\MATCHwg}{\MATCH_g}
\newcommand{\rg}[1]{\textcolor{orange}{#1}}

\subsection{The main construction: the categories $\MATCHwg^N$}
\label{ss:constructCCwg}
$\;$ 

Here we give the construction, summarised by the well-definedness theorem, for the CC-with-glue categories $\MATCHwg^N$ - in (\ref{th:CCwg}). 
We need various inputs, starting with an order on words.

\mdef 
Revlex order example (order in the case $\U{3}^4 \rightarrow \Lamm{3}{4}$):
\[
\begin{array}{cccccccccccccc}
w    &    1111 & 2111 & 3111 & 1211 & 2211 & 3211 & 1311 & 2311 & 3311 & 1121 & ... 
\\
\ff(w) &    \underline{400}  & \underline{310}  & \underline{301} &  310  & \underline{220}   & \underline{211} &  301 & 211 & \underline{202} & 310 & ...
\end{array}
\]
In the $\ff(w)$ row we have underlined the first instance of each composition in this order.

Observe that 
the 
(revlex) 
first instance of 
a word $w \in \NN^n$ with 
$\ff(w) = \; \lambda \; =\; \lambda_1 \lambda_2 ...\lambda_N$ is the orbit element with 
weakly decreasing $w_i$:
\beq  \label{eq:orbrep}
\overline{\ff}(\lambda) \;  := \; 
\underbrace{N,N,...,N}_{\lambda_N \; copies}, 
 \underbrace{N\!-\!1, N\!-\!1, ..., N\!-\!1}_{\lambda_{N\!-\!1}},...,
 \underbrace{11,,1}_{\lambda_1} 
\eq

\mdef   \label{de:Lammo}
Fix $N \in \N$. For $n \in \N$ give an order $(\Lamm{N}{n}, <)$
by ordering $\NN^n$ in revlex, apply $\ff$, then $\lambda < \mu$ if $\lambda$ first appears 
before $\mu$ in sequence $\ff(\NN^n)$. 
\\
(For later use we set $\lambda \in \Lamm{N}{n}$, $\mu \in \Lamm{N}{n'}$ 
incomparable unless $n=n'$.)

\mdef   \label{de:CCwg}
We say  a matrix $M \in \MAT^N$ is {\em rank-$N$ charge-conserving-with-glue} 
(CCwg)
if $\langle w | M | v \rangle \neq 0 $ implies $\ff(w) \leq \ff(v)$
(that is, $\ff(w) > \ff(v)$ implies $\langle w | M | v \rangle = 0 $).
\\
(By the incomparability convention only the zero non-square matrices are CCwg.)

\mdef Examples. Let us write $\MATCHwg^N(n,n)$ for the set of CCwg matrices. Then
\[
\begin{bmatrix} 1 & \rg1 & \rg1 & \rg1 \\ &1&1&\rg1 \\ &1&1&\rg1 \\&&&1\end{bmatrix}
\in \MATCHwg^2(2,2)
\]
\[
\begin{bmatrix} 
1 & \rg1 & \rg1 & \rg1 & \rg1 &\rg1 &\rg1 &\rg1 &\rg1 \\
  & 1 & \rg1 & 1 &\rg1&\rg1&\rg1&\rg1&\rg1 \\
  &   & 1 & &  \rg1& \rg1&1&\rg1&\rg1 \\
  & 1 & \rg1  &1&\rg1&\rg1&\rg1&\rg1&\rg1 \\
  &&&&1&\rg1&&\rg1&\rg1 \\
  &&&&&1&&1 &\rg1\\
  &&1&&\rg1&\rg1&1&\rg1&\rg1 \\
  &&&&&1&&1&\rg1 \\
  &&&&&&&&1
  \end{bmatrix}
\in \MATCHwg^3(2,2)
\]
(omitted entries are 0). 
Indeed the black entry positions 
correspond to a row $w$ and column $v$ that 
satisfy $\ff(w) = \ff(v)$; 
and the orange satisfy $\ff(w) < \ff(v)$. 
\\ 
Observe that we have a kind of partially-sorted upper-block-triangularity. 

\mdef  \label{de:glue}
In general we call the positions 
$\langle w | M | v \rangle  $
in a matrix $M \in \MAT^N(n,n)$
that satisfy $\ff(w) < \ff(v)$ {\em glue} positions. 
\\
We may call positions that satisfy $\ff(w) = \ff(v)$ {\em CC} 
(charge-conserving)
{positions}. 

\newcommand{\KKK}{\overline{K'}}

\mdef  \label{de:K'K}
For each $N,n,m$ 
we write $K': \MATCHwg^N(n,m) \rightarrow \MATCHwg^N(n,m)$ for the projection 
that sends all glue positions to zero. 
And $\KKK$ for the complementary projection that sends all non-glue 
(i.e. CC) positions to zero. 
Observe that $  M \mapsto K'(M)$ also gives a well-defined map  
$K: \MATCHwg^N(n,m) \rightarrow \MATCH^N(n,m)$. 
Of course all these maps are linear. 
\\ 
We write $\gamma$ for the equivalence relation on $ \MATCHwg^N(n,m)  $ 
given by $a \gamma b$ if $a-b$ is non-zero at most in glue positions. 
Thus  $\MATCH^N(n,m) $ is a transversal of the classes of $\gamma$. 

\medskip

To prove our main theorem 
\ref{th:CCwg}
below we will  need to understand the order 
$(\Lamm{N}{n},>)$ a bit better. 

\ignore{{
\newcommand{\lav}{\lambda}
\newcommand{\mw}{\mu}
\newcommand{\off}{\overline{\ff}}
\newcommand{\unk}[2]{\raisebox{-5pt}{${#1 \atop #2}$}}
\newcommand{\unkk}[2]{\raisebox{-5pt}{$\overset{\scalebox{1.02}{$#1$}}{\scriptscriptstyle{#2}}$}} 
\newcommand{\unkkk}[2]{\raisebox{-4pt}{\begin{array}{c} #1 \\ {}^{#2}\end{array}}}
}}

\mdef   \label{de:Lammoo}
Define {\em an} order  $(\Lamm{N}{n},\prec)$  
by $\lav \prec  \mw$ 
(i.e. sequence $\lav_1 \lav_2 ... \lav_n \;\prec\; \mw_1 \mw_2 ... \mw_n$)
if the first difference from
the left is, for some $i$, $\lav_i \neq \mw_i$, and then $\lav_i > \mw_i$.

\mdef 
Example: 
Consider $\lambda = 53104$ and $\mu=53203 \in \Lamm{5}{13}$.
The first difference is $\lambda_3 = 1 < 2= \mu_3$, so $\lambda \succ \mu$.

To compare with $(\Lamm{N}{n},<)$: 
With $\NN^n$ in revlex order, the first  
instance of a word $w$ with $\ff(w)=$ 53104 $>$ (meaning first instance comes later in revlex order than)  the first instance of a word with $\ff(v)= 53203$.
For 53104 the first instance is the word 
$\off(53104) = 5555322211111$, by (\ref{eq:orbrep}).
For 53203 it is 5553322211111. 
Starting at 5553322211111 we count up, first to 
1114322211111, then 
\\ 2114322211111,
..., 5554322211111, 1115322211111, ..., 5555322211111.
\\
Observe that since early digits change quickest in revlex the right-hand end of 
the word (here 322211111) does not change in this sequence. 
The immediately preceding letter starts at 3 and gradually ascends to 5.

\mdef {\bf Proposition}. \label{pr:dagger}
Fix $N,n\in\N$. 
The order $(\Lamm{N}{n},\prec)$ from (\ref{de:Lammoo}) is our order 
$(\Lamm{N}{n}, <)$
from (\ref{de:Lammo}).

\medskip

\noindent {\em Proof}. 
Put $\NN^n$ in revlex order. 
Keep in mind (\ref{eq:orbrep}).
So suppose $\lambda \prec \mu$, and compare words $ \off(\lambda) , \off(\mu) $. 
The compositions are the same up to $\lambda_{i-1}$ for some $i$ 
(and then $\lambda_i > \mu_i$) 
so both words 
end in the same descending pattern 
$$
\underbrace{i,i,...,i}_{\mu_i},
\underbrace{i-1,i-1,...,i-1}_{\lambda_{i-1}},i-2,...,
\underbrace{1,1,...,1}_{\lambda_1}.
$$
Note that the immediately preceding letter is $i$ for  $\off(\lambda)$ and $>i$ for $\mu$
(as in the 53104 example). But of course here this digit is gradually ascending in revlex, so $\lambda<\mu$.
\qed 



\mdef 
For $v \in \NN^n$ we set $\#_i(v) = \ff(v)_i$, the number of occurrences of symbol $i$ in $v$.  
$\;$ 
For example $\#_2 (1121) = \ff(1121)_2 =  1$.

\mdef 
Consider  $v =  v_1 v_2 ... v_n v_{n+1} ... v_{n+m}
\in \NN^{n+m}$.  
We denote by  
${\unkk{v}{n-}} \in \NN^n$ the subword that is the first $n$ symbols of $v$, 
and by 
$\unkk{v}{-m} \in \NN^m $ the last $m$ symbols,
so that 
$\unkk{v}{n-} \unkk{v}{-m} =v$.

\mdef {\bf Proposition}. \label{pr:ddagger}
Consider  $v,w \in \NN^{n+m}$.   
\\ 
(I) If $\ff(v) < \ff(w)$ then  
\ignore{{
$$
\ff(v_1, v_2, ..., v_n, v_{n+1},..., v_{n+m})
\; < \;  
\ff( w_1, w_2, ...,w_n,w_{n+1},...,w_{n+m} )
$$
%
We have 
${\unkk{v}{n-}}, 
{\unkk{w}{n-}} \in \NN^n$,
$\unkk{v}{-m},\unkk{w}{-m} \in \NN^m $,
the indicated 
projections to subwords 
so that $\unkk{v}{n-} \unkk{v}{-m} =v$ (resp. for $w$).
}}
either $\ff(\unkk{v}{n-}) < \ff(\unkk{w}{n-})$ or 
$\ff(\unkk{v}{-m})<\ff(\unkk{w}{-m})$, or both. 
\\
(II) If $\ff(v) = \ff(w)$ then either 
$\ff(\unkk{v}{n-}) = \ff(\unkk{w}{n-})$ {\em and} 
$\; \ff(\unkk{v}{-m}) = \ff(\unkk{w}{-m})$ 
or both are inequalities and in opposite directions.
\\
(III) If $\ff(v) > \ff(w)$ then at least one of 
$\ff(\unkk{v}{n-}) > \ff(\unkk{w}{n-})$  and 
$\; \ff(\unkk{v}{-m}) > \ff(\unkk{w}{-m})$ 
holds. 

\medskip  

\noindent 
{\em Proof}. 
Of course 
$\ff(\unkk{v}{n-}) +\ff(\unkk{v}{-m}) = \ff(v)$ for any $v \in \NN^{n+m}$.
\\
(I)
Suppose for a contradiction that both  
$\ff(\unkk{v}{n-}) \geq \ff(\unkk{w}{n-})$  and 
$\ff(\unkk{v}{-m}) \geq \ff(\unkk{w}{-m})$.
The case both equal gives $\ff(v)=\ff(w)$ --- a contradiction.
So then 
by (\ref{de:Lammoo},\ref{pr:dagger})
there is a lowest $i$ so that $\#_i (\unkk{v}{.-.}) > \#_i(\unkk{w}{.-.})$
on at least one `side', i.e. for $.\!-\!.$ is $n-$ or $-m$.
So then 

\beq  \label{eq:hashsquarej}
\begin{array}{ccccc} 
\#_j (\unkk{v}{n-}) &+ \#_j(\unkk{v}{-m}) =&\#_j(v)    \\
|| & || & || \\
\#_j (\unkk{w}{n-}) &+ \#_j(\unkk{w}{-m}) =&\#_j(w)
\end{array}
\eq 
for $j <i$, and (with at least one $\leq$ strict)
$$
\begin{array}{ccccc} 
\#_i (\unkk{v}{n-}) &+ \#_i(\unkk{v}{-m}) =&\#_i(v)    \\
{\mathsf{V}}| & {\mathsf V}| & {\mathsf V} \\
\#_i (\unkk{w}{n-}) &+ \#_i(\unkk{w}{-m}) =&\#_i(w)
\end{array}
$$
--- a contradiction, by the $\prec$ formulation of $(\Lamm{N}{n},<)$ 
allowed by (\ref{pr:dagger}).
\ignore{{
...
Then these \ppm{ / }
the orbit representatives as in (\ref{eq:orbrep})
 are the same up to $v_{i-1}$ for some $i$ (possibly $i=1$) and then 
$v_i > w_i$.
Now break them both into lengths $n|m$.
\\
If $i\leq n $ then $v_1..v_n < w_1 .. w_n$.
\\
If $i>n$ then $v_{n+1} ..v_i v_{i+1} ..v_{n+m}$ 
is first different from $w_{n+1} ..w_{n+m}$ at $v_i > w_i$,
so $v_{n+1}..v_{n+m} < w_{n+1}..w_{n+m}$.
So, altogether, at least one subword obeys $v_{[]} < w_{[]}$.
\ppm{[-finish this!]}
}}
\\
(II) Suppose $\ff(\unkk{v}{n-}) =\ff(\unkk{w}{n-})$. 
Then the outer two vertical equalities hold in (\ref{eq:hashsquarej}) for all $j$;
and hence the inner equality holds. 
Similarly if   $\ff(\unkk{v}{n-}) > \; \ff(\unkk{w}{n-})$ then there is a lowest $i$ 
where  $\#_i(\unkk{v}{n-} ) < \; \; \#_i(\unk{w}{n-})$, and then  
$\#_i(v) = \#_i(w)$ forces 
$\#_i(\unkk{v}{-m} ) \; > \; \; \#_i(\unkk{w}{-m})$. 
And similarly with the inequalities reversed. 
\\
(III) Similar argument to (I). 
\qed

\medskip 

\mdef {\bf{Theorem}}.  \label{th:CCwg} 
Fix $N \in \N$. 
The subset of CCwg morphisms of $\MAT^N$ gives a strict monoidal subcategory
(hence denoted $\MATCHwg^N$). 

\medskip 

\noindent 
{\em Proof}. 
We require to show closure under the category and monoidal compositions
(observe that the units are clearly present). 
\\ 
For the category composition, 
consider $L,M \in \MAT^N$,  
composable, and in the CCwg subset.
If either is not square then $LM=0$ so there is nothing to check.
So take $L,M \in \MAT^N(n,n)$.
For $w,v \in \N^n$ 
\[
\langle w | LM | v \rangle \; 
= \;\; \sum_{u \in \N^n} \langle w | L | u \rangle \; \langle u | M | v \rangle 
\]
We require to show that $\ff(w) \not\leq \ff(v)$ (i.e. $\ff(w) > \ff(v)$) implies
$  \langle w | LM | v \rangle =0 $.
Suppose $\ff(w) > \ff(v)$. We have   
\[
\langle w | LM | v \rangle \; 
= \; \sum_{u \in \N^n \over \ff(u)<\ff(w)} 
        \overbrace{\langle w | L | u \rangle}^{=0} \langle u | M | v \rangle 
  + \sum_{u \in \N^n \over \ff(u) \geq \ff(w)} 
      \langle w | L | u \rangle \overbrace{\langle u | M | v \rangle}^{=0}
      \;\; = \; 0
\]
as required, where the first indicated factor is 0 by the CCwg condition on $L$;
and the second because $\ff(u) \geq \ff(w) >\ff(v)$ implies $\ff(u) > \ff(v)$ here. 
\\
For the monoidal composition, take $L \in \MAT^N(n,n')$, $M\in\MAT^N(m,m')$.
If $n\neq n'$ or $m\neq m'$ then $L\otimes M=0$ so there is nothing to check, 
so take 
$n=n'$, $m=m'$. 
For $w,v\in \NN^{n+m}$
we require to show $\langle w | L \otimes M |v\rangle =0$ if $\ff(w)>\ff(v)$.
Suppose $\ff(w) > \ff(v)$.
Recall we have 
\beq  \label{eq:kronekerid} 
\langle  {{\scalebox{1.02}{$w$}}}   | L \otimes M |      {v} \rangle  
\; = \; 
  \langle \unkk{w}{n-} | L  |\unkk{v}{n-}\rangle  \;   \langle \unkk{w}{-m} |  M | \unkk{v}{-m}\rangle
\eq 
It is enough to show   
one of 
$\ff( \unkk{w}{n-} ) > \ff(\unkk{v}{n-})$ or  
$\ff( \unkk{w}{-m} ) > \ff(\unkk{v}{-m})$
given $\ff(w) > \ff(v)$.
\ignore{{
For this we need to understand a property of the $(\Lamm{N}{n},>)$ order as 
in (\ref{pr:dagger}) and (\ref{pr:ddagger}).


\mdef (Back to the proof of (\ref{th:CCwg})) 
}}
We indeed have this from (\ref{pr:ddagger}), so that one or other factor in (\ref{eq:kronekerid}) vanishes here, hence the entry vanishes as required.
\qed

\mdef {\bf Proposition}. 
Let $N \in \N$. The category $\MATCH^N$ is a monoidal subcategory of $\MATCHwg^N$. 
\medskip

\noindent 
{\em Proof}. 
The CC condition defining $ \MATCH^N$ can be expressed as $R_{vw} \neq 0$ implies $\ff(v) = \ff(w)$.
\qed

\mdef    \label{pa:2radical}
Note the following from (\ref{pr:ddagger}) and the formulation in the proof 
of Thm.\ref{th:CCwg}
above. 
\medskip 

\noindent 
{\bf Lemma}. 
Let 
$L \in \MATCHwg^N(n,n)$, $M\in\MATCHwg^N(m,m)$, 
and $w,v \in \NN^{n+m}$. 
\medskip 
\\
(A) If $\ff(w) =\ff(v)$ then 
\\
(I) 
if $\ff(\unkk{w}{n-}) = \ff(\unkk{v}{n-})$ then 
$\ff(\unkk{w}{-m}) = \ff(\unkk{v}{-m})$ and both factors in  
$ \langle w | L \otimes M |v\rangle $  
from (\ref{eq:kronekerid})
come from 
`CC entries' in $L$ and $M$. 
$\;$ 
On the other hand 
\\
(II) 
if $\ff(\unkk{w}{n-}) \neq  \ff(\unkk{v}{n-})$
then one of the factors is `glue', and the other is zero.
...

\medskip 

\noindent (B)
And on the other hand if $\ff(w) < \ff(v)$ (i.e. we are looking at a matrix entry 
$ \langle w | L \otimes M |v\rangle $ 
in a glue position) then at least one factor 
comes from a `glue entry' in $L$ or $M$. ...
\qed 

\medskip

\subsection{Towards a glue calculus for representation theory}  
\label{ss:glue-calculus}
$\;$ 

Let us now consider the generic glue ansatz for an $R$-matrix
- extending the generic charge-conserving ansatz solved to classify braid representations $F:\Bcat\rightarrow\MATCH^N$ in \cite{MR1X}. 
(Or indeed for representations of generators in any other such strict monoidal category.)
We want to establish properties of the constraints that depend on whether a parameter is glue (i.e. allowed non-zero in CCgw but not allowed in CC) or non-glue. 

For example 
\beq    \label{eq:Morange}
D=\begin{bmatrix}
    \alpha & \rg \beta \\ & \gamma
\end{bmatrix}  \in \MATCHwg^2(1,1),  \hspace{1cm}
M=
\begin{bmatrix} 
a & \rg b & \rg c & \rg d \\ 
&f&g&\rg h \\ 
&k&l&\rg m \\
&&& r
\end{bmatrix}
\in \MATCHwg^2(2,2)
\eq
has the glue coloured \rg{orange}.

The basics:
\[
\langle 11 | M  = [1,0,0,0]M = [a,b,c,d] ,
\hspace{1cm} 
\langle 22 | M  = [0,0,0,1]M = [0,0,0,r] 
\]
\[
\langle 22 | M | 12 \rangle \; = \; [0,0,0,1] \; M\begin{bmatrix}
    0 \\ 0 \\ 1 \\ 0
\end{bmatrix} = 0 
\]
- consistent with  
$\ff(22) > \ff(12) $ 
as in (\ref{de:CCwg}). 

Meanwhile for example
\[
D \otimes D = \begin{bmatrix}
    \alpha \alpha & \alpha \rg\beta & \rg\beta \alpha & \rg{\beta\beta} \\
                  & \alpha\gamma   &           &  \rg\beta \gamma \\
                  &                & \gamma\alpha &\gamma\rg\beta \\
                  &             &                & \gamma\gamma 
\end{bmatrix}
\]
\[
M\otimes M \; = \; \left[ \begin{array}{cccc|cccc|cccc|cccc}
    aa&a\rg b&a \rg c&a\rg d &\rg{b}a &\rg{bb} &\rg{bc}&\rg{bd} & ... \\
      & af   & ag    &a\rg{h}&        &\rg{b}f &\rg{b}g&\rg{bh} &  ...\\
      & ak   & al    &a\rg{m}&        &\rg{b}k &\rg{b}l&\rg{bm} & ... \\
      &&&ar&     &&&\rg{b}r & ...  \\ \hline 
      &&&&fa& f\rg{b} &f\rg c& f\rg d& ... \\
      &&&&&ff& fg&f\rg h& ... \\
      &&&&&fk&fl&f\rg m& ... \\
      &&&&&&&fr& ... \\ \hline
      &&&&ka&k\rg b&k\rg c&k\rg d& ... \\ 
      &&&&&kf&kg  &k\rg h& ... \\
      &&&&&kk&kl&k\rg m& ... \\
      &&&&&&& kr & ... \\ \hline 
      &&&&&&&& \ddots
\end{array}
\right]
\]

\newcommand{\unkkkk}[2]{
\begin{array}{c} #1 \\ {#2}\end{array}
}

\mdef 
Let $L,M \in \MATCHwg^N(n,n)$ and 
consider matrix entries 
$
\langle w| LM |v\rangle  .
$
When $\ff(w)=\ff(v)$ we have 
\[
\langle w | LM | v \rangle \; 
= \; \sum_{\unkkkk{u \in \N^n }{ \ff(u)<\ff(w)}} 
        \overbrace{\langle w | L | u \rangle}^{=0} \langle u | M | v \rangle 
  + \sum_{\unkkkk{u \in \N^n }{ \ff(u) = \ff(w)}} 
      \langle w | L | u \rangle 
      \langle u | M | v \rangle
\] \[  \hspace{1in} 
  + \sum_{\unkkkk{u \in \N^n }{ \ff(u) > \ff(w)}} 
      \langle w | L | u \rangle 
      \overbrace{
      \langle u | M | v \rangle
      }^{=0}
\]
from which we see that we get contributions only from the strictly CC entries in $L,M$. 
We deduce the following. 


\mdef  {\bf Proposition}.  \label{pr:id-post-glue}
Consider a given set of rank-$N$ CCwg matrices in level-$n$. 
Write $M_1, M_2, ... \in \MATCHwg^N(n,n)$ for these matrices. 
Any identity expressed in terms of the matrices  
$\{ M_i \}$ 
(such as $M_1 M_2 M_1 = M_2 M_1 M_2$) 
and obeyed by
them is also obeyed by the matrices 
$K(M_i)$
obtained by projecting out the glue. 
\qed 

\mdef    \label{le:w<v}
Now consider $\ff(w) < \ff(v)$ - we have
\[
\langle w | LM | v \rangle \; 
= \; \sum_{\unkkkk{u \in \N^n }{ \ff(u)<\ff(w)}} 
        \overbrace{\langle w | L | u \rangle}^{=0} \langle u | M | v \rangle 
  + \sum_{\unkkkk{u \in \N^n }{ \ff(u) = \ff(w)}} 
      \langle w | L | u \rangle 
      \rg{\langle u | M | v \rangle}
\] \[  \hspace{1in} 
  + \sum_{\unkkkk{u \in \N^n }{ \ff(u) > \ff(w)}} 
      \rg{\langle w | L | u \rangle} 
      \langle u | M | v \rangle
\]
- we see that every such entry contains a contribution with a glue factor 
(coloured \rg{orange})
from $L$ or $M$ or both. 

We deduce the following. 

\newcommand{\Alg}{{\mathsf A}}

\mdef  \label{cth:HoG}
{\bf{Theorem}}. [HoG Theorem.]
(I) Glue is a nilpotent ideal 
- meaning that for each $n$ the glue subset of the algebra $\MATCHwg^N(n,n)$ is a nilpotent ideal. 
Thus for $N,n>1$ the algebra  $\MATCHwg^N(n,n)$  is non-semisimple. 
And furthermore any subalgebra $A$ with $\KKK(A)\neq 0$ 
(as in (\ref{de:K'K}))
is non-semisimple.
\\
(II) Also if $\rho: \Alg \rightarrow \MATCHwg^N(n,n)$ is a representation of an algebra $\Alg$ then so is $K\circ \rho$ 
(again as in (\ref{de:K'K})); 
and the image is 
a quotient by {\em part (or all) of} the radical. 
In particular the irreducible factors of $\rho(\Alg)$ and $(K\circ\rho)(\Alg)$ are the same. 
 \medskip   \\ 
\ppmm{{\em Proof}.  (I) We  see from the final term in the expansion in (\ref{le:w<v}) that products of glue elements are glue elements between indices of greater separation in the order (i.e. $\ff(v)>\ff(u)>\ff(w)$). But for any finite $n$ the ordered set is finite, so, iterating, all products with sufficiently many factors vanish. 
\\ 
(II) From (\ref{pr:id-post-glue}), or directly, we have that 
$(K\circ\rho)(ab) = K(\rho(a) \rho(b)) = \; (K\circ\rho)(a) . (K\circ\rho)(b)$. 
\qed }

\mdef Remark. Observe that a representation of an algebra that factors surjectively through a semisimple algebra need not be semisimple (since the radical could be in the kernel). But a rep that factors surjectively through a non-semisimple algebra is necessarily non-semisimple. 

\medskip 

\mdef Next we can turn to the monoidal product. 
In a nutshell, this part of the calculus can be pulled through from (\ref{pa:2radical}). 

\medskip 

There are some questions as to how to present this as an analogue of Jacobson radicals - questions arising from the non-canonical-ness of higher representation theory. 
So rather than exhibiting a grand theory here we will 
limit ourselves to an example application. 

We have shown, for example, that every braid representation of form 
$F: \Bcat \rightarrow \MATCHwg^N$ `restricts' (cf. e.g. the Wedderburn--Malcev Theorem) 
to a charge-conserving representation.
\ignore{{ 
Consider the functors 
$$
I: \MATCH^N \rightarrow \MATCHwg^N
$$ 
(by inclusion) 
and 
$$
J: \MATCHwg^N \rightarrow \MATCH^N
$$ 
(kill glue entries). 

\ppm{[to finish]}
}}